\documentclass[a4paper]{article}
\usepackage{amssymb,amsmath}
\usepackage{graphicx}
\usepackage{fullpage}
\usepackage{listings,relsize}
\usepackage{subfig}
\usepackage{setspace}
\usepackage{varwidth}
\usepackage{pdflscape}
\usepackage{amsthm}
\usepackage{algpseudocode}
\usepackage{multirow}
\usepackage{algorithm}
\usepackage{comment}
\usepackage{hyperref}
\usepackage[titletoc,title]{appendix}
\usepackage[mathlines]{lineno}
\newcommand{\Keywords}[1]{\par\noindent
{\small{\em Keywords\/}: #1}}

\newtheorem{theorem}{Theorem}[section]%  meant for continuous numbers
\newtheorem{proposition}[theorem]{Proposition}% 
\newtheorem{lemma}[theorem]{Lemma}% 
\newtheorem{example}{Example}%
\newtheorem{assumption}{Assumption}[section]

\newtheorem{definition}{Definition}%
\newtheorem{condition}{Condition}

\title{A Reliability Theory of Compromise Decisions for Large-Scale Stochastic Programs}
\author{Shuotao Diao\footnote{Department of Industrial Engineering and Management Sciences, Northwestern University, Evanston IL, 60208, USA}, Suvrajeet Sen\footnotemark[\value{footnote}] \footnote{Epstein Department of Industrial and Systems Engineering, University of Southern California, Los Angeles CA, 90089, USA}\\
\texttt{shuotao.diao@northwestern.edu, s.sen@usc.edu}}
\date{First Version: \today  }
\begin{document}
\doublespacing
\maketitle
\begin{abstract}
% Abstract
Stochastic programming models can lead to very large-scale optimization problems for which it may be impossible to enumerate all possible scenarios. In such cases, one adopts a sampling-based solution methodology in which case the reliability of the resulting decisions may be suspect. For such instances, it is advisable to adopt methodologies that promote variance reduction. One such approach goes under a framework known as ``compromise decision", which requires multiple replications of the solution procedure. This paper studies the reliability of stochastic programming solutions resulting from the ``compromise decision" process. This process is characterized by minimizing an aggregation of objective function approximations across replications, presumably conducted in parallel. We refer to the post-parallel-processing problem as the problem of ``compromise decision". We quantify the reliability of compromise decisions by estimating the expectation and variance of the ``pessimistic distance" of sampled instances from the set of true optimal decisions. Such pessimistic distance is defined as an estimate of the largest possible distance of the solution of the sampled instance from the ``true" optimal solution set. The Rademacher average of instances is used to bound the sample complexity of the compromise decision.
 % Keywords
\end{abstract}
\section{Motivation}\label{sec:intro}

Artificial Intelligence, Machine Learning, and Data Science (AMD) have become major sources of applications of stochastic optimization. In general, the class of AMD models are often based on the idea of Empirical Risk Minimization (ERM) in which vast amounts of data can be characterized via a class of functions (e.g., affine, convex etc.) which can be used for predictive purposes.  Concomitantly, we are also witnessing the growth of decision and design (DD) optimization for new applications in sustainable energy, environment, and even health-care decisions using Sample Average Approximation (SAA) \cite{kleywegt2002sample}.  These two classes of ``applications'' (AMD and DD) tend to have slightly different algorithmic requirements:  AMD problems tend to seek {\it functions} which perform well across a range of plausible data using ERM, while DD seeks a finite dimensional {\it vector} which is predicted to perform well across a sampled collection of plausible scenarios using SAA.  Statistically speaking, ERM and SAA are cousins of each other, drawing from the same ``well'' of sample complexity theory developed within the statistical learning (SL) literature.  However, there are real-world implications associated with models of predictive and prescriptive analytics: in case of the former, small errors in predictions are often ``easily forgiven,'' (e.g., errors in predicting precipitation) whereas, errors in DD can often represent loss of coordination among multi-variate choices (e.g., different product lines), and can lead to operational havoc. Because coordination across multiple facets may contribute to much of the value-proposition of DD, it is important to have plans which dovetail well, even in the presence of significant uncertainty. 
For these reasons, choices which are recommended in the DD setting often call for greater reliability, especially because uncertainty has the potential to up-end the entire plan.  To accommodate such coordination, our prior work has put forward the concept of ``compromise decisions'' for various versions of stochastic programming (SP) structures (e.g., Stochastic LPs \cite{sen2016mitigating}, Stochastic Multi-Stage LPs \cite{xu2023compromise}, and Stochastic MIPs \cite{xu2023ensemble})).  While the specifics of compromise decisions for each class is designed to address the specific structure of the model, they are all intended to enhance the reliability of decisions.  However unlike the above papers, which are computationally motivated, this paper is envisioned as one providing {\it mathematical/statistical} foundations for the concept of Compromise Decisions.

%Prior computational results with D/D-SP have illustrated that certain alternative computational approaches for decision models can provide a more robust class of decisions for many classes of optimization models (e.g., Stochastic LPs \cite{sen2016mitigating}, Stochastic MSLPs \cite{xu2023compromise}, and Stochastic MIPs \cite{xu2023ensemble}).  The computational results of the above papers demonstrate the need for, and the power of {\it low variance decisions}, while emphasizing optimality.   This new emphasis will enhance the reliability of decision-making resulting from SP models.  Here the use of the term ``reliability'' is intended to express a desire for SP decisions to be relatively stable, even though they may result from multiple runs using sampling. {\it In essence, this goal is intended to help SP models support decision-making in which both decisions as well as their predicted objective function estimates are deemed acceptable in terms of their variability.}  While such an approach presumes the availability of multiple processors (i.e., parallel processing for replications), such a requirement is well justified in an age when computing speed and memory are both ubiquitous, and inexpensive.    
                                
There are several challenges which present barriers to overcoming the goals outlined above: a)  Traditional SAA (as well as ERM) focuses on sample complexity based on probability estimates of approximation quality, b) it is not obvious how multiple sampling-based runs should be combined to provide one decision for the purposes of implementation, c) it is also burdensome to characterize the worst-case scenario (in distribution) of the estimated solution generated by a given SP algorithm. One recommendation (due to Nesterov) is to simply average outputs resulting from parallel runs (e.g., using multiple Stochastic Gradient Descent (SGD) runs).  The challenge here is that it is unclear whether the variability associated with decisions are to be considered in some way. In this paper, we plan to provide sample complexity of DD associated with convex SP models.  We expect that future papers will present more comprehensive theory to accommodate integer as well as other non-convex settings. 

\indent Importance sampling is one popular approach to reduce the variance associated with objective function estimates (of values as well as gradient/subgradient estimates) when using Monte Carlo sampling \cite{deo2023achieving}. It has been widely used to reduce variance in gradient estimates for portfolio  \cite{zhao2015stochastic, kawai2018optimizing} and credit risk optimization \cite{glasserman2005importance}. Kozmík and Morton \cite{kozmik2015evaluating} also propose an importance sampling methodology to reduce the variance of the upper bound estimate of optimal cost from the Stochastic Dual Dynamic Programming (SDDP) algorithms in a risk-averse setting. Carson and Maria \cite{carson1997simulation} summarized that the key to the success of importance sampling lies in the appropriate change of the probability measure used for rare event simulation.

Additionally, other popular variance reduction methods include linear control random variables method (\cite{shapiro2003monte}), in which a correlated random variable with mean zero is added to the objective function, as well as using common random numbers for variance reduction approach (\cite{Higle1998VarianceRA}, \cite{kleinman1999simulation}).  For a review of modern variance reduction techniques in stochastic optimization, we refer the readers to a survey by Homem-de-Mello and Bayraksan \cite{homem2014stochastic}. \\
\indent Other than the above-mentioned variance reduction techniques, Pasupathy and Song \cite{pasupathy2021adaptive} propose an adaptive sequential SAA framework where the inner loop stopping criterion depends on the optimality gap and pointwise variance estimates and the outer loop utilizes previously estimated solution as a warm start.   \\
\indent When memory and computational time are no longer the bottleneck in computation (e.g. \cite{zhang2015memory}), performing several replications of sampling procedures, perhaps in parallel, becomes an alternative to improve the objective function estimate and/or optimal solution estimate. Sen and Liu \cite{sen2016mitigating} propose a closed-loop methodology based on a composite function which aggregates the sum of piecewise linear objective function approximations (from multiple {\it independent runs}), together with a regularizing term which measures deviation from a sample average of decisions.  The above approximation is what \cite{sen2016mitigating} refers to as the compromise problem, and the optimum solution of that problem is referred to as the compromise decision.  The computational results reported in the aforementioned paper illustrates the power of aggregated approximations for stochastic linear programming problems, and their computational results demonstrate the variance reduction due to the compromise decision approach.

At approximately the same time as the publication of \cite{sen2016mitigating}, the International Conference on SP (ICSP-13 in 2016) was held in Buzios, Brazil where several SP luminaries made the case for recommending general-purpose methodologies which could support {\it low variance decisions} for SP problems.  During that session it was noted that while SP methodology was strongly supported by SAA-based theory, most SP algorithns were not designed to provide computational support for establishing the {\it reliability of decisions} proposed by SP software. %This paper is motivated by the  ``disconnect'' between SP theory and computations.

%It is important to recognize that the method is not only ideally suited for parallel computing, but it provides far superior computational results than have been obtained by methods other methods, especially for instances such as the notorious SSN model \cite{SDCosares1994} which has been used in several computational experiments (e.g., \cite {NJLShapiro2009}) to validate their approach known as Robust Stochastic Approximation.  However, many approaches prior to 2016 remained focused on providing good estimates of providing the optimal objective value, although identifying specific optimal decisions were l themselves were not intended as a critical component of the decision-making process because their modus-operandi could not specify decisions which would ensure  the optimal value could be realied upon implementing the specific decision. \\
%%%%%%%%%%%% What was deleted?
\indent Motivated by the  ``disconnect'' between SP theory and computations, we aim to present a theory which is based on a small set of principles that can be used to explain the computational success reported in the above-mentioned papers. Indeed, we focus on studying the finite-sample complexity of the Compromise Decision approach to solve the following generic stochastic program%\footnote {%Formal definitions will be introduced in later sections.} }
\begin{equation} \label{eq:generic sp}
    \min_{x \in X} \ f(x) \triangleq \mathbb{E}_{\tilde{\xi}}[F(x,\tilde{\xi})] \text{\footnotemark}, 
\end{equation}
\footnotetext{$\mathbb{E}_{\tilde{\xi}}[F(x,\tilde{\xi})] = \int_{\Omega} F(x,\tilde{\xi}(\omega) ) d \mathbb{P} = \int_{\Xi} F(x,\xi) \mu(d \xi)$.}
where $\tilde{\xi}: \Omega \mapsto \Xi \subset \mathbb{R}^d$ is a random variable with distribution $\mu$ defined on a probability space $(\Omega,\Sigma_{\Omega}, \mathbb{P})$, $X \subset \mathbb{R}^p$ is the feasible region of $x$, and $F: X \times \Xi \mapsto \mathbb{R}$ is a Carath\'eodorian function (i.e., continuous in $X$ and measurable for almost every $\xi \in \Xi$). 
 \\
\indent Given $m$ ($m \geq 2$) replications, we let $\xi_i^n$, $\hat{f}_n(x;\xi_i^n)$ and $\hat{x}(\xi_i^n)$ denote the sampled data sets (each of size $n$), objective function approximations (based on sampling-based estimates), and decision estimate in the $i^\text{th}$ replication of Monte Carlo sampling to recommend an approximate solution to the problem in (\ref{eq:generic sp}), respectively. The compromise decision problem for convex SP, based on $m$ replications is formulated as follows:
\begin{equation} \label{eq:generic compromise sp}
    \min_{x \in X} \ \frac{1}{m} \sum_{i=1}^m \hat{f}_n(x;\xi_i^n) + \frac{\rho}{2} \left\|x - \frac{1}{m} \sum_{i=1}^m \hat{x}_n(\xi_i^n) \right\|^2.
\end{equation}
The addition of the quadratic regularizer to the objective function has been widely used in SP algorithms such as the proximal point method \cite{rockafellar1976monotone}, proximal bundle method \cite{kiwiel1990proximity}, mirror descent method \cite{nemirovski2009robust}, regularized SD \cite{higle1994finite}, regularized SDDP \cite{asamov2018regularized, guigues2020regularized}, and more recently, SDLP \cite{doi:10.1137/19M1290735}. Furthermore, the use of the decision-based regularizer in (\ref{eq:generic compromise sp}) can be interpreted as one way of countering overfitting. SP models in our setting often involve multi-dimensional random variables and the sampling distribution $\hat{\mu}_N$ often involves $N$ which is so large, as to be untenable for constrained optimization. In this sense, the size of $n$ which we use in SP is relatively small as to be similar to the case of overfitting in statistics . % elaborate on countering overfitting 

\begin{algorithm}[!htbp]
\caption{Compromise Decision Approach} \label{alg:compromise decision approach}
\begin{algorithmic}
\State {\sl Initialization}: Set sample size $n \in \mathbb{Z}_+$, replication number $m \in \mathbb{Z}_+$, $\rho \in (0,+\infty)$, and an \textsl{Algorithm} $\Upsilon$ for the \textsl{replication} step. 
\State {\sl Replication} Step: For $i = 1,2,\ldots,m$, use \textsl{Algorithm} $\Upsilon$ compute estimated decision, $\hat{x}(\xi^n_i)$, and function estimation, $\hat{f}_n(x;\xi^n_i)$, of the $i^{\text{th}}$ replication. 
\State {\sl Aggregation} Step: Aggregate function estimates $\frac{1}{m} \sum_{i=1}^m \hat{f}_n(x;\xi_i^n)$ with certain augmentation. Aggregate estimated decisions as $\frac{1}{m} \sum_{i=1}^m \hat{x}_n(\xi_i^n)$. Obtain the $\epsilon$-optimal solution of the compromise decision problem in (\ref{eq:generic compromise sp}). 
%\begin{equation*}
%    \min_{x \in X} \ \frac{1}{m} \sum_{i=1}^m \hat{f}_n(x;\xi_i^n) + \frac{\rho}{2} \left\|x - \frac{1}{m} \sum_{i=1}^m \hat{x}_n(\xi_i^n) \right\|^2.
%\end{equation*}
\end{algorithmic}
\end{algorithm}

The compromise decision problem consists of an estimated aggregation of sampled value function approximations, and a penalty based on a distance to the average decision estimate. It can also be interpreted as performing one step of the proximal point method for solving $\frac{1}{m} \sum_{i=1}^m \hat{f}_n(x;\xi_i^n)$ given that the initial estimate is $\frac{1}{m} \sum_{i=1}^m \hat{x}_n(\xi_i^n)$. Let $\bar{x}_N(\xi^N) = \frac{1}{m} \sum_{i=1}^m \hat{x}_(\xi^n_i)$ and let $x^c_N(\xi^N)$ denote the optimal solution of (\ref{eq:generic compromise sp}). It is obvious that if $\bar{x}_N(\xi^N)$ and $x^c_N(\xi^N)$ agree, then both are optimal to $\min\limits_{x \in X} \ \frac{1}{m}\sum_{i=1}^m \hat{f}_n(x;\xi_i^n)$. This observation has been transformed into a stopping rule for compromise SD (\cite{sen2016mitigating}). In short, the Compromise Decision approach consists of two steps: a  {\sl replication} step and an {\sl aggregation} step. We summarize the Compromise Decision approach in Algorithm \ref{alg:compromise decision approach}. \\
\indent Motivated by the computational evidence of successful use of the Compromise Decision problem \cite{sen2016mitigating, xu2023compromise, xu2023ensemble}, this paper aims to expound on a mathematical foundation for those successes. In particular, we shall address the following questions which will support a mathematical (as opposed to computational) basis for Compromise Decisions: 
% need to rewrite 
\begin{itemize}
    %\item[1.] What is the sample complexity of the margin of error on the objective function involving replications? 
    \item[1.] How should one quantify the reliability of the estimated decisions of stochastic programs? 
    %\item[2.] How the penalty coefficient $\rho$ interacts with the function estimate and decision estimate? 
    \item[2.] What is the reliability of a compromise decision when a particular algorithm is used in each replication? 
    \item[3.] More specifically, what is the reliability of compromise decisions with SD algorithm involved in each replication of a stochastic QP?
    %\item[3.] What is the sample complexity of $\epsilon$-optimal solution of (\ref{eq:generic compromise sp})? In particular, what is the large deviation bound of the distance between the compromise decision and optimal solution set of (\ref{eq:generic sp})? Furthermore, what is the upper bound of the variance of the distance between the compromise decision and optimal solution set of (\ref{eq:generic sp})?
    %\item[4.]  What is the large deviation bound of the compromise decision and optimal solution set of (\ref{eq:generic sp}) when the Stochastic Decomposition (SD) algorithm is used to attain $\{\hat{f}^i_n(x)\}_{i=1}^m$ and $\{\hat{x}^i_n\}_{i=1}^m$? Moreover, what is the upper bound of the variance of distance between the compromise decision and optimal solution set of (\ref{eq:generic sp}) when SD is used for each replication?  Note that using variance as one of the metrics to guide stopping rules, requires that we study sample complexity of compound metrics (as in \cite{ermoliev2013sample}), and for our purposes, it will be best to estimate upper  bounds on variance, and stopping when such a bound is acceptable.   
\end{itemize}

The above agenda is not only novel, but it provides a theoretical justification for reliable decision-making using SP. It is important to recognize that in the absence of such a theory, sampling-based computational algorithms for SP (including traditional SDDP \cite{pereira1991multi}) are likely to lead to {\it both sub-optimality} and/or {\it lower reliability} (i.e.,  {\it greater variability}).  Computational evidence of such risks, especially for multi-stage SP, are presented in \cite{xu2023compromise}.

This paper is organized as follows. In section \ref{sec:sample complexity background}, we review Rademacher complexity and its use in bounding the sample complexity of point estimates of the objective function and its variance. In section \ref{sec:classic compromise SP}, we present both \textsl{exact} and \textsl{inexact} compromise decision problem formulation and its quantitative reliability. In section \ref{sec:algorithms compromise sp}, we discuss the sample complexity analysis of the compromise decision when a cutting-plane-type algorithm is used to solve each replication of the approximation problem. Finally, we provide the theoretical evidence of the success of using SD algorithms for the compromise decision problem in section \ref{sec:SD compromise sp}.

\subsection{Notations}
We let $\tilde{\xi} : \Omega \mapsto \Xi \subset \mathbb{R}^d$ denote a random vector defined on a probability space $(\Omega,\Sigma_{\Omega}, \mathbb{P})$ and let $\xi$ denote one realization of $\tilde{\xi}$. Let $\tilde{\xi}_1, \tilde{\xi}_2, \ldots, \tilde{\xi}_n$ denote independent and identically distributed (i.i.d.) copies of $\tilde{\xi}$. For $i = 1,2,\ldots,n$, we let $\xi_i$ denote the realization of $\tilde{\xi}_i$ and let $\xi^n \triangleq \{\xi_1,\xi_2,\ldots,\xi_n \}$ denote the set of realizations of $n$ i.i.d. copies of $\tilde{\xi}$. Let $X \subset \mathbb{R}^p$ be a nonempty compact set of decisions, and let $F: X \times \Xi \mapsto \mathbb{R}$ be a Carath\'eodorain function. We let $\Pr\{\cdot\}$ denote  
the probability of an event, and let $\|\cdot\|$ denote the Euclidean norm.\\
\indent Without further specification, we let $n$ denote the sample size of each replication and let $m$ denote the number of replications. We let $\xi^n_i$ denote the sample set with size $n$ in the $i^{\text{th}}$ replication. In particular, we let $\xi^n_i = \{\xi_{(i-1)n + j} \}_{j=1}^n$. We let $N = mn$ denote the total sample size and let $\xi^N = \cup_{i=1}^m \xi^n_i$ denote the mega-sample set. Without loss of generality, we consider the case in which the sample size for each replication is the same, although it is straightforward to extend the analysis to the case of heterogeneous sample sizes among the replications. \\
\indent As for the notations in the Rademacher complexity, we let $\tilde{\sigma}_1, \tilde{\sigma}_2, \ldots, \tilde{\sigma}_n$ be i.i.d. random variables with $\tilde{\sigma}_i$ for $i = 1,2,\ldots, n$ being equally likely to be $1$ or $-1$. That is, $\Pr(\tilde{\sigma} = 1) = \Pr(\tilde{\sigma} = -1) = \frac{1}{2}$. Furthermore, we require that $\tilde{\sigma}_i$ are independent of $\tilde{\xi}$ for all $i = 1,2,\ldots, n$. \\
\indent To study $\epsilon$-optimality of a solution, we shall define the following metric to measure the \textsl{pessimistic} distance between two sets (see \cite{ermoliev1991normalized,ermoliev2013sample, pasupathy2021adaptive} for similar uses).
\begin{definition}[Pessimistic Distance] \label{def:pessimistic distance}
    Let $A$ and $B$ be two sets from an Euclidean space. The pessimistic distance of $A$ to $B$ is defined as follows: 
    \begin{equation} \label{eq:Delta Set}
    \Delta(A,B) \triangleq \sup_{a \in A} \inf_{b \in B} ~ \|a - b\|.
\end{equation}
\end{definition}
An interpretation of $\Delta(A,B)$ is that it is the largest distance from a point of set $A$ to set $B$. For instance, if $A$ is the set of possible estimated solutions and $B$ is the set of optimal solutions. Then $\Delta(A,B)$ is the farthest distance from the estimated solution to the optimal solution set.\\
\indent The following theorem describes the Lipschitzian behavior of the $\epsilon$-optimal solution set in terms of the metric defined in \eqref{eq:Delta Set}. \\
 \begin{theorem} \label{theorem:lipschitz continuity of solution set}
\cite{ermoliev1991normalized} Assume that $X$ is a nonempty compact convex set and $f: X \mapsto \mathbb{R}$ is a lower semicontinuous convex function. Let the following definitions hold: $D_X = \max_{x, x' \in X} \|x - x' \|$; $\theta^* = \min_{x \in X} f(x)$; $\epsilon' > \epsilon > 0$, $X_{\epsilon} = \{ x \in X: f(x) \leq \theta^* + \epsilon\}$, and $X_{\epsilon'} = \{ x \in X: f(x) \leq \theta^* + \epsilon' \}$. Then the following holds:
\begin{equation*}
    \Delta (X_{\epsilon'}, X_{\epsilon}) \leq \frac{\epsilon' - \epsilon}{\epsilon} D_X.
\end{equation*}
\end{theorem}
\proof
See \cite[Theorem 3.11]{ermoliev1991normalized}.  
\endproof
The results of Theorem \ref{theorem:lipschitz continuity of solution set} relate the perturbation of the objective function (e.g., estimation error from sampling) to the $\epsilon$-solution set. It will be one of the major tools for studying the reliability of the compromise decisions. When $\epsilon' = \epsilon > 0$, $\Delta (X_{\epsilon'}, X_{\epsilon}) = 0$ and $\frac{\epsilon' - \epsilon}{\epsilon} D_X =0$. Hence, the inequality in Theorem \ref{theorem:lipschitz continuity of solution set} still holds when $\epsilon' = \epsilon > 0$. \\
%\indent Let $X^c_{N,\epsilon}$ denote the $\epsilon$-optimal solution set of problem (\ref{eq:generic compromise sp}) and let $X^*_\epsilon$ denote the $\epsilon$-optimal solution set of the true problem (\ref{eq:generic sp}). With the notations introduced above, we will quantify the reliability of the compromise decisions by answering the following two questions:
%\begin{enumerate}
%    \item What is the sample complexity of $\mathbb{E}[\Delta(X^c_{N,\epsilon}, X^*_\epsilon)]$? 
%    \item What is the sample complexity of $\text{Var}[\Delta(X^c_{N,\epsilon}, X^*_\epsilon)]$? 
%\end{enumerate}
% need a table of notations 
\indent Finally, we summarize key notations and their use in computation and/or analysis in Appendix \ref{appendix:table of notations}. 

\subsection{Contributions} %intro:contributions
The contributions of this paper are two-fold: (1) We propose a unifying framework to compute compromise decisions by aggregating information from multiple replications of SP runs; (2) We provide a quantitative way to measure the reliability of the compromise decision. \\ 
\indent Here, we summarize our contribution to the reliability of the compromise decisions when a class of SP methods are involved in solving each replication in Table \ref{tab:summary:reliability of compromise decision}. We say that if one compromise decision has better reliability than another compromise decision, then the former one has a faster convergence rate in both expectation and variance of the pessimistic distance (in Definition \ref{def:pessimistic distance}) to the ($\epsilon$-)optimal solution set. Thus we use a 2-tuple to describe the reliability of the compromise decision. In particular, the first element (of the 2-tuple) will denote the expectation of the pessimistic distance, whereas, the second element of the 2-tuple measures the variance of the pessimistic distance to the ($\epsilon$-)optimal solution set.  Due to the sublinear convergence rate of SD for SQQP problems, we observe that the compromise decision using SD (in section \ref{sec:SD compromise sp}) has a higher reliability than the one using SAA (in section \ref{sec:classic compromise SP}) or cutting-plane-type methods (in section \ref{sec:algorithms compromise sp}). It is also expected that deterministic cutting-plane methods (such as Benders Decomposition or L-shaped method) and SAA have the same reliability because a deterministic method can only deliver what is permissible within the SAA setup.  
\begin{table}[!htbp]
    \centering
    \begin{tabular}{|c|c|c|}
        \hline
         Method in the {\sl replication} step&  Problem type & Reliability\\
         \hline
         SAA & Convex & $\left(O(\frac{1}{n^\lambda}), O(\frac{1}{mn^\lambda} + \frac{1}{n^{2\lambda}})) \right), \ \lambda \in (0, \frac{1}{2})$\\
         \hline
         Cutting-plane-type methods & Convex & Same as above\\%$\left(O(\frac{1}{n^\lambda}), O(\frac{1}{mn^\lambda} + \frac{1}{n^{2\lambda}})) \right), \ \lambda \in (0, \frac{1}{2})$\\
         \hline
         SD & Two-stage SQQP & $\left(O(\frac{1}{n}), O(\frac{1}{mn} + \frac{1}{n^2})\right)$\\
         \hline
    \end{tabular}
    \caption{Reliability of Compromise Decisions}
    \label{tab:summary:reliability of compromise decision}
\end{table}

\section{Background on Sample Complexity: Point Estimates and Rademacher Average} \label{sec:sample complexity background}

In order to set the stage for the goals outlined earlier, we begin by reviewing the notion of Rademacher average and its use in bounding a sample average approximation of the objective function and a sample variance of the random cost function.  

As mentioned by \cite{bartlett2002rademacher}, ``Rademacher complexity is commonly used to describe the data-dependent complexity of a function class". One key advantage of (empirical) Rademacher average (or complexity) is that it can be measured from a finite sample set (see \cite{bartlett2002rademacher} for i.i.d. cases, and \cite{mohri2008rademacher} for non-i.i.d. cases). As a result, it can be used to estimate the finite-sample error of a function class. Rademacher average has been widely used in neural networks (\cite{bartlett2017spectrally}), support vector machine (\cite{sun2011multi}), and decision trees (\cite{kaariainen2003rademacher}). 

%We start by restating the classic stochastic programming problem as follows: 
%\begin{equation} \label{eq:classic SP}
%    \min_{x \in X} f(x) \triangleq \mathbb{E}_{\tilde{\xi}}[F(x,\tilde{\xi})].
%\end{equation}
The common finite-sample approximation of \eqref {eq:generic sp} is known as
the sample average approximation (SAA), and is written as follows.
\begin{equation} \label{eq:classic SP:SAA}
    \min_{x \in X} f_n(x;\xi^n) \triangleq \frac{1}{n} \sum_{i=1}^n F(x,\xi_i).
\end{equation}
Since we will study multiple replications of sampling, we write the sample set explicitly in the argument of the $f_n(x;\cdot)$ to distinguish different sample sets. A similar notational style will apply to sample variance, estimated solutions and compromise decisions. 

The variance of the random cost function, $F(x,\tilde{\xi})$, parameterized by the decision $x$, is defined as 
\begin{equation*}
    \text{Var}[F(x,\tilde{\xi})] \triangleq \mathbb{E}_{\tilde{\xi}} \left[\left(F(x,\tilde{\xi}) - f(x)\right)^2 \right].
\end{equation*}
The unbiased estimate of the variance of $F(x,\tilde{\xi})$ is formulated as 
\begin{equation} \label{eq:classic SP:variance estimate}
    s^2_n(x;\xi^n) \triangleq \frac{1}{n-1} \sum_{i=1}^n \left[F(x,\xi_i) - \hat{f}_n(x) \right]^2. 
\end{equation}
We let $Y$ denote a convex compact set in $\mathbb{R}$, and define a compound function $H: X \times Y \times \Xi \mapsto \mathbb{R}$ with $H(x,y,\xi) = (F(x,\xi) - y)^2$. As suggested in \cite{ermoliev2013sample}, the variance of $F(x,\tilde{\xi})$ can be written as a compound function as follows:  
\begin{equation*}
  \text{Var}[F(x,\tilde{\xi})] = \mathbb{E}_{\tilde{\xi}}\left[H(x,\mathbb{E}_{\tilde{\xi}}[F(x,\tilde{\xi})],\tilde{\xi})\right] = \mathbb{E}_{\tilde{\xi}}\left[(F(x,\tilde{\xi}) - \mathbb{E}_{\tilde{\xi}}[F(x,\tilde{\xi})])^2 \right].
\end{equation*}
Furthermore, the sample variance, $s_n^2(x)$, can be rewritten as the sum of compound functions below.
\begin{equation} \label{eq:sample variance:compound}
    s_n^2(x;\xi^n) = \frac{1}{n-1} \sum_{i=1}^n H(x,\frac{1}{n} F(x,\xi_i), \xi_i).
\end{equation}
Throughout the paper, we make the following assumptions: 
% assumptions part 1
\begin{assumption} \label{assumption:part 1}
 Let $X \subset \mathbb{R}^p$ be a nonempty compact set of decisions, and let $F: X \times \Xi \mapsto \mathbb{R}$ be a Carath\'eodorain function. 
\begin{enumerate}
    \item[1.] $X \subset \mathbb{R}^p$ is a nonempty compact convex set contained in a cube whose edge-length is $D$.
    \item[2.] $F(x,\xi)$ is H\"older continuous in $x$ with constant $L_F$ and $\gamma \in (0,1]$ for all $\xi \in \Xi$. That is, $|F(x,\xi) - F(y,\xi)| \leq L_F \|x - y\|^\gamma$ for all $x, y \in X$ and $\xi \in \Xi$.  
    \item[3.] There exists $M_F \in (0,\infty)$ such that $\sup_{x \in X, \xi \in \Xi} |F(x,\xi)| < M_F$. Let $Y \subset \mathbb{R}$ be $Y \triangleq [-M_F, M_F]$.  
\end{enumerate}
\end{assumption}
We note that the assumption of boundedness of the feasible region and the objective function is common in the SP literature \cite{shapiro2005complexity, nemirovski2009robust, ermoliev2013sample}. Also, the H\"older continuity condition of the objective function is a generalization of its Lipschitzian counterpart. The introduction of the soundness parameters and H\"older continuity-related parameters are later used to bound the Rademacher average of the random cost function and its variance.\\
\indent Next we introduce the definition of the Rademacher average of a finite set of vectors. 
\begin{definition}[\cite{boucheron2005theory}] \label{def:rademacher average of finite set}
    Let $\tilde{\sigma}_1, \tilde{\sigma}_2, \ldots, \tilde{\sigma}_n$ be i.i.d. random variables with $\tilde{\sigma}_i$ for $i = 1,2,\ldots, n$ being equally likely to be $1$ or $-1$ (i.e., $\Pr(\tilde{\sigma}_i = 1) = \Pr(\tilde{\sigma}_i = -1) = \frac{1}{2}$ for $i \in \{1,\ldots,n\}$). Let $A \subset \mathbb{R}^n$ be a bounded set of vectors $a = (a_1, \ldots,a_n)$, the Rademacher average associated with $A$ is defined below:
    \begin{equation}
        R_n(A) = \mathbb{E}_{\tilde{\sigma}} \left[\sup_{a \in A} \frac{1}{n} \left|\sum_{i=1}^n \tilde{\sigma}_i a_i \right| \right].
    \end{equation}
\end{definition}
% preliminaries
The following lemma yields an upper bound of the Rademacher average of a finite set. Note that it is a modification of \cite[Theorem 3]{boucheron2005theory} which requires a slight re-orientation. 
\begin{lemma} \label{lemma:rademacher:discrete}
If $A = \{a^{(1)},\ldots,a^{(N)} \} \subset \mathbb{R}^n$ is a finite set, then 
\begin{equation}
    R_n(A) \leq \max_{j = 1,\ldots,N} \|a^{(j)} \| \frac{\sqrt{2 \log{{2}N}}}{n}
\end{equation}
\end{lemma}
\proof
See the Appendix \ref{appendix:proof of lemma:rademacher:discrete}.  
\endproof

Lemma \ref{lemma:rademacher:discrete} is fundamental to the study of the upper bound of the Rademacher average of a function class. 
%Recall that we let $\tilde{\xi}_1, \tilde{\xi}_2, \ldots, \tilde{\xi}_n$ denote $n$ i.i.d. copies of $\tilde{\xi}$ and let $(\xi_1,\ldots,\xi_n) = \xi^n$ denote a realization of $(\tilde{\xi}_1,\ldots,\tilde{\xi}_n) \triangleq \tilde{\xi}^n$. 
\noindent The Rademacher average of a function class is defined as follows. 
\begin{definition}[\cite{ermoliev2013sample}]  \label{def:rademacher average of a function class}
    Let $\tilde{\sigma}_1, \tilde{\sigma}_2, \ldots, \tilde{\sigma}_n$ be i.i.d. random variables with $\tilde{\sigma}_i$ for $i = 1,2,\ldots, n$ being equally likely to be $1$ or $-1$. For a set of points $(\xi_1,\ldots,\xi_n) = \xi^n$ in $\Xi$ and a sequence of functions $\{F(\cdot,\xi_i): X \mapsto \mathbb{R} \}$, the Rademacher average of a function class is defined as:
    \begin{equation}
        R_n(F,\xi^n) \triangleq \mathbb{E}_{\tilde{\sigma}} \left[ \sup_{x \in X} \left| \frac{1}{n} \sum_{i=1}^n \tilde{\sigma}_i F(x,\xi_i) \right|\right].
    \end{equation}
    Furthermore, the Rademacher average of a family of functions $\{F(\cdot, \xi): X \mapsto \mathbb{R}\}_{\xi \in \Xi}$ is defined as:
    \begin{equation}
        R_n(F,\Xi) \triangleq \sup_{\xi_1 \in \Xi, \ldots, \xi_n \in \Xi} R_n(F,\xi^n).
    \end{equation}
\end{definition}

\indent The following lemma provides upper bounds of Rademacher average of a sequence of functions studied by \cite{ermoliev2013sample}, although we modify the affected constant resulting from Lemma \ref{lemma:rademacher:discrete}. Note that the lemma below assumes that the feasible region is discrete and it is free of the Assumption \ref{assumption:part 1}. Moreover, it should be regarded as a building block for further results.  
\begin{lemma}[\cite{ermoliev2013sample}] \label{lemma:rademacher:discrete set}
    Let the set $X \subset \mathbb{R}^p$ be discrete and contain a finite number of elements $|X| \leq m^p$, and assume $\sup_{x \in X} |F(x,\xi_i)| \leq M(\xi_i) \leq \max_{1 \leq i \leq n} M(\xi_i) = M_n(\xi^n)$, then the following holds:
    \begin{equation}
        R_n(F,\xi^n) \leq M_n(\xi^n) \sqrt{2 \log 2|X|/n} \leq M_n(\xi^n) \sqrt{2p (\log {2}m) / n}
    \end{equation}
\end{lemma}

With the help of Lemma \ref{lemma:rademacher:discrete set}, we can derive the upper bound of Rademacher average of the random cost function, $F(x,\xi)$, in the following lemma.
\begin{lemma}[\cite{ermoliev2013sample}] \label{lemma:rademacher complexity of objective:holder continuity}
    Suppose that Assumption \ref{assumption:part 1} holds. Then 
    \begin{equation} \label{eq:rademacher:holder continuity01}
        R_n(F,\xi^n) \leq \left(L_F D^{\gamma} p^{\frac{\gamma}{2}} + M_F \sqrt{2 ({\log2} + \frac{p}{2\gamma}\log n)} \right)/ \sqrt{n},
    \end{equation}
    where $L_F$ and $\gamma$ are the constants in the H\"older continuity of $F(\cdot,\xi)$, $n$ is the sample size, $p$ is the dimension of the decision, $x$, and $D$ is the edge length of the cube that contains the feasible region, $X$. Furthermore, for $\lambda \in (0,\frac{1}{2})$, we have 
    \begin{equation}
          R_n(F,\xi^n) \leq \frac{N_F}{n^\lambda}
    \end{equation}
    where $N_F = L_F D^{\gamma} p^{\frac{\gamma}{2}} + {M_F \sqrt{2 (\log2)}} + \frac{M_F p^{1/2}}{\sqrt{\gamma (1 - 2 \lambda)e}}$.
\end{lemma}
\proof
   See the Appendix \ref{appendix:proof of lemma:rademacher complexity of objective:holder continuity} 
\endproof
Since the constant $N_F$ in Lemma \ref{lemma:rademacher complexity of objective:holder continuity} stays the same for any $\xi \in \Xi$, we conclude that 
\begin{equation*}
    R_n(F,\Xi) = \sup_{\xi_1 \in \Xi, \ldots, \xi_n \in \Xi} R_n(F,\xi^n) \leq %\sup_{\xi_1 \in \Xi, \ldots, \xi_n \in \Xi} \frac{N_F}{n^\lambda} = 
    \frac{N_F}{n^\lambda}.
\end{equation*}
\indent To derive the upper bound of Rademacher average of the compound function, $H(x,y,\xi) = (F(x,\xi) - y)^2$, we need to show that $H(x,y,\xi)$ is Lipschtiz continuous in $(x,z)$ given that $F(x,\xi)$ is H\"older continuous in $x$ and $F(x,\xi)$ is uniformly bounded.
\begin{lemma} \label{lemma:Lipschitz continuity of h}
    Suppose that Assumption \ref{assumption:part 1} holds. Then $H(x,y,\xi) \triangleq (F(x,\xi) - y)^2$ is Lipchitz continuous in $(x,y)$ with Lipschitz constant $4M_F\sqrt{L_F^2 + 1}$ (i.e., $|H(x_1,y_1,\xi) - H(x_2,y_2,\xi)| \leq 4M_F\sqrt{L_F^2 + 1} \|(x_1,y_1) - (x_2,y_2)\|, \ \forall \ (x_1,y_1), \ (x_2, y_2) \in X \times Y$). Furthermore, $H(x,y,\xi)$ is Lipschitz continuous in $y$ with Lipschitz constant $4M_F$ and $H(x,y,\xi)$ is uniformly bounded by $4M_F^2$. 
\end{lemma}
\proof
    See Appendix \ref{appendix:proof of lemma:Lipschitz continuity of h}.
\endproof

The upper bound of the Rademacher average of the compound function $H(x,y,\xi)$ follows from Lemmas \ref{lemma:rademacher complexity of objective:holder continuity} and \ref{lemma:Lipschitz continuity of h}, which is summarized below.
\begin{lemma} \label{lemma:rademacher complexity of compound objective:holder continuity}
   Let $\tilde{\sigma}_1, \tilde{\sigma}_2, \ldots, \tilde{\sigma}_n$ be i.i.d. random variables with $\tilde{\sigma}_i$ for $i = 1,2,\ldots, n$ being equally likely to be $1$ or $-1$. Let the Rademacher average of the set of sequences of $H$ be 
    $$
    R_n(H,\xi^n) = \mathbb{E}_{\tilde{\sigma}} \sup_{x \in X, y \in Y} \left| \frac{1}{n} \sum_{i=1}^n \tilde{\sigma}_i H(x,y,\xi_i) \right|. 
    $$
    Suppose that Assumption \ref{assumption:part 1} holds. Let $L_H = 4 M_F\sqrt{L_F^2 + 1}$ and $M_H = 4M_F^2$. For any $\lambda \in (0,\frac{1}{2})$, we have the following result:
    \begin{equation}
        R_n(H,\xi^n) \leq \frac{N_H}{n^\lambda},
    \end{equation}
    where $N_H =  L_H D (p+1)^{\frac{1}{2}} + M_H \sqrt{2 (\log2)} + \frac{M_F (p+1)^{1/2}}{\sqrt{(1 - 2 \lambda)e}}$. 
\end{lemma}
\proof
    This is a direct result of combining Lemmas \ref{lemma:rademacher complexity of objective:holder continuity} and \ref{lemma:Lipschitz continuity of h}. 
\endproof
In the following, we will use the Rademacher average to study the sampling error of the objective function estimates. Given a sample set $\xi^n$, we define the supremum of sampling error of the objective function estimate below.
\begin{definition} \label{def:supremum of sampling error of the objective function estimate}
    Let $\xi^n = \{\xi_1, \xi_2, \ldots, \xi_n\}$ denote the realizations of $n$ i.i.d. copies of $\tilde{\xi}$. The supremum of sampling error of the objective function estimate is defined as follows:
    \begin{equation} \label{eq:estimation error of f}
    \delta_n^f (\xi^n) = \sup_{x \in X} \left| \frac{1}{n} \sum_{i=1}^n F(x,\xi_i) - \mathbb{E}_{\tilde{\xi}}[F(x,\tilde{\xi})] \right|.
\end{equation}
\end{definition}

In Appendix \ref{appendix:symmetric argument of rademacher averages}, we explain how to use the Rademcaher average of a function class, shadow random variables (i.e., $\xi_1', \ldots, \xi_n'$ so that $\xi_1', \ldots, \xi_n', \xi_1, \ldots, \xi_n$ are i.i.d.), and symmetric argument to bound $\mathbb{E}[\delta_n^f (\tilde{\xi}^n)]$. For more details about the symmetric argument, please see \cite{boucheron2005theory}. \\
\indent We summarize the bound of $\delta_n^f (\xi^n)$ derived by Ermoliev and Norkin \cite{ermoliev2013sample} using Rademcaher average of a function class in the following theorem. 
\begin{theorem} \label{theorem:bound of function estimation error}
    Suppose that Assumption \ref{assumption:part 1} holds. Let $N_F$ be a constant defined in Lemma \ref{lemma:rademacher complexity of objective:holder continuity}. For any $\lambda \in (0,\frac{1}{2})$, then the following holds:
    \begin{enumerate}
        \item \begin{equation*}
    \mathbb{E}[\delta_n^f(\tilde{\xi}^n)] \leq 2 R_n(F,\Xi) \leq \frac{2 N_F}{n^\lambda},
    \end{equation*}
    \item \begin{equation*}
    \Pr\left\{n^\lambda \delta_n^f (\tilde{\xi}^n) \geq 2 N_F + t \right\} \leq \exp \left(-\frac{t^2}{2 M^2_F} \right) .
\end{equation*}
    \end{enumerate}
\end{theorem}
\begin{proof}
    See \cite[Corollary 3.2]{ermoliev2013sample}. 
\end{proof}
% second moment
By Assumption \ref{assumption:part 1}, we observe that $|\frac{1}{n} \sum_{i=1}^n F(x,\xi_i) - \mathbb{E}[F(x,\tilde{\xi})]| \leq 2 M_F$, which implies that $\delta_n^f(\xi^n) \leq 2 M_F$. One direct consequence of Theorem \ref{theorem:bound of function estimation error} is that  
\begin{equation} \label{eq:bound of the varaince}
\begin{aligned}
    \text{Var}[\delta_n^f(\tilde{\xi}^n)] &= \mathbb{E}[(\delta_n^f(\tilde{\xi}^n))^2] - \mathbb{E}[\delta_n^f(\tilde{\xi}^n)]^2 \\
    & \leq \mathbb{E}[(\delta_n^f(\tilde{\xi}^n))^2] \leq \mathbb{E}[2 M_F \delta_n^f(\xi^n)] \leq \frac{4M_F N_F}{n^\lambda}. 
\end{aligned}
\end{equation}

We derive the sample complexity of the sample variance in the next lemma.
\begin{lemma} \label{lemma:complexity of compound variance}
    Denote 
    \begin{equation*}
        \delta_n^h(\xi^n)  =  \sup_{x \in X} \left|  \frac{1}{n} \sum_{i=1}^n H\left(x, \frac{1}{n} \sum_{j=1}^n F(x,\xi_j), \xi_i \right) - \mathbb{E}_{\tilde{\xi}}\left[H\left(x,\mathbb{E}_{\tilde{\xi}}[F(x,\tilde{\xi})],\tilde{\xi} \right) \right] \right|
    \end{equation*}
    and 
    \begin{equation*}
        \hat{\delta}_n(\xi^n) = \sup_{x \in X, y \in Y}  \left| \frac{1}{n} \sum_{i=1}^n H(x,y,\xi_i) - \mathbb{E}_{\tilde{\xi}}\left[H(x,y,\tilde{\xi}) \right] \right|.
    \end{equation*}
    Suppose that Assumption \ref{assumption:part 1} holds. Then the following holds:
    \begin{enumerate}
    %\begin{itemize}
        \item \begin{equation*}
            \delta^h_n(\xi^n) \leq 4 M_F \delta^f_n(\xi^n) + \hat{\delta}_n(\xi^n)
            \end{equation*}
        \item \begin{equation*}
        \begin{aligned}
            \mathbb{E}[\delta^h_n(\tilde{\xi}^n)] &\leq 8 M_F R_n (F,\Xi) + 2 R_n (H,\Xi) \\
            &\leq \frac{8 M_F N_H + 2 N_F }{n^\lambda},
        \end{aligned}
        \end{equation*}
        where $\lambda \in (0,\frac{1}{2})$, $N_F$ is a constant defined in Lemma \ref{lemma:rademacher complexity of objective:holder continuity}, and $N_H$ is a constant defined in Lemma \ref{lemma:rademacher complexity of compound objective:holder continuity}. 
    %\end{itemize}
    \end{enumerate}
\end{lemma}
\proof
% move the proof to the appendix
See Appendix \ref{appendix:proof of lemma:complexity of compound variance}.     
\endproof

With a slight abuse of the notation, we let $\sigma(x) = \text{Var}[F(x,\tilde{\xi})]$. We observe that 
\begin{equation} \label{eq:bound:sample variance01}
    \begin{aligned}
        \left| s_n^2(x) - \sigma^2(x) \right| &= \left| \frac{n}{n-1}  \frac{1}{n} \sum_{i=1}^n H(x,\frac{1}{n} \sum_{i=1}^n F(x,\xi_k), \xi_i) - \sigma^2(x)\right| \\
        &\leq \left| \frac{1}{n} \sum_{i=1}^n H(x,\frac{1}{n} \sum_{i=1}^n F(x,\xi_k), \xi_i) - \sigma^2(x)\right| \\
        & \ + \frac{1}{n-1} \left|\frac{1}{n} \sum_{i=1}^n H(x,\frac{1}{n} \sum_{i=1}^n F(x,\xi_k), \xi_i) \right|.
    \end{aligned}
\end{equation}
By Lemmas \ref{lemma:Lipschitz continuity of h} and \ref{lemma:complexity of compound variance}, it follows from (\ref{eq:bound:sample variance01}) that 
\begin{equation} \label{eq:bound:sample variance02}
     \sup_{x \in X} \left| s_n^2(x,\tilde{\xi}^n) - \sigma^2(x) \right| \leq 4 M_F \delta^f_n(\tilde{\xi}^n) + \hat{\delta}_n(\tilde{\xi}^n) + \frac{4 M_F^2}{n-1},
\end{equation}
and hence,
\begin{equation*}
\begin{aligned}
    \mathbb{E}\left[\sup_{x \in X} \left| s_n^2(x,\tilde{\xi}^n) - \sigma^2(x) \right| \right] %&\leq 8 M_F R_n (F,\Xi) + 2 R_n (H,\Xi) + \frac{4 M_F^2}{n-1} \\
    &\leq \frac{8 M_F N_H + 2 N_F }{n^\lambda} + \frac{4 M_F^2}{n-1}. 
\end{aligned}
\end{equation*}
Let $Z_{1 - \frac{\alpha}{2}}$ be the Z score of two-sided $1 - \alpha$ confidence interval (i.e., $Z_{1 - \frac{\alpha}{2}} = \Phi^{-1}(1 - \frac{\alpha}{2})$, where $\Phi^{-1}(\cdot)$ is the inverse of the cumulative distribution function of standard normal distribution). Let $s^2_{n}(x;\xi^n_i)$ be the sample variance of $F(x,\xi)$ in the $i^\text{th}$ replication. With the bound in (\ref{eq:bound:sample variance02}), we could derive the sample complexity of the margin of error with $m$ replications as follows (see Appendix \ref{appendix:margin of error} for detailed derivation):
\begin{equation} \label{eq:margin of error:bound}
\begin{aligned}
    \mathbb{E}\left[\frac{1}{m}\sqrt{\frac{\sum_{i=1}^m s^2_{n}(x;\tilde{\xi}^n_i)}{n}} Z_{1 - \frac{\alpha}{2}}\right] &\leq \sqrt{\frac{\sigma^2(x)}{mn} + \frac{8 M_F N_H + 2 N_F}{mn^{(1+\lambda)}} + \frac{4 M_F^2}{mn(n-1)}} Z_{1 - \frac{\alpha}{2}} \\
    &\leq O((mn)^{-\frac{1}{2}}).
\end{aligned}
\end{equation}
The upper bound given in (\ref{eq:margin of error:bound}) not only agrees with the common sense, which is $\left(O\left((mn)^{-\frac{1}{2}} \right) \right)$, but also indicates that the bias term ($\frac{8 M_F N_H + 2 N_F}{mn^{(1+\lambda)}} + \frac{4 M_F^2}{mn(n-1)}$) diminishes faster than the unbiased term $\frac{\sigma^2(x)}{mn}$.

\section{Sample Complexity of Compromise Decisions} \label{sec:classic compromise SP}
%\subsection{Problem Formulation}
In this section, we shall formulate an \textsl{exact} compromise decision problem and its \textsl{inxeact} counterpart. We aim to study the quantitative reliability of the associated compromise decisions. %Assumptions  
We make one additional convexity assumption of the random cost function, $F(x,\xi)$, and one additional assumption on the sampling scheme, below:
\begin{assumption} \label{assumption:part 2}
    $F(x,\xi)$ is convex in $x \in X$ for every $\xi \in \Xi$.
\end{assumption}
\begin{assumption} \label{assumption:part 2(2)}
    Given integer replication number $m \in \mathbb{Z}_+$ and integer sample size per replication $n \in \mathbb{Z}_+$. Let $N \triangleq m n$ denote the total sample size. Assume that $\tilde{\xi}_1, \tilde{\xi}_2, \ldots, \tilde{\xi}_N$ are i.i.d. random variables which follow the distribution of $\tilde{\xi}$.\\
    \indent Furthermore, Let $\tilde{\xi}^N \triangleq \{\tilde{\xi}_1, \tilde{\xi}_2, \ldots, \tilde{\xi}_N \}$ denote the collection of all samples.
    For $i \in \{1,2,\ldots,m\}$, let $\tilde{\xi}^n_i \triangleq \{\tilde{\xi}_{n(i-1) + 1}, \tilde{\xi}_{n(i-1) + 2}, \ldots, \tilde{\xi}_{n i} \}$ denote the collection of samples used in the $i^{\text{th}}$ replication. 
\end{assumption}

 We start with a basic formulation where we can solve each replication of the SAA problem to optimality and use it to give a brief introduction to the compromise decision methodology. The compromise decision methodology consists of two key steps, which are {\sl replication} step and {\sl aggregation} step (see Algorithm \ref{alg:compromise decision approach}). In the {\sl replication} step, we build a SAA function for each replication, 
 \begin{equation} \label{eq:SAA function in each replication}
     f_{n}(x;\xi^n_i) = \frac{1}{n} \sum_{j=1}^{n} F(x,\xi_{(i-1)n + j}), \ i = 1,2,\ldots,m. 
 \end{equation}
 Then we solve the SAA problem to obtain the minimizer of $f_{n}(x;\xi^n_i)$, which is denoted by $x_{n}(\xi^n_i)$. The problem below illustrates the {\sl replication} step. 
\begin{equation*}
    x_{n}(\xi^n_i) \in \arg \min_{x \in X} f_{n}(x;\xi^n_i), \ i = 1,2,\ldots,m.
\end{equation*}
In the {\sl aggregation} step, we find the compromise decision that minimizes the {\sl aggregation} of the information from SAA functions and their corresponding minimizers. Given a regularizer $\rho \in (0, \infty)$, this {\sl aggregation} problem (which we refer to as \textsl{exact} compromise decision problem) can be formulated as follows: 
\begin{equation} \label{eq:exact compromise SP}
    \min_{x \in X} \frac{1}{m} \sum_{i=1}^m f_{n}(x;\xi^n_i) + \frac{\rho}{2} \left\|x - \frac{1}{m} \sum_{i=1}^m x_{n}(\xi^n_i) \right\|^2.
\end{equation}
We let $\tilde{x}_{N,\rho}(\xi^N)$ denote the optimal solution of the \textsl{exact} compromise decision problem in (\ref{eq:exact compromise SP}). Let $\theta_{N,\rho}(\xi^N)$ and $\theta^*$ denote the optimal values of \textsl{exact} compromise decision problem in (\ref{eq:exact compromise SP}) and true problem in (\ref{eq:generic sp}), respectively. We shall show that $\mathbb{E}[|\theta_{N,\rho}(\tilde{\xi}^N) - \theta^*|] \leq \frac{6m - 4}{m n^\lambda} N_F$ and $\text{Var}\left[|\theta_{N,\rho}(\tilde{\xi}^N) - \theta^*| \right] \leq \frac{36 M_F N_F}{m n^{\lambda}} + \frac{36 N_F^2}{n^{2\lambda}}$, for $\lambda \in (0, \frac{1}{2})$ and a problem-specific constant $N_F$ defined in Lemma \ref{lemma:rademacher complexity of objective:holder continuity}. 

% epsilon optimal solution
In the second part of this section, we will focus on a more general case in which only inexact optimal solutions from $m$ replications are accessible to us. Let $X_{n,\epsilon}(\xi^n_i)$ denote an $\epsilon$-optimal solution set of the $i^\text{th}$ replication of the SAA problem:
\begin{equation} \label{eq:epsilon-optimal solution set for each replication}
    X_{n,\epsilon}(\xi^n_i) = \bigg\{x \in X: f_n(x; \xi^n_i) \leq \text{minimum}\Big\{f_{n}(x; \xi^n_i) | x \in X \Big\} + \epsilon \bigg\}
\end{equation}
We let $x_{n,\epsilon}(\xi^n_i)$ denote one element of $X_{n,\epsilon}(\xi^n_i)$ (i.e., $x_{n,\epsilon}(\xi^n_i) \in X_{n,\epsilon}(\xi^n_i)$). As a result, the \textsl{inexact} compromise decision problem which allows a collection of $\epsilon$-optimal solutions from all the $m$ replications is written below:
\begin{equation} \label{eq:inexact compromise SP}
    \min_{x \in X} \frac{1}{m} \sum_{i=1}^m f_{n}(x; \xi^n_i) + \frac{\rho}{2} \left\|x - \frac{1}{m} \sum_{i=1}^m x_{n,\epsilon}(\xi^n_i) \right\|^2.
\end{equation}
Let $\hat{X}_{N,\rho,\epsilon}$ and $X^*_\epsilon$ denote $\epsilon$-optimal solution sets of \textsl{inexact} compromise decision problem in (\ref{eq:inexact compromise SP}) and true problem in (\ref{eq:generic sp}), respectively.  We shall measure the reliability of the compromise decisions to (\ref{eq:inexact compromise SP}) by showing that 
\begin{enumerate}
    \item $\mathbb{E}[\Delta (\hat{X}_{N,\rho,\epsilon}(\tilde{\xi}^N), X^*_\epsilon)] \leq \frac{D_X N_F }{\epsilon} \frac{8m - 4}{m n^{\lambda}} + 2 \sqrt{\epsilon/ \rho}$.
    \item $\Pr \left\{\frac{\epsilon n^\lambda}{2 D_X} \Delta(\hat{X}_{N,\rho,\epsilon}(\tilde{\xi}^N), X_{\epsilon}^*) \geq C + t \right\} \leq m \exp \left\{-\frac{m^2 t^2}{2 M^2_F(2m - 1)^2} \right\}$ for $\rho = n$.
    \item $\text{Var}[\Delta(\hat{X}_{N,\rho,\epsilon}(\tilde{\xi}^N), X_{\epsilon}^*)] \leq \frac{64 D_X^2 M_F N_F}{m \epsilon^2 n^\lambda} + (\frac{8D_X N_F}{\epsilon n^\lambda} + 2 \sqrt{\epsilon / \rho})^2.$
\end{enumerate}

\subsection{Sample Complexity of the Exact Compromise Decision Problem} \label{sec:sample complexity:cost}
 In this subsection, we focus on deriving the sample complexity of the optimal cost of the \textsl{exact} compromise decision problem (\ref{eq:exact compromise SP}). Recall that we let $f(x)$ denote the expectation of the random cost function (i.e., $f(x) = \mathbb{E}_{\tilde{\xi}}[F(x,\tilde{\xi})]$ ) and let $\theta^*$ denote the optimal cost of the true problem (i.e., $\theta^* = \min_{x \in X} f(x)$).\\ %We also let $\xi^N$ denote the collection of all the samples across $m$ replications. \\
 %\indent In this subsection, we aim to study the sample complexity of the optimal cost of problem (\ref{eq:exact compromise SP}). 
 \indent We start with introducing several notations in the {\sl replication} step. We let $\theta_{n}(\xi^n_i)$ denote the optimal cost of the $i^{\text{th}}$ SAA problem:
 \begin{equation} \label{eq:optimal cost of one replication of SAA}
     \theta_{n}(\xi^n_i) = \min_{x \in X} f_{n} (x;\xi^n_i), \ i = 1,2,\ldots,m.
 \end{equation}
 %We let $\tilde{x}_{N,\rho}(\xi^N)$ and $\theta_{N,\rho}(\xi^N)$ denote the optimal solution and optimal cost of the following Compromise Decision problem, respectively:
 %\begin{equation*}
 %    \begin{aligned}
 %       \min_{x \in X} \frac{1}{m} \sum_{i=1}^m f_{n}(x;\xi^n_i) + \frac{\rho}{2} 
 %        \left\|x - \frac{1}{m} \sum_{i=1}^m x_{n}(\xi^n_i) \right\|^2.
         %&\tilde{x}_{N,\rho}(\xi^N) \in \arg \min_{x \in X} \frac{1}{m} \sum_{i=1}^m f_{n}(x;\xi^n_i) + \frac{\rho}{2} 
         %\left\|x - \frac{1}{m} \sum_{i=1}^m x_{n}(\xi^n_i) \right\|^2, \\
         %& \theta_{N,\rho}(\xi^N) = \min_{x \in X} \frac{1}{m} \sum_{i=1}^m f_{n}(x;\xi^n_i) + \frac{\rho}{2} \left\|x - \frac{1}{m} \sum_{i=1}^m x_{n}(\xi^n_i) \right\|^2.
 %    \end{aligned}
 %\end{equation*}
In the last section, we have defined the estimation error of objective function estimate in each replication as follows: 
\begin{equation} \label{eq:estimate error}
    \begin{aligned}
        \delta^f_{n}(\xi^n_i) = \sup_{x \in X} |f_{n} (x;\xi^n_i) - f(x)|, \ i = 1,2,\ldots,m.
    \end{aligned}
\end{equation}
With a symmetric argument (see \cite{boucheron2005theory,ermoliev2013sample} or (\ref{eq:symmetric05}) in Appendix \ref{appendix:symmetric argument of rademacher averages}), we have shown that $\mathbb{E}[\delta^f_{n}(\tilde{\xi}^n_i)] \leq 2 R_n(F,\Xi) \leq \frac{2 N_F}{n^{\lambda}}$ for any $\lambda \in (0,\frac{1}{2})$.\\
\indent Next, we provide a relation for deriving the upper bound of the variance of interest. Given two non-negative random variables $X, Y$ so that $0 \leq X \leq Y \ a.s.$, we have the following relationship between $\text{Var}[X]$ and $\text{Var}[Y]$: 
 \begin{equation} \label{eq:common bound on the variance}
     \text{Var}[X] = \mathbb{E}[X^2] - \left(\mathbb{E}[X] \right)^2 \leq \mathbb{E}[X^2] \leq \mathbb{E}[Y^2] = \text{Var}[Y] + \left(\mathbb{E}[Y] \right)^2,
 \end{equation}

\begin{theorem} \label{thm:sample complexity cost of compromise decision problem}
    Suppose that Assumptions \ref{assumption:part 1}, \ref{assumption:part 2}, and \ref{assumption:part 2(2)} hold. Let $N_F$ be a constant defined in Lemma \ref{lemma:rademacher complexity of objective:holder continuity}. Then the following hold:
    \begin{enumerate} 
        \item \begin{equation*}
        \mathbb{E}\left[|\theta_{N,\rho}(\tilde{\xi}^N) - \theta^*|\right] \leq \frac{6m - 4}{m} R_n(F,\Xi) \leq \frac{6m - 4}{m n^\lambda} N_F, 
    \end{equation*}
        \item  \begin{equation*}
            \text{Var}\left[|\theta_{N,\rho}(\tilde{\xi}^N) - \theta^*| \right] \leq \frac{36 M_F N_F}{m n^{\lambda}} + \frac{36 N_F^2}{n^{2\lambda}},
        \end{equation*}
    \end{enumerate}
    
   \noindent where $R_n(F,\Xi)$ is the Rademacher average associated with $F$ and sample size $n$ in Definition \ref{def:rademacher average of a function class}, and $\lambda \in (0, \frac{1}{2})$.%, $N_F = L_F D^{\gamma} d^{\frac{\gamma}{2}} + {M_F \sqrt{2 (\log2)}} + \frac{M_F d^{1/2}}{\sqrt{\gamma (1 - 2 \lambda)e}} $.
\end{theorem}
\proof
To begin with, we give an overview of the proof strategy. The proof consists of three steps. In step 1, we aim to use the supremum of the sampling error defined in (\ref{eq:estimation error of f})
to bound $|\theta_{N,\rho}(\tilde{\xi}^N) - \theta^*|$. The idea is based on the convexity of the objective function in Assumption \ref{assumption:part 2} and triangular inequality. By convexity, we observe that $f_{n} \left(\frac{1}{m} \sum_{j=1}^m x_{n}(\tilde{\xi}^n_j);\tilde{\xi}^n_i \right) \leq  \frac{1}{m} \sum_{j=1}^m f_{n} \left(x_{n}(\tilde{\xi}^n_j);\tilde{\xi}^n_i \right)$. Note that
$|f_{n}\left(x;\tilde{\xi}^n_i \right) - f(x)|$ is bounded above by $\delta^f_{n}(\tilde{\xi}^n_i)$ and $|f_{n}\left(x;\tilde{\xi}^n_j \right) - f(x)|$ is bounded above by $\delta^f_{n}(\tilde{\xi}^n_j)$. Then by triangular inequality, $|f_{n}\left(x;\tilde{\xi}^n_i \right) - f_{n}\left(x;\tilde{\xi}^n_j \right)|$ is bounded above by $\delta^f_{n}(\tilde{\xi}^n_i) + \delta^f_{n}(\tilde{\xi}^n_j)$. \\
\indent In step 2, we utilize Theorem \ref{theorem:bound of function estimation error} to derive the upper bound of the supremum of the sampling error in expectation. In step 3, we use the variance relation in (\ref{eq:common bound on the variance}) and the upper bound of the variance of supremum of the sampling error in (\ref{eq:bound of the varaince}) to get the upper bound of $\text{Var}\left[|\theta_{N,\rho}(\tilde{\xi}^N) - \theta^*| \right]$. \\

\noindent \textbf{(Step 1: Upper Bound of $|\theta_{N,\rho}(\tilde{\xi}^N) - \theta^*|$)}. We aim to derive the upper bound and lower bound of $\theta_{N,\rho}(\tilde{\xi}^N) - \theta^*$ separately. By optimality of the problem (\ref{eq:exact compromise SP}), we have 
\begin{equation} \label{eq:complexity:cost1}
\begin{aligned}
    \theta_{N,\rho}(\tilde{\xi}^N) &= \frac{1}{m} \sum_{i=1}^m f_{n}\left(\tilde{x}_{N,\rho}(\tilde{\xi}^N);\tilde{\xi}^n_i \right) + \frac{\rho}{2} \left\| \tilde{x}_{N,\rho}(\tilde{\xi}^N) - \frac{1}{m} \sum_{i=1}^m x_{n}(\tilde{\xi}^n_i) \right\|^2 \\
    &\leq \frac{1}{m} \sum_{i=1}^m f_{n} \left(\frac{1}{m} \sum_{j=1}^m x_{n}(\tilde{\xi}^n_j);\tilde{\xi}^n_i \right) \\
    &\leq \frac{1}{m^2} \sum_{i=1}^m  \sum_{j=1}^m f_{n} \left(x_{n}(\tilde{\xi}^n_j);\tilde{\xi}^n_i \right) \quad  \text{by convexity of } f_{n} (\cdot ;\tilde{\xi}^n_i).
\end{aligned}
\end{equation}
On the other hand, for $i, j \in \{1,2,\ldots,m\}$ and $i \neq j$, we have
\begin{equation} \label{eq:complexity:cost2}
\begin{aligned}
    f_{n}\left(x_{n}(\tilde{\xi}^n_j);\tilde{\xi}^n_i \right) &=  f_{n}\left(x_{n}(\tilde{\xi}^n_j);\tilde{\xi}^n_j \right) + f_{n}\left(x_{n}(\tilde{\xi}^n_j);\tilde{\xi}^n_i\right) - f_{n}\left(x_{n}(\tilde{\xi}^n_j);\tilde{\xi}^n_j\right)
    \\
    &\leq \theta_{n}(\tilde{\xi}^n_j) + \left|f_{n}\left(x_{n}(\tilde{\xi}^n_j);\tilde{\xi}^n_i \right) - f_{n}\left(x_{n}(\tilde{\xi}^n_j);\tilde{\xi}^n_j \right) \right| \\
    &\leq \theta_{n}(\tilde{\xi}^n_j) + \left|f_{n}\left(x_{n}(\tilde{\xi}^n_j);\tilde{\xi}^n_i \right) - f\left(x_{n}(\tilde{\xi}^n_j)\right) \right| \\
    &\quad + \left|f_{n}\left(x_{n}(\tilde{\xi}^n_j);\tilde{\xi}^n_j \right) - f\left(x_{n}(\tilde{\xi}^n_j) \right) \right| \\
    &\leq \theta_{n}(\tilde{\xi}^n_j) + \delta^f_{n}(\tilde{\xi}^n_i) + \delta^f_{n}(\tilde{\xi}^n_j), \quad \text{ by (\ref{eq:estimate error})}.
\end{aligned}
\end{equation}
Let $x^* \in \arg \min_{x \in X} f(x)$, then we have 
\begin{equation} \label{eq:complexity:cost2:02}
    \begin{aligned}
        \theta_{n}(\tilde{\xi}^n_i) - \theta^* &\leq f_{n}(x^*;\tilde{\xi}^n_i) - \theta^*  \quad \quad \text{by the definition of } \theta_{n}(\tilde{\xi}^n_i) \text{ in (\ref{eq:optimal cost of one replication of SAA})}\\
        & = f_{n}(x^*;\tilde{\xi}^n_i) - f(x^*) \\
        & \leq \sup_{x \in X} |f_{n}(x;\tilde{\xi}^n_i) - f(x)| = \delta^f_{n}(\tilde{\xi}^n_i).
    \end{aligned}
\end{equation}
Furthermore, we have 
\begin{equation}
    \begin{aligned}
        \theta_{n}(\tilde{\xi}^n_i) - \theta^* &\geq f_{n}\left(x_{n}(\tilde{\xi}^n_i);\tilde{\xi}^n_i \right) - f\left(x_{n}(\tilde{\xi}^n_i)\right) && \text{ by } f\left(x_{n}(\tilde{\xi}^n_i) \right) \geq \theta^* \\
        & \geq - \sup_{x \in X} \left|f_{n}(x;\tilde{\xi}^n_i) - f(x) \right| = - \delta^f_{n}(\tilde{\xi}^n_i).
    \end{aligned}
\end{equation}
The combination of (\ref{eq:complexity:cost1}) and (\ref{eq:complexity:cost2}) implies that 
\begin{equation} \label{eq:complexity:cost3}
    \theta_{N,\rho}(\tilde{\xi}^N) \leq \frac{1}{m} \sum_{i=1}^m \theta_{n}(\tilde{\xi}^n_i) + \frac{1}{m^2}\sum_{\substack{i,j=1,\\ i \neq j}}^m \left(\delta^f_{n}(\tilde{\xi}^n_i) + \delta^f_{n}(\tilde{\xi}^n_j) \right).
\end{equation}
By subtracting $\theta^*$ from both sides of (\ref{eq:complexity:cost3}), we have the upper bound of $\theta_{N,\rho}(\tilde{\xi}^N) - \theta^*$ shown below:
\begin{equation} \label{eq:complexity:cost4}
\begin{aligned}
    \theta_{N,\rho}(\tilde{\xi}^N) - \theta^* &\leq \frac{1}{m} \sum_{i=1}^m \left(\theta_{n}(\tilde{\xi}^n_i) - \theta^* \right) + \frac{1}{m^2}\sum_{\substack{i,j=1,\\ i \neq j}}^m \left(\delta^f_{n}(\tilde{\xi}^n_i) + \delta^f_{n}(\tilde{\xi}^n_j) \right) \\ 
    &\leq \frac{1}{m} \sum_{i=1}^m \delta^f_{n}(\tilde{\xi}^n_i) + \frac{1}{m^2} \sum_{\substack{i,j=1,\\ i \neq j}}^m \left(\delta^f_{n}(\tilde{\xi}^n_i) + \delta^f_{n}(\tilde{\xi}^n_j) \right) . 
\end{aligned} 
\end{equation}
As for the lower bound of $\theta_{N,\rho}(\tilde{\xi}^N) - \theta^*$, we observe that 
\begin{equation*}
    \begin{aligned}
        \frac{1}{m} \sum_{i=1}^m f_{n}(\tilde{x}_{N,\rho}(\tilde{\xi}^N);\tilde{\xi}^n_i) + \frac{\rho}{2} \left\| \tilde{x}_{N,\rho}(\tilde{\xi}^N) - \frac{1}{m} \sum_{i=1}^m x_{n}(\tilde{\xi}^n_i) \right\|^2 &\geq \frac{1}{m} \sum_{i=1}^m f_{n} \left(\tilde{x}_{N,\rho}(\tilde{\xi}^N);\tilde{\xi}^n_i \right) \\
        &\geq \frac{1}{m} \sum_{i=1}^m \theta_{n}(\tilde{\xi}^n_i). %\ \text{by definition of } \theta_{n}^i(\tilde{\xi}^n_i).
    \end{aligned}
\end{equation*}
Hence, this implies that 
\begin{equation} \label{eq:complexity:cost5}
\begin{aligned}
    \theta_{N,\rho}(\tilde{\xi}^N) - \theta^* &\geq \frac{1}{m} \sum_{i=1}^m \left(\theta_{n}(\tilde{\xi}^n_i) - \theta^* \right) \\
    & \geq -\frac{1}{m} \sum_{i=1}^m \delta^f_{n}(\tilde{\xi}^n_i)
\end{aligned}
\end{equation}
Combining (\ref{eq:complexity:cost4}) and (\ref{eq:complexity:cost5}) we derive the upper bound of $\left|\theta_{N,\rho}(\tilde{\xi}^N) - \theta^* \right|$ below:
\begin{equation} \label{eq:complexity:cost6}
    \left|\theta_{N,\rho}(\tilde{\xi}^N) - \theta^* \right|
    \leq \frac{1}{m} \sum_{i=1}^m \delta_{n}(\tilde{\xi}^n_i) + \frac{1}{m^2} \sum_{\substack{i,j=1,\\ i \neq j}}^m \left(\delta^f_{n}(\tilde{\xi}^n_i) + \delta^f_{n}(\tilde{\xi}^n_j) \right).  
\end{equation}
\noindent \textbf{(Step 2: Upper bound of $\mathbb{E}\left[\left|\theta_{N,\rho}(\tilde{\xi}^N) - \theta^*\right| \right]$)}. Taking expectation on both sides of (\ref{eq:complexity:cost6}) and applying Theorem \ref{theorem:bound of function estimation error}, we get 
\begin{equation*} 
\begin{aligned}
    \mathbb{E}\left[\left|\theta_{N,\rho}(\tilde{\xi}^N) - \theta^*\right| \right] &\leq 2 R_n(F,\Xi) + \frac{4m(m-1)}{m^2} R_n(F,\Xi) \\
    &= \frac{6m - 4}{m n^\lambda} N_F. 
\end{aligned}
\end{equation*}
\noindent \textbf{(Step 3: Upper bound of $\text{Var}\left[\left|\theta_{N,\rho}(\tilde{\xi}^N) - \theta^*\right|\right]$)}. Next, we will derive the upper bound of the variance of $|\theta_{N,\rho}(\tilde{\xi}^N) - \theta^*|$. According to (\ref{eq:complexity:cost6}), we have 
\begin{equation} \label{eq:complexity:cost7}
    \left|\theta_{N,\rho}(\tilde{\xi}^N) - \theta^* \right| \leq \sum_{i=1}^m \left(\frac{1}{m} + \frac{2(m-1)}{m^2} \right) \delta^f_n(\tilde{\xi}^n_i) \leq \sum_{i=1}^m \frac{3}{m} \delta^f_n(\tilde{\xi}^n_i).
\end{equation}
By Assumption \ref{assumption:part 2(2)}, $\tilde{\xi}^n_i$ is independent of $\tilde{\xi}^n_j$ for $i\neq j$. Hence, we have 
\begin{equation} \label{eq:complexity:cost8}
    \text{Var}\left[\sum_{i=1}^m \frac{3}{m} \delta^f_n(\tilde{\xi}^n_i)\right] = \frac{9}{m^2} \sum_{i=1}^m \text{Var}\left[\delta^f_n(\tilde{\xi}^n_i)\right] \leq \frac{36 M_F N_F}{m n^{\lambda}}.
\end{equation}
The last inequality in (\ref{eq:complexity:cost8}) is obtained by using the upper bound of the variance of $\delta^f_n(\tilde{\xi}^n)$ in (\ref{eq:bound of the varaince}). Again, by Assumption \ref{assumption:part 2(2)}, we have 
\begin{equation} \label{eq:complexity:cost9}
    \mathbb{E}\left[\sum_{i=1}^m \frac{3}{m} \delta^f_n(\tilde{\xi}^n_i)\right] \leq \frac{6 N_F}{n^\lambda}.
\end{equation}
Hence, using the relation in (\ref{eq:common bound on the variance}), and combining  (\ref{eq:complexity:cost7}) - (\ref{eq:complexity:cost9}), we have
\begin{equation*}
    \text{Var}\left[\left|\theta_{N,\rho}(\tilde{\xi}^N) - \theta^*\right|\right] \leq \frac{36 M_F N_F}{m n^{\lambda}} + \frac{36 N_F^2}{n^{2\lambda}}.
\end{equation*}
\endproof

In Theorem \ref{thm:sample complexity cost of compromise decision problem}, the upper bound of the variance of $|\theta_{N,\rho}(\tilde{\xi}^N) - \theta^*|$ consists of two components, where $\frac{36 M_F N_F}{m n^{\lambda}}$ is due to the average of $m$ replications of sample average approximation of the true function and $\frac{36 N_F^2}{n^{2\lambda}}$ is due to the sampling error from each replication.

\subsection{Sample Complexity of the Inexact Compromise Decision Problem} \label{sec:sample complexity:solution}
Here, we aim to study the reliability of the $\epsilon-$optimal compromise decisions, $\hat{x}_{N, \rho, \epsilon}(\xi^N)$, to the \textsl{inexact} compromise decision problem (\ref{eq:inexact compromise SP}). To achieve this goal, we first need to find a suitable upper bound of the pessimistic distance, $\Delta(\hat{X}_{N,\rho,\epsilon}(\xi^N), X_{\epsilon}^*)$. With the help of theorem of the Lipschtizan behavior of the $\epsilon$-optimal solution set (Theorem \ref{theorem:lipschitz continuity of solution set}), we will show that such upper bound could be further bounded by a function of the supremum of the sampling error of the objective function estimate (see Definition \ref{def:supremum of sampling error of the objective function estimate}). \\
\indent To begin with, we introduce several notations in the {\sl aggregation} step to ease the analysis. 
Let $\bar{f}_N(x;\xi^N)$ denote the average of SAA functions across $m$ replications: 
\begin{equation} \label{eq:aggregated SAA objective function}
    \bar{f}_N(x;\xi^N) = \frac{1}{m} \sum_{i=1}^m f_{n}(x;\xi^n_i).
\end{equation}
We then let $\bar{\theta}_{N}(\xi^N)$ denote the minimum of $\bar{f}_N(x;\xi^N)$ over $X$:
\begin{equation} \label{eq:aggregated SAA problem}
    \bar{\theta}_{N}(\xi^N) = \min_{x \in X} \bar{f}_N(x;\xi^N).
\end{equation}
We further let $\bar{X}_{N,\epsilon}(\xi^N)$ denote $\epsilon$-optimal solution set of the aggregated SAA problem in (\ref{eq:aggregated SAA problem}):
\begin{equation} \label{eq:epsilon-optimal solution set:agg}
    \bar{X}_{N,\epsilon}(\xi^N) = \left\{x \in X: \bar{f}_N(x;\xi^N) \leq \bar{\theta}_{N}(\xi^N) + \epsilon \right\}.
\end{equation}
We now proceed to introduce several notations for the \textsl{inexact} compromise decision problem in (\ref{eq:inexact compromise SP}).
Given $x_{n,\epsilon}(\xi^n_i) \in X_{n,\epsilon}(\xi^n_i), \ i = 1,\ldots,m$, defined in (\ref{eq:epsilon-optimal solution set for each replication}), we let $\theta_{N,\rho,\epsilon}(\xi^N)$ denote the optimal cost of the \textsl{inexact} compromise decision problem:
\begin{equation} \label{eq:compromise decision problem:SAA}
    \theta_{N,\rho,\epsilon}(\xi^N) = \min_{x \in X} \bar{f}_N(x;\xi^N) + \frac{\rho}{2} \left\|x - \frac{1}{m} \sum_{i=1}^m x_{n,\epsilon}(\xi^n_i) \right\|^2.
\end{equation}
The associated $\epsilon$-optimal solution set of the \textsl{inexact} compromise decision problem is defined as follows:
\begin{equation} \label{notations:inexact compromise decicisions}
        \hat{X}_{N, \rho, \epsilon}(\xi^N) = \left\{x \in X: \bar{f}_N(x;\xi^N) + \frac{\rho}{2} \left\|x - \frac{1}{m} \sum_{i=1}^m x_{n,\epsilon}(\xi^n_i) \right\|^2 \leq \theta_{N,\rho,\epsilon} + \epsilon \right\}. 
\end{equation}
Recall that $\theta^*$ is the optimal cost of the true problem in (\ref{eq:generic sp}).%(i.e., $\theta^* = \min\limits_{x \in X} f(x)$) 
We let $X^*_{\epsilon}$ denote the $\epsilon$-optimal solution set of the true problem: 
\begin{equation} \label{eq:epsilon optimal solution set:true problem}
    X^*_{\epsilon} = \left\{x \in X: f(x) \leq \theta^* + \epsilon \right\}.
\end{equation}

The following lemma sets up the relationship among elements of $\epsilon$-optimal solution set of the \textsl{inexact} compromise decision problem, $m$ replications of SAA problems, and the true problem. 
\begin{lemma} \label{lemma:compromise:relation}
     Let $\hat{X}_{N,\rho,\epsilon}(\xi^N)$, $\bar{X}_{N,\epsilon}(\xi^N)$, $X_{\epsilon}^*$, and $X_{n,\epsilon}(\xi^n_j)$ be defined in (\ref{notations:inexact compromise decicisions}), (\ref{eq:epsilon-optimal solution set:agg}), (\ref{eq:epsilon optimal solution set:true problem}), and (\ref{eq:epsilon-optimal solution set for each replication}), respectively. Let $\hat{x}_{N,\rho,\epsilon}(\xi^N) \in \hat{X}_{N,\rho,\epsilon}(\xi^N)$, $\bar{x}_{N,\epsilon}(\xi^N) \in \bar{X}_{N,\epsilon}(\xi^N)$, $x_{\epsilon}^* \in X_{\epsilon}^*$, $x_{n,\epsilon}(\xi^n_j) \in X_{n,\epsilon}(\xi^n_j)$, $j = 1,2,\ldots,m$. Then the following relation holds:
    \begin{equation*}
         \left\|\hat{x}_{N,\rho,\epsilon}(\xi^N) - x_{\epsilon}^* \right\| \leq \left\| \bar{x}_{N,\epsilon}(\xi^N) - \frac{1}{m} \sum_{j=1}^m x_{n,\epsilon}(\xi^n_j) \right\| + \left\|\frac{1}{m} \sum_{j=1}^m x_{n,\epsilon}(\xi^n_j) - x_{\epsilon}^* \right\| + 2 \sqrt{\epsilon / \rho}. 
    \end{equation*}
\end{lemma}
% move to the appendix
\proof
We provide a brief sketch here, and postpone the details to Appendix \ref{appendix:proof of lemma:compromise:relation}. In essence, we use the optimality of the compromise decision problem in (\ref{eq:compromise decision problem:SAA}) to bound $\left\| \hat{x}_{N,\rho,\epsilon}(\xi^N) - \frac{1}{m} \sum_{j=1}^m x_{n,\epsilon}(\xi^n_j) \right\|$. The final result follows from an  application of the triangle inequality. .
\endproof

With the upper bound of $\|\hat{x}_{N,\rho,\epsilon}(\xi^N) - x_{\epsilon}^* \|$ obtained in Lemma \ref{lemma:compromise:relation}, we derive the upper bound of pessimistic distance of $\hat{X}_{N,\rho,\epsilon}(\xi^N)$ to $X_{\epsilon}^*$, $\Delta(\hat{X}_{N,\rho,\epsilon}(\xi^N), X_{\epsilon}^*)$, in the next lemma.
\begin{lemma} \label{lemma:compromise:relation2}
      Let $\hat{X}_{N,\rho,\epsilon}(\xi^N)$, $\bar{X}_{N,\epsilon}(\xi^N)$, $X_{\epsilon}^*$, and $X_{n,\epsilon}(\xi^n_j)$ be defined in (\ref{notations:inexact compromise decicisions}), (\ref{eq:epsilon-optimal solution set:agg}), (\ref{eq:epsilon optimal solution set:true problem}), and (\ref{eq:epsilon-optimal solution set for each replication}), respectively. Let $\hat{x}_{N,\rho,\epsilon}(\xi^N) \in \hat{X}_{N,\rho,\epsilon}(\xi^N)$, $\bar{x}_{N,\epsilon}(\xi^N) \in \bar{X}_{N,\epsilon}(\xi^N)$, $x_{\epsilon}^* \in X_{\epsilon}^*$, $x_{n,\epsilon}(\xi^n_j) \in X_{n,\epsilon}(\xi^n_j)$, $j = 1,2,\ldots,m$. Then the following relation holds:
    \begin{equation*}
    \begin{aligned}
        \Delta \left(\hat{X}_{N,\rho,\epsilon}(\xi^N), X_{\epsilon}^* \right) &\leq  \inf_{x \in \bar{X}_{N,\epsilon}(\xi^N) } \left\| \frac{1}{m} \sum_{j=1}^m x_{n,\epsilon}(\xi^n_j) - x \right\| \\
        & \quad + \inf_{x \in X_{\epsilon}^*} \left\|\frac{1}{m} \sum_{j=1}^m x_{n,\epsilon}(\xi^n_j) - x \right\| + 2 \sqrt{\epsilon / \rho}.
    \end{aligned}
    \end{equation*}
\end{lemma}
\proof 
See Appendix \ref{appendix:proof of lemma:compromise:relation2}. 
\endproof
In the next result, we aim to use Theorem \ref{theorem:lipschitz continuity of solution set} and estimation errors defined in (\ref{eq:estimation error of f}) to bound $\inf\limits_{x \in \bar{X}_{N,\epsilon}(\xi^N) } \| \frac{1}{m} \sum_{j=1}^m x_{n,\epsilon}(\xi^n_j) - x\|$.
\begin{lemma} \label{lemma:compromise:relation3}
    Let $\bar{X}_{N,\epsilon}(\xi^N)$ and $X_{n,\epsilon}(\xi^n_j)$ be defined in (\ref{eq:epsilon-optimal solution set:agg}) and (\ref{eq:epsilon-optimal solution set for each replication}), respectively. Let $\delta_n^f(\cdot)$ be defined in Definition \ref{def:supremum of sampling error of the objective function estimate}. Let $x_{n,\epsilon}(\xi^n_j) \in X_{n,\epsilon}(\xi^n_j)$, $j = 1,2,\ldots,m$. Suppose that Assumptions \ref{assumption:part 1}, \ref{assumption:part 2}, and \ref{assumption:part 2(2)} hold. The following relation holds:
    \begin{equation*}
         \inf_{x \in \bar{X}_{N,\epsilon}(\xi^N)} \left\| \frac{1}{m} \sum_{j=1}^m x_{n,\epsilon}(\xi^n_j) - x \right\| \leq \frac{\frac{1}{m^2} \sum_{\substack{i,j=1 \\ i \neq j}}^m \Big[\delta_{n}^f(\xi^n_j) + \delta^f_{n}(\xi^n_j) \Big]}{ \epsilon} D_X.
    \end{equation*}
\end{lemma}
\proof
We first provide a proof sketch, which consists of two steps. In Step 1, we show that $\frac{1}{m} \sum_{j=1}^m x_{n,\epsilon}(\xi^n_j)$ is an $\epsilon'$-optimal solution of the problem in (\ref{eq:aggregated SAA problem}), where $\epsilon' = \epsilon + \frac{1}{m^2}\sum_{\substack{i,j=1 \\ i \neq j}}^m (\delta_{n}(\xi^n_i) + \delta_{n}(\xi^n_j))$. The key part of the proof is to derive the following relation, 
$$
f_{n} (x_{n,\epsilon}(\xi^n_j);\xi^n_i) \leq \theta_{n}(\xi^n_j) + \epsilon + \delta_{n}(\xi^n_i) + \delta_{n}(\xi^n_j),
$$
where $\theta_{n}(\xi^n_j)$ is defined in (\ref{eq:optimal cost of one replication of SAA}). Note that if we evaluate $f_{n} (x;\xi^n_i)$ at the point $x_{n,\epsilon}(\xi^n_j)$, it should depend on three factors: (1) the minimum of $f_{n} (x;\xi^n_i)$ over $X$; (2) the systematic error $\epsilon$; (3) the sampling errors from $f_{n} (x;\xi^n_i)$ and $f_{n} (x;\xi^n_j)$. In Step 2 of the proof, we can apply Theorem \ref{theorem:lipschitz continuity of solution set}
to derive the upper bound on $\inf_{x \in \bar{X}_{N,\epsilon}(\xi^N)} \left\| \frac{1}{m} \sum_{j=1}^m x_{n,\epsilon}(\xi^n_j) - x \right\|$. \\
\noindent (\textbf{Step 1: To show $\sum_{j=1}^m x_{n,\epsilon}(\xi^n_j)$ is an $\epsilon'$-optimal solution of the problem in (\ref{eq:aggregated SAA problem})}). We let $\bar{x}_N(\xi^N) \in \arg \min_{x \in X} \bar{f}_N(x;\xi^N)$, where $\bar{f}_N(x;\xi^N)$ is defined in (\ref{eq:aggregated SAA objective function}). First, by convexity assumption in Assumption \ref{assumption:part 2}, we observe that 
\begin{equation} \label{eq:lemma:compromise:relation3:01}
    \begin{aligned}
        &\frac{1}{m} \sum_{i=1}^m f_{n} \left(\frac{1}{m} \sum_{j=1}^m x_{n,\epsilon}(\xi^n_j);\xi^n_i \right) \\
        &\leq \frac{1}{m^2} \sum_{i=1}^m \sum_{j=1}^m f_{n}(x_{n,\epsilon}(\xi^n_j);\xi^n_i) \quad \text{by convexity in Assumption \ref{assumption:part 2}} \\
        &= \frac{1}{m^2} \sum_{i=1}^m \sum_{j=1}^m \left[f_{n} (\bar{x}_N(\xi^N);\xi^n_j) + f_{n}^i (x_{n,\epsilon}(\xi^n_j);\xi^n_i) - f_{n} (\bar{x}_N(\xi^N);\xi^n_j) \right].
    \end{aligned}
\end{equation}
Similar to the derivation in (\ref{eq:complexity:cost2}) from Theorem \ref{thm:sample complexity cost of compromise decision problem}, for $i \neq j$, we have
\begin{equation} \label{eq:lemma:compromise:relation3:02}
\begin{aligned}
    f_{n} (x_{n,\epsilon}(\xi^n_j);\xi^n_i) &= f_{n} (x_{n,\epsilon}(\xi^n_j);\xi^n_j) + f_{n}(x_{n,\epsilon}(\xi^n_j);\xi^n_i) - f_{n}(x_{n,\epsilon}(\xi^n_j);\xi^n_j) \\
    &\leq \theta_{n}(\xi^n_j) + \epsilon + |f_{n} (x_{n,\epsilon}(\xi^n_j);\xi^n_i) - f_{n}(x_{n,\epsilon}(\xi^n_j);\xi^n_j)| \\
    &\leq \theta_{n}(\xi^n_j) + \epsilon + \sup_{x \in X} |f_{n} (x;\xi^n_i) - f(x)| + \sup_{x \in X} |f_{n} (x;\xi^n_j) - f(x)| \\
    &= \theta_{n}(\xi^n_j) + \epsilon + \delta^f_{n}(\xi^n_i) + \delta^f_{n}(\xi^n_j) .
\end{aligned}
\end{equation}
The combination of (\ref{eq:lemma:compromise:relation3:01}) and (\ref{eq:lemma:compromise:relation3:02}) implies that 
\begin{equation} \label{eq:epsilon-optimal solution of SAA:proof01}
    \begin{aligned}
        \frac{1}{m} \sum_{i=1}^m f_{n} \left(\frac{1}{m} \sum_{j=1}^m x_{n,\epsilon}(\xi^n_j);\xi^n_i \right) &\leq \frac{1}{m^2} \sum_{i=1}^m \sum_{j=1}^m \Bigg[f_{n}\left(\bar{x}_N(\xi^N);\xi^n_j \right) + \theta_{n}(\xi^n_j) + \epsilon \\
        & \quad + \big(\delta_{n}^f(\xi^n_i) + \delta_{n}^f(\xi^n_j)\big) \mathbb{I}(i \neq j) - f_{n}(\bar{x}_N(\xi^N);\xi^n_j) \Bigg] \\
        &= \theta_N(\xi^N) + \epsilon + \frac{1}{m^2} \sum_{j=1}^m \Big[\theta_{n}(\xi^n_j) - f_{n}\Big(\bar{x}_N(\xi^N);\xi^n_j\Big)\Big] \\
        & \quad + \frac{1}{m^2} \sum_{\substack{i,j=1 \\ i \neq j}}^m \Big[\delta^f_{n}(\xi^n_i) + \delta^f_{n}(\xi^n_j)\Big] \\
        &\leq \theta_N(\xi^N) + \epsilon + \frac{1}{m^2} \sum_{\substack{i,j=1 \\ i \neq j}}^m \Big[\delta^f_{n}(\xi^n_i) + \delta^f_{n}(\xi^n_j)\Big] .
    \end{aligned}
\end{equation}
Hence, equation (\ref{eq:epsilon-optimal solution of SAA:proof01}) shows that $\sum_{j=1}^m x_{n,\epsilon}(\xi^n_j)$ is the $\epsilon'$-optimal solution of the problem $\min_{x \in X} \bar{f}_N(x;\xi^N)$, where $\epsilon' = \epsilon + \frac{1}{m^2} \sum_{\substack{i,j=1 \\ i \neq j}}^m [\delta^f_{n}(\xi^n_i) + \delta^f_{n}(\xi^n_j)]$. \\
\noindent (\textbf{Step 2: To derive upper bound of $ \inf_{x \in \bar{X}_{N,\epsilon}} \| \frac{1}{m} \sum_{j=1}^m x_{n,\epsilon}(\xi^n_j) - x \|$ using Theorem \ref{theorem:lipschitz continuity of solution set}})
Let $\epsilon' = \epsilon + \frac{1}{m^2}\sum_{\substack{i,j=1 \\ i \neq j}}^m (\delta^f_{n}(\xi^n_i) + \delta^f_{n}(\xi^n_j))$, then $\frac{1}{m} \sum_{j=1}^m x_{n,\epsilon}(\xi^n_j) \in \bar{X}_{N,\epsilon'}(\xi^N)$. Hence, by Theorem \ref{theorem:lipschitz continuity of solution set}, we have 
\begin{equation}
\begin{aligned}
    \inf_{x \in \bar{X}_{N,\epsilon}} \| \frac{1}{m} \sum_{j=1}^m x_{n,\epsilon}(\xi^n_j) - x \| &\leq \Delta(\bar{X}_{N,\epsilon'}(\xi^N), \bar{X}_{N,\epsilon}(\xi^N)) \\
    &\leq \frac{\epsilon' - \epsilon}{\epsilon} D_X \leq \frac{\frac{1}{m^2} \sum_{\substack{i,j=1 \\ i \neq j}}^m (\delta^f_{n}(\xi^n(\xi^n_i) + \delta^f_{n}(\xi^n_j))}{ \epsilon} D_X. 
\end{aligned}
\end{equation}
\endproof
Similar to the proof of Lemma \ref{lemma:compromise:relation3}, we will apply Theorem \ref{theorem:lipschitz continuity of solution set} to obtain the upper bound of $\inf\limits_{x \in X^*_{\epsilon}} \| \frac{1}{m} \sum_{j=1}^m x^j_{n,\epsilon} - x \|$.  
\begin{lemma} \label{lemma:compromise:relation4}
    Let $X_{n,\epsilon}(\xi^n_j)$ and $X_{\epsilon}^*$ be defined in (\ref{eq:epsilon-optimal solution set for each replication}) and (\ref{eq:epsilon optimal solution set:true problem}), respectively. Let $x_{n,\epsilon}(\xi^n_j) \in X_{n,\epsilon}(\xi^n_j)$, $j = 1,2,\ldots,m$.  Let $\delta_n^f(\cdot)$ be defined in Definition \ref{def:supremum of sampling error of the objective function estimate}. Suppose that Assumptions \ref{assumption:part 1}, \ref{assumption:part 2}, and \ref{assumption:part 2(2)} hold. Then the following relation holds:
    \begin{equation*}
         \inf_{x \in X^*_{\epsilon}} \| \frac{1}{m} \sum_{j=1}^m x_{n,\epsilon}(\xi^n_j) - x \| \leq \frac{\frac{2}{m} \sum_{j=1}^m \delta^f_{n}(\xi^n_j)}{\epsilon} D_X.
    \end{equation*}
\end{lemma}
% move the proof to the appendix
\proof
The proof strategy is similar to Lemma \ref{lemma:compromise:relation3}. See Appendix \ref{appendix:proof of lemma:compromise:relation4}.
\endproof
With the bounds obtained in Lemmas \ref{lemma:compromise:relation2} - \ref{lemma:compromise:relation4}, we now present the finite-sample complexity of the compromise decisions in the following theorem. 
\begin{theorem} \label{theorem:complexity of epsilon optimal solution}
    Suppose that Assumptions \ref{assumption:part 1}, \ref{assumption:part 2}, and \ref{assumption:part 2(2)} hold. Let $\hat{X}_{N,\rho,\epsilon}(\xi^N)$ be defined in (\ref{notations:inexact compromise decicisions}). Let $\lambda \in (0, \frac{1}{2})$ and $N_F$ be a constant defined in Lemma \ref{lemma:rademacher complexity of objective:holder continuity}. Let $\delta_n^f(\cdot)$ be defined in Definition \ref{def:supremum of sampling error of the objective function estimate}.%$N_F = L_F D^{\gamma} d^{\frac{\gamma}{2}} + M_F \sqrt{2 (\log2)} + \frac{M_F d^{1/2}}{\sqrt{\gamma (1 - 2 \lambda)e}} $. 
    Then the following hold:
    \begin{enumerate}
        \item[1.] \begin{equation} \label{eq:proof:theorem:complexity of epsilon optimal solution:01}
        \Delta(\hat{X}_{N,\rho,\epsilon}(\xi^N), X_{\epsilon}^*) \leq \frac{\frac{2}{m} \sum_{j=1}^m \delta^f_{n}(\xi^n_j) + \frac{1}{m^2} \sum_{\substack{i,j=1 \\ i \neq j}}^m \Big[\delta^f_{n}(\xi^n_i) + \delta^f_{n}(\xi^n_j)\Big]}{\epsilon} D_X + 2 \sqrt{\epsilon / \rho}. 
    \end{equation}
    \item[2.] 
    \begin{equation*}
        \mathbb{E}[\Delta(\hat{X}_{N,\rho,\epsilon}(\tilde{\xi}^N), X_{\epsilon}^*)] \leq \frac{D_X N_F}{\epsilon}\frac{8m - 4}{m n^\lambda} + 2 \sqrt{\epsilon/\rho}.
    \end{equation*}
    \item[3.] Further let $\rho = n$ and $C = \frac{\epsilon^{\frac{3}{2}}}{D_X} + \frac{(4m - 2) N_F}{m}$, then we have
    \begin{equation*}
        \Pr \left\{\frac{\epsilon n^\lambda}{2 D_X} \Delta(\hat{X}_{N,\rho,\epsilon}(\tilde{\xi}^N), X_{\epsilon}^*) \geq C + t \right\} \leq m \exp \left\{-\frac{m^2 t^2}{2M^2_F(2m - 1)^2} \right\} .
    \end{equation*}
    \end{enumerate}
\end{theorem}
\proof
As with some of the previous results, we begin with an overview of the proof. Part 1 (\ref{eq:proof:theorem:complexity of epsilon optimal solution:01}) of the theorem follows by directly combining the results of Lemmas \ref{lemma:compromise:relation2}, \ref{lemma:compromise:relation3}, and \ref{lemma:compromise:relation4}. Hence, we will skip the proof of Part 1. Part 2 of the theorem follows by taking the expectation of both sides of (\ref{eq:proof:theorem:complexity of epsilon optimal solution:01}) and then applying Theorem \ref{theorem:bound of function estimation error}.1. Part 3 of the theorem follows by applying Theorem \ref{theorem:bound of function estimation error}.2 into (\ref{eq:proof:theorem:complexity of epsilon optimal solution:01}). \\
\noindent \textbf{(Proof of Part 2)} By Theorem \ref{theorem:bound of function estimation error}, note that $\mathbb{E}[\delta^f_n(\tilde{\xi}^n_i)] \leq 2R_n(F,\Xi)$ for $i = 1,2,\ldots,m$. By taking the expectation of both sides of (\ref{eq:proof:theorem:complexity of epsilon optimal solution:01}) and by the i.i.d assumption in Assumption \ref{assumption:part 2(2)}, we have 
\begin{equation*}
    \begin{aligned}
        \mathbb{E}[\Delta(\hat{X}_{N,\rho,\epsilon}(\tilde{\xi}^N), X_{\epsilon}^*)] &\leq \frac{4 R_n(F,\Xi) + \frac{4m^2 - 4m}{m^2} R_n(F,\Xi)}{\epsilon} + 2 \sqrt{\epsilon/\rho} \\
        &\leq \frac{D_X N_F}{\epsilon}\frac{8m - 4}{m n^\lambda} + 2 \sqrt{\epsilon/\rho} && \text{by Theorem \ref{theorem:bound of function estimation error}.} 
    \end{aligned}
\end{equation*}
\noindent \textbf{(Proof of Part 3)} Equation (\ref{eq:proof:theorem:complexity of epsilon optimal solution:01}) implies that 
\begin{equation} \label{eq:theorem:complexity of epsilon optimal solution:01}
    \begin{aligned}
        &\Pr \left\{\frac{\epsilon n^\lambda}{2 D_X} \Delta(\hat{X}_{N,\rho,\epsilon}(\tilde{\xi}^N), X_{\epsilon}^*) \geq C + t \right\} \\
        &\leq \Pr \left\{ \frac{n^\lambda}{m} \sum_{j=1}^m \delta^f_{n}(\tilde{\xi}^n_j) + \frac{n^\lambda}{2m^2} \sum_{\substack{i,j=1 \\ i \neq j}}^m \Big[\delta^f_{n}(\tilde{\xi}^n_i) + \delta^f_{n}(\tilde{\xi}^n_j)\Big] + \frac{\epsilon^{\frac{3}{2}} n^\lambda }{D_Xn^\frac{1}{2}} \geq C + t \right\} \\
        &\leq \Pr \left\{ \frac{n^\lambda}{m} \sum_{j=1}^m \delta^f_{n}(\tilde{\xi}^n_j) + \frac{n^\lambda}{2 m^2} \sum_{\substack{i,j=1 \\ i \neq j}}^m \Big[\delta^f_{n}(\tilde{\xi}^n_i) + \delta^f_{n}(\tilde{\xi}^n_j)\Big] + \frac{\epsilon^{\frac{3}{2}}}{D_X} \geq C + t \right\} \ \text{because } \frac{n^\lambda}{n^\frac{1}{2}} \leq 1 \\
        &\leq \Pr \left\{ \sum_{i=1}^m \frac{n^\lambda(2m - 1)}{m^2} \delta^f_n(\tilde{\xi}^n_i) \geq \frac{(4m - 2) N_F}{m} + t\right\} 
    \end{aligned}
\end{equation}
Given two random variables $X_1, X_2$ and constants $\epsilon$, we have $\Pr(X_1 + X_2 \geq \epsilon) \leq \Pr (X_1 \geq \frac{\epsilon}{2}) + \Pr(X_2 \geq \frac{\epsilon}{2})$ (see \cite[equation (5)]{shapiro2006complexity} for a similar argument). Hence it follows from (\ref{eq:theorem:complexity of epsilon optimal solution:01}) that 
\begin{equation}
    \begin{aligned}
        &\Pr \left\{ \sum_{i=1}^m \frac{n^\lambda(2m - 1)}{m^2} \delta^f_n(\tilde{\xi}^n_i) \geq \frac{(4m - 2) N_F}{m} + t\right\} \\
        &\leq \sum_{i=1}^m \Pr \left\{\frac{n^\lambda(2m - 1)}{m^2} \delta^f_n(\tilde{\xi}^n_i) \geq \frac{(4m - 2)  N_F}{m^2} + \frac{t}{m} \right\}\\
        &= \sum_{i=1}^m \Pr \left\{n^\lambda \delta^f_n(\tilde{\xi}^n_i) \geq  2N_F + \frac{m t}{2m - 1} \right\} \\
        &\leq m \exp \left\{-\frac{m^2 t^2}{2M^2_F(2m - 1)^2} \right\} \quad \text{by Theorem \ref{theorem:bound of function estimation error}.2 }. 
    \end{aligned}
\end{equation}
\endproof

Finally, we close this section by deriving the upper bound on the variance of $\Delta(\hat{X}_{N,\rho,\epsilon}(\xi^N), X_{\epsilon}^*)$ (i.e., the largest distance from the $\epsilon$-optimal solution set of the compromise decision problem to the $\epsilon$-optimal solution set of the true problem). 

\begin{theorem} \label{thm:variance:saa}
    Suppose that Assumptions \ref{assumption:part 1}, \ref{assumption:part 2}, and \ref{assumption:part 2(2)} hold. Let $\hat{X}_{N,\rho,\epsilon}(\xi^N)$ and $X^*_{\epsilon}$ be defined in (\ref{notations:inexact compromise decicisions}) and (\ref{eq:epsilon optimal solution set:true problem}), respectively. Let $\lambda \in (0, \frac{1}{2})$ and $N_F$ be a constant defined in Lemma \ref{lemma:rademacher complexity of objective:holder continuity}. %$N_F = L_F D^{\gamma} d^{\frac{\gamma}{2}} + M_F \sqrt{2 (\log2)} + \frac{M_F d^{1/2}}{\sqrt{\gamma (1 - 2 \lambda)e}} $. 
    Then the following holds:
    \begin{equation*}
        \text{Var}[\Delta(\hat{X}_{N,\rho,\epsilon}(\tilde{\xi}^N), X_{\epsilon}^*)] \leq \frac{64 D_X^2 M_F N_F}{m \epsilon^2 n^\lambda} + \left(\frac{8D_X N_F}{\epsilon n^\lambda} + 2 \sqrt{\epsilon / \rho} \right)^2.
    \end{equation*}
    Furthermore, let $K$ be a positive constant. Letting $\rho \geq K n$, we have $\mathbb{E}[\Delta(\hat{X}_{N,\rho,\epsilon}(\tilde{\xi}^N), X_{\epsilon}^*)] = O(\frac{1}{n^{\lambda}})$ and $\text{Var}[\Delta(\hat{X}_{N,\rho,\epsilon}(\tilde{\xi}^N), X_{\epsilon}^*)] = O(\frac{1}{mn^\lambda} + \frac{1}{n^{2\lambda}})$. 
\end{theorem}
\begin{proof}
Here is a proof sketch: we use the relation in Theorem \ref{theorem:complexity of epsilon optimal solution}.1, relation in (\ref{eq:common bound on the variance}), and, bound of variance of estimation error in (\ref{eq:bound of the varaince}) to derive the upper bound of $\text{Var}[\Delta(\hat{X}_{N,\rho,\epsilon}(\tilde{\xi}^N), X_{\epsilon}^*)]$. In particular, we make use of the assumption that the sample set in replication $i$ is independent of the sample set in replication $j$ (given $i \neq j$) to decompose the variance. \\
\indent Next, we shall give the details of the proof. It follows from Theorem \ref{theorem:complexity of epsilon optimal solution} that 
\begin{equation} \label{eq:thm:variance:saa:01}
    \begin{aligned}
        \Delta(\hat{X}_{N,\rho,\epsilon}(\xi^N), X_{\epsilon}^*) &\leq \frac{\frac{2}{m} \sum_{j=1}^m \delta^f_{n}(\xi^n_j) + \frac{1}{m^2} \sum_{\substack{i,j=1 \\ i \neq j}}^m [\delta^f_{n}(\xi^n_i) + \delta^f_{n}(\xi^n_j)]}{\epsilon} D_X + 2 \sqrt{\epsilon / \rho} \\
        &\leq  \sum_{i=1}^m \frac{4 D_X}{m\epsilon} \delta^f_n(\xi_j^n) + 2 \sqrt{\epsilon / \rho}.
    \end{aligned}
\end{equation}
According to (\ref{eq:bound of the varaince}), we have $\text{Var}[\delta^f_n(\tilde{\xi}^n)] \leq \frac{4 M_F N_F}{n^{\lambda}}$. Hence, by i.i.d. assumption in Assumption \ref{assumption:part 2(2)}, we can derive the upper bound of the variance of the right-hand side of \ref{eq:thm:variance:saa:01} as follows:
\begin{equation} \label{eq:thm:variance:saa:02}
    \begin{aligned}
        \text{Var}\left[\sum_{i=1}^m \frac{4 D_X}{m\epsilon} \delta^f_n(\tilde{\xi}_j^n) + 2 \sqrt{\epsilon / \rho} \right] &= \sum_{i=1}^m \frac{16 D_X^2}{m^2\epsilon^2} \text{Var}[\delta^f_n(\tilde{\xi}_j^n)] \\
        &\leq \frac{64 D_X^2 M_F N_F}{m \epsilon^2 n^\lambda}. 
    \end{aligned}
\end{equation}
Using equations (\ref{eq:thm:variance:saa:01}), (\ref{eq:thm:variance:saa:02}), and Theorem \ref{theorem:bound of function estimation error}.1, we can derive the upper bound of the variance of $\Delta(\hat{X}_{N,\rho,\epsilon}(\tilde{\xi}^N), X_{\epsilon}^*)$ as follows:
\begin{equation} \label{eq:thm:variance:saa:03}
    \begin{aligned}
        &\text{Var}\left[\Delta(\hat{X}_{N,\rho,\epsilon}(\tilde{\xi}^N), X_{\epsilon}^*) \right] \\
        &\leq \text{Var}\left[\sum_{i=1}^m \frac{4}{m\epsilon} D_X\delta_n(\tilde{\xi}_j^n) + 2 \sqrt{\epsilon / \rho} \right] + \left(\mathbb{E}\left[\sum_{i=1}^m \frac{4}{m\epsilon} D_X\delta_n(\tilde{\xi}_j^n) + 2 \sqrt{\epsilon / \rho}\right] \right)^2 \\
        &\leq \frac{64 D_X^2 M_F N_F}{m \epsilon^2 n^\lambda} + \left(\frac{8D_X N_F}{\epsilon n^\lambda} + 2 \sqrt{\epsilon / \rho} \right)^2.
    \end{aligned}
\end{equation}
Finally, the big $O$ results follows by plugging $\rho \geq K n$ into Theorem \ref{theorem:complexity of epsilon optimal solution}.2 and (\ref{eq:thm:variance:saa:03}). 
\end{proof}
% remark

% sections for algorithms for compromise SP
\section{Algorithms for Compromise Decision Problems} \label{sec:algorithms compromise sp}
In the previous section, we have discussed the reliability of the compromise decision using SAA in the {\sl replication} step. There are two issues which remain unresolved: (1) How should we solve the SAA problems in the {\sl replication} step? (2) How should we simplify the aggregated objective function estimate without losing reliability?\\
\indent In this section, we will continue studying the case in which a proper algorithm is applied to solve each replication of the SAA problem. We aim to provide a framework to combine the computational objects which are created during a replication of a cutting-plane-type algorithm, and then use such information to formulate an \textsl{algorithm-augmented} compromise decision problem. \\
\indent Given a sample set $\xi^n$, we let $\hat{x}_n(\xi^n)$ and $\hat{f}_n(x;\xi^n)$ denote the estimated solution and approximated objective function of SAA objective function $f_n(x;\xi^n)$ (i.e., $f_n(x;\xi^n) = \frac{1}{n} \sum_{i=1}^n F(x,\xi_i)$) output by the algorithm. 
We assume that the algorithm constructs a piecewise linear outer approximation of the SAA objective function. 
\begin{condition}[Outer Approximation] \label{condition:outer approximation}
    $\hat{f}_n(x;\xi^n)$ is a convex piecewise linear approximation of $f_n(x;\xi^n)$ (i.e., $\hat{f}_n(x;\xi^n) = \max\limits_{\ell \in L} \{\alpha_\ell + \langle \beta_\ell, x \rangle \}$, where $L$ is the index set of the pieces of $\hat{f}_n(x;\xi^n)$), such that $\hat{f}_n(x;\xi^n) \leq f_n(x;\xi^n)$ for all $x \in X$.  
\end{condition}
When the algorithm terminates, we assume that the following condition holds. 
\begin{condition}[Termination Condition] \label{condition:termination criterion} 
 Let $\hat{x}_n(\xi^n) \in \arg\min\limits_{x \in X} \hat{f}_n(x;\xi^n)$. There exists $\epsilon_1 \in (0,\infty)$ such that $f_n\left(\hat{x}_n(\xi^n);\xi^n \right) - \hat{f}_n \left(\hat{x}_n(\xi^n);\xi^n \right) \leq \epsilon_1$. 
\end{condition}
Condition \ref{condition:outer approximation} ensures that it outputs an outer approximation of the SAA of the objective function. Condition \ref{condition:termination criterion} is often used as a stopping criterion for Benders' type decomposition-based algorithms (see \cite[Chapter 6.5]{bertsimas1997introduction} for an example). For instance, Kelley's Cutting Plane Methods (\cite{kelley1960cutting}) and other Benders' type decomposition-based algorithms (\cite{oliveira2011inexact,philpott2008convergence,ruszczynski1986regularized,van1969shaped}) will satisfy the conditions above. Overall, Conditions \ref{condition:outer approximation} and \ref{condition:termination criterion} altogether ensure that an algorithm yields an $\epsilon_1$-optimal solution of the SAA problem. To see why, let $x_n(\xi^n) \in \arg \min\limits_{x \in X} f_n(x;\xi^n)$, $\theta_n(\xi^n) = \min\limits_{x \in X} f_n(x;\xi^n)$, and $\hat{\theta}_n(\xi^n) = \min\limits_{x \in X} \hat{f}_n(x;\xi^n)$. Then by Conditions \ref{condition:outer approximation} and \ref{condition:termination criterion}, we derive the following relation:
 \begin{equation} \label{eq:outer approximation:relation}
     \theta_n(\xi^n) = f_n(x_n(\xi^n);\xi^n) \geq \hat{f}_n\left(x_n(\xi^n);\xi^n \right) \geq \hat{\theta}_n(\xi^n) = \hat{f}_n\left(\hat{x}_n(\xi^n);\xi^n \right).
 \end{equation}
 Hence, by Condition \ref{condition:termination criterion} and inequality in (\ref{eq:outer approximation:relation}), we have 
 \begin{equation*}
 \begin{aligned}
     f_n\left(\hat{x}_n(\xi^n);\xi^n \right) &= \hat{f}_n\left(\hat{x}_n(\xi^n);\xi^n \right) + f_n\left(\hat{x}_n(\xi^n);\xi^n \right) - \hat{f}_n\left(\hat{x}_n(\xi^n);\xi^n \right) \\
     %&\leq \theta_n(\xi^n) + f_n\left(\hat{x}_n(\xi^n);\xi^n \right) - \hat{f}_n\left(\hat{x}_n(\xi^n);\xi^n \right) \\
     &\leq \theta_n(\xi^n) + \epsilon_1 
 \end{aligned}
 \end{equation*}
\indent To answer the second question, we discuss a necessary pre-processing in the {\sl aggregation} step. For $i = 1,2,\ldots,m$, we let $\hat{x}_n(\xi^n_i)$ and $\hat{f}_n(x;\xi^n_i)$ denote the estimated solution of $\min\limits_{x \in X} f_n(x;\xi^n_i)$ and surrogate function of $f_n(x;\xi^n_i)$ (i.e., $f_n(x;\xi^n_i)$ is defined in (\ref{eq:SAA function in each replication})), respectively. Given $\epsilon_2 \geq 0$, in the pre-processing step, we augment $\hat{f}_n(x;\xi^n_i)$ by letting 
\begin{equation} \label{eq:augmented piecewise linear funtcion}
\check{f}_n(x;\xi^n_i) = \max \{\hat{f}_n(x;\xi^n_i), \max_{j=1, \ldots, m} \left\{ f_n(\hat{x}_n^j;\xi^n_i) + \langle \upsilon_{n,\epsilon_2}(\hat{x}_n^j;\xi^n_i), x - \hat{x}^j_n \rangle - \epsilon_2 \} \right\},
\end{equation}
% notation of subgradient
where $\upsilon_{n,\epsilon_2} (\hat{x}_n^j;\xi^n_i)$ is the $\epsilon_2$-subgradient of $f_n(\cdot;\xi^n_i)$ at $\hat{x}_n^j$. This pre-processing controls the deviation of the augmented function from its associated SAA function at the candidate solution from each replication. \\
\indent With the augmented functions and candidate solution from each replication, we can set up the compromise decision program below:
\begin{equation} \label{eq:compromise SP with approximation}
    \min_{x \in X} \frac{1}{m} \sum_{i=1}^m \check{f}_n(x;\xi^n_i) + \frac{\rho}{2} \left\|x - \frac{1}{m} \sum_{j=1}^m \hat{x}_n(\xi^n_j) \right\|^2.
\end{equation}
To distinguish (\ref{eq:compromise SP with approximation}) from previously introduced compromise decision problems in section \ref{sec:classic compromise SP}, we shall refer to the problem (\ref{eq:compromise SP with approximation}) as \textsl{algorithm-augmented} compromise decision problem. \\
% change < to \leq Apr 3, 2024
\indent We let $\hat{\theta}_{N}(\xi^N)$ denote the minimum of the average of the approximation functions across $m$ replications:
\begin{equation} \label{eq:algorithm:min of average of the approximation function}
    \hat{\theta}_{N}(\xi^N) = \min_{x \in X} \frac{1}{m} \sum_{i=1}^m \check{f}_{n}(x;\xi^n_i).
\end{equation}
Given a constant $\epsilon > 0$, we then let $\hat{X}_{N,\epsilon}(\xi^N)$ denote the $\epsilon$-optimal solution set of the problem in (\ref{eq:algorithm:min of average of the approximation function}):
\begin{equation} \label{eq:epsilon-optimal solution of augmented agg problem}
    \hat{X}_{N,\epsilon}(\xi^N) = \left\{x \in X: \frac{1}{m} \sum_{i=1}^m \check{f}_{n}(x;\xi^n_i) \leq \bar{\theta}_{N}(\xi^N) + \epsilon \right\}
\end{equation}
We let $\hat{\theta}_{N,\rho,\epsilon}(\xi^N)$ and denote $\check{X}_{N, \rho, \epsilon}(\xi^N)$ the optimal cost and $\epsilon$-optimal solution set of \textsl{algorithm-augmented} compromise decision problem in (\ref{eq:compromise SP with approximation}):
\begin{equation*}
    \hat{\theta}_{N,\rho,\epsilon}(\xi^N) = \min_{x \in X} \frac{1}{m} \sum_{i=1}^m \check{f}_{n}(x;\xi^n_i) + \frac{\rho}{2} \left\|x - \frac{1}{m} \sum_{i=1}^m \hat{x}_{n}(\xi^n_i) \right\|^2,
\end{equation*}
\begin{equation} \label{eq:epsilon-optimal solution set}
        \check{X}_{N, \rho, \epsilon}(\xi^N) = \left\{x \in X: \frac{1}{m} \sum_{i=1}^m f_{n}(x;\xi^n_i) + \frac{\rho}{2} \left\|x - \frac{1}{m} \sum_{i=1}^m \hat{x}_{n}(\xi^n_i) \right\|^2 \leq \hat{\theta}_{N,\rho,\epsilon}(\xi^N) + \epsilon \right\}. 
\end{equation}
By the same proof technique in Lemma \ref{lemma:compromise:relation2}, we can derive the upper bound of the pessimistic distance of $\check{X}_{N,\rho,\epsilon}(\xi^N)$ to $X_{\epsilon}^*$:
\begin{equation*}
\begin{aligned}
        \Delta(\check{X}_{N,\rho,\epsilon}(\xi^N), X_{\epsilon}^*) &\leq  \inf_{x \in \hat{X}_{N,\epsilon}(\xi^N) } \left\| \frac{1}{m} \sum_{j=1}^m \hat{x}_{n}(\xi^n_j) - x \right\|  \\
        &\quad + \inf_{x \in X_{\epsilon}^*}\left\|\frac{1}{m} \sum_{j=1}^m \hat{x}_{n}(\xi^n_j) - x \right\| + 2 \sqrt{\epsilon / \rho}.
\end{aligned}
\end{equation*}
We derive the upper bounds of $\inf_{x \in X_{\epsilon}^*}\|\frac{1}{m} \sum_{j=1}^m \hat{x}_{n}(\xi^n_j) - x\|$ and $\inf_{x \in \hat{X}_{N,\epsilon}(\xi^N) } \| \frac{1}{m} \sum_{j=1}^m \hat{x}_{n}(\xi^n_j) - x\|$ in the following two lemmas. 
\begin{lemma} \label{lemma:bound 1 of algorithm for compromise sp}
    Let $X^*$ be defined in (\ref{eq:epsilon optimal solution set:true problem}). Let $\hat{x}_{n}(\xi^n_j)$ be the candidate solution in the $j^{\text{th}}$ replication, for $j = 1,2,\ldots,m$. Let $\delta_n^f(\cdot)$ be defined in Definition \ref{def:supremum of sampling error of the objective function estimate}. Suppose that Assumptions \ref{assumption:part 1}, \ref{assumption:part 2}, and \ref{assumption:part 2(2)} hold. Further, suppose that Conditions \ref{condition:outer approximation} and \ref{condition:termination criterion} hold. If $\epsilon \leq \epsilon_1$, then the following relation holds:
    \begin{equation*}
        \inf_{x \in X_{\epsilon}^*}\left\|\frac{1}{m} \sum_{j=1}^m \hat{x}_{n}(\xi^n_j) - x \right\| \leq \frac{\tau_1 + \frac{2}{m} \sum_{j=1}^m \delta^f_{n}(\xi^n_j)}{\epsilon} D_X,%\frac{\epsilon_1 + \epsilon_2 - \epsilon + \frac{2}{m} \sum_{j=1}^m \delta^f_{n}(\xi^n_j)}{\epsilon} D_X,
    \end{equation*}
    where $\tau_1 = \epsilon_1 - \epsilon$.
\end{lemma}
\proof
We aim to show that $\frac{1}{m} \sum_{j=1}^m \hat{x}_{n}(\xi^n_j)$ is an $\epsilon'$-optimal solution of true problem in (\ref{eq:generic sp}), where $\epsilon' = \epsilon_1 + \frac{2}{m} \sum_{j=1}^m \delta^f_{n}(\xi^n_j)$. Next, we can utilize Theorem \ref{theorem:lipschitz continuity of solution set} to finish the proof. See Appendix \ref{appendix:proof of lemma:bound 1 of algorithm for compromise sp} for detailed proof.
\endproof

\begin{lemma} \label{lemma:bound 2 of algorithm for compromise sp}
Let $\hat{X}_{N,\epsilon}(\xi^N)$ be defined in (\ref{eq:epsilon-optimal solution of augmented agg problem}). Let $\hat{x}_{n}(\xi^n_j)$ be the candidate solution in the $j^{\text{th}}$ replication, for $j = 1,2,\ldots,m$. Let $\delta_n^f(\cdot)$ be defined in Definition \ref{def:supremum of sampling error of the objective function estimate}. Suppose that Assumptions \ref{assumption:part 1}, \ref{assumption:part 2}, and \ref{assumption:part 2(2)} hold. Further suppose that Conditions \ref{condition:outer approximation} and \ref{condition:termination criterion} hold. If $\epsilon \leq \epsilon_1 + \frac{m-1}{m} \epsilon_2$, then the following relation holds:
\begin{equation*}
    \inf_{x \in \hat{X}_{N,\epsilon}(\xi^N) } \left\| \frac{1}{m} \sum_{j=1}^m \hat{x}_{n}(\xi^n_j) - x \right\| \leq \frac{\tau_2 + \frac{1}{m^2} \sum_{\substack{i,j=1 \\ i \neq j}}^m \left(\delta^f_{n}(\xi^n_i) + \delta^f_{n}(\xi^n_j) \right)}{ \epsilon} D_X, %\frac{\epsilon_1 + \epsilon_2 - \epsilon + \frac{1}{m^2} \sum_{\substack{i,j=1 \\ i \neq j}}^m \left(2 \epsilon_2 + \delta_{n}(\xi^n_i) + \delta_{n}(\xi^n_j) \right)}{ \epsilon} D_X,
\end{equation*} 
where $\tau_2 = \epsilon_1 + \frac{m-1}{m} \epsilon_2 - \epsilon$.
\end{lemma}
\proof
Here, we aim to show that $\frac{1}{m} \sum_{j=1}^m \hat{x}_{n}(\xi^n_j)$ is the $\epsilon''$-optimal solution of the aggregated problem in (\ref{eq:algorithm:min of average of the approximation function}), where $\epsilon'' = \tau_2 + \frac{1}{m^2} \sum_{\substack{i,j=1 \\ i \neq j}}^m \left(\delta^f_{n}(\xi^n_i) + \delta^f_{n}(\xi^n_j) \right)$. Then we can apply Theorem \ref{theorem:lipschitz continuity of solution set} to derive the upper bound of its associated pessimistic distance. See Appendix \ref{appendix:proof of lemma:bound 2 of algorithm for compromise sp} for detailed proof.
\endproof

We derive the finite-sample complexity of the \textsl{algorithm-augmented} compromise decision problem (\ref{eq:compromise SP with approximation}) in the following theorem.
\begin{theorem} \label{theorem:complexity of compromise sp with piecewise linear approximation}
    Let $\check{X}_{N,\rho,\epsilon}(\tilde{\xi}^N)$ and $X^*_\epsilon$ be defined in (\ref{eq:epsilon-optimal solution set}) and (\ref{eq:epsilon optimal solution set:true problem}), respectively. Suppose that Assumptions \ref{assumption:part 1}, \ref{assumption:part 2}, and \ref{assumption:part 2(2)} hold. Further suppose that Conditions \ref{condition:outer approximation} and \ref{condition:termination criterion} hold. Let $\lambda \in (0, \frac{1}{2})$ and $N_F$ be a constant defined in Lemma \ref{lemma:rademacher complexity of objective:holder continuity}. 
    If $\epsilon \leq \epsilon_1$, then the following hold:
    \begin{itemize}
        \item[1.] \begin{equation*}
        \begin{aligned}
    \mathbb{E}\left[\Delta(\check{X}_{N,\rho,\epsilon}(\tilde{\xi}^N), X_{\epsilon}^*)\right] &\leq \frac{D_X N_F}{\epsilon}\frac{8m - 4}{m n^\lambda} + \frac{(\tau_1 + \tau_2) D_X}{\epsilon} + 2\sqrt{\epsilon/\rho},
        \end{aligned}
\end{equation*}
where $\tau_1 = \epsilon_1 - \epsilon$ and $\tau_2 = \epsilon_1 + \frac{m-1}{m} \epsilon_2 - \epsilon$.
    \item[2.] Furthermore, if $\epsilon_1 = \epsilon$, $\epsilon_2 = 0$, $\rho = n$, and $C = \frac{\epsilon^{\frac{3}{2}}}{D_X} + \frac{(4m - 2) N_F}{m}$, then
    \begin{equation*}
        \Pr \left\{\frac{\epsilon n^\lambda}{2 D_X} \Delta(\check{X}_{N,\rho,\epsilon}(\tilde{\xi}^N), X_{\epsilon}^*) \geq C + t \right\} \leq m \exp \left\{-\frac{m^2 t^2}{2 M^2_F(2m - 1)^2} \right\}.
    \end{equation*}
    \end{itemize}
\end{theorem}
\proof
The proof strategy is similar to the Theorem \ref{theorem:complexity of epsilon optimal solution}. By Lemmas \ref{lemma:bound 1 of algorithm for compromise sp} and \ref{lemma:bound 2 of algorithm for compromise sp}, we have 
\begin{equation} \label{eq:proof of complexity of compromise sp with piecewise linear approximation01}
    \begin{aligned}
        \Delta\left(\check{X}_{N,\rho,\epsilon}(\tilde{\xi}^N), X_{\epsilon}^* \right) %&\leq  \inf_{x \in \hat{X}_{N,\epsilon} } \left\| \frac{1}{m} \sum_{j=1}^m \hat{x}_{n}(\tilde{\xi}^n_j) - x \right\| + \inf_{x \in X_{\epsilon}^*} \left\|\frac{1}{m} \sum_{j=1}^m \hat{x}^j_{n}(\tilde{\xi}^n_j) - x \right\| + 2 \sqrt{\epsilon / \rho} \\
        %&\leq \frac{\epsilon_1 + \epsilon_2 - \epsilon + \frac{2}{m} \sum_{j=1}^m \delta^f_{n}(\tilde{\xi}^n_j)}{\epsilon} D_X \\
        &\leq \frac{\tau_1 + \frac{2}{m} \sum_{j=1}^m \delta^f_{n}(\tilde{\xi}^n_j)}{\epsilon} D_X \\
        & \quad + \frac{\tau_2 + \frac{1}{m^2} \sum_{\substack{i,j=1 \\ i \neq j}}^m \left(\delta^f_{n}(\tilde{\xi}^n_i) + \delta^f_{n}(\tilde{\xi}^n_j) \right)}{ \epsilon} D_X + 2 \sqrt{\epsilon / \rho} 
        %& \quad + \frac{\epsilon_1 + \epsilon_2 - \epsilon + \frac{1}{m^2} \sum_{\substack{i,j=1 \\ i \neq j}}^m \left(2 \epsilon_2 + \delta^f_{n}(\tilde{\xi}^n_i) + \delta^f_{n}(\tilde{\xi}^n_j) \right)}{ \epsilon} D_X + 2 \sqrt{\epsilon / \rho} 
    \end{aligned}
\end{equation}
According to Theorem \ref{theorem:bound of function estimation error}, we have 
    $\mathbb{E}[\delta_n^f(\tilde{\xi}^n_j)] \leq \frac{2 N_F}{n^\lambda}$.
Hence, by taking the expectation of both sides of (\ref{eq:proof of complexity of compromise sp with piecewise linear approximation01}), we have 
\begin{equation*}
    \mathbb{E}\left[\Delta(\check{X}_{N,\rho,\epsilon}(\tilde{\xi}^N), X_{\epsilon}^*)\right] \leq \frac{D_X N_F}{\epsilon}\frac{8m - 4}{m n^\lambda} + \frac{(\tau_1 + \tau_2) D_X}{\epsilon} + 2\sqrt{\epsilon/\rho}.%\frac{D_X N_F}{\epsilon}\frac{8m - 4}{m n^\lambda} + \frac{\epsilon_1 + \epsilon_2 - \epsilon}{\epsilon} 2D_X + \frac{2m (m - 1) \epsilon_2}{m^2 \epsilon} D_X+ 2\sqrt{\epsilon/\rho}.
\end{equation*}
% there is some issue \epsilon should not depend on n 
Pick $\epsilon_1 = \epsilon$, $\epsilon_2 = 0$, and $C = \frac{\epsilon^{\frac{3}{2}}}{D_X} + \frac{(4m - 2) N_F}{m}$. Then $\tau_1 = \tau_2 = 0$. Hence, we have 
\begin{equation*}
    \mathbb{E}\left[\Delta(\check{X}_{N,\rho,\epsilon}(\tilde{\xi}^N), X_{\epsilon}^*)\right] \leq \frac{D_X N_F}{\epsilon}\frac{8m - 4}{m n^\lambda} + 2\sqrt{\epsilon/n}.
\end{equation*}
%Since $\frac{n^\lambda}{\sqrt{n}} \leq 1$, equation (\ref{eq:proof of complexity of compromise sp with piecewise linear approximation01}) also implies that 
%\begin{equation}
%    \begin{aligned}
%        &\Pr \left\{\frac{\epsilon n^\lambda}{2 D_X} \Delta(\hat{X}_{N,\rho,\epsilon}(\tilde{\xi}^N), X_{\epsilon}^*) \geq C + t \right\} \\
%        &\leq \Pr \Bigg\{ \frac{n^\lambda}{m} \sum_{j=1}^m \delta_{n}^f(\tilde{\xi}^n_j) + \frac{n^\lambda}{2m^2} \sum_{\substack{i,j=1 \\ i \neq j}}^m (\delta_{n}^f(\tilde{\xi}^n_i) + \delta_{n}^f(\tilde{\xi}^n_j)) + \frac{\epsilon^{\frac{3}{2}} n^\lambda }{D_X \sqrt{n}}  \geq C + t \Bigg\} \\
%        &\leq \Pr \left\{ \frac{n^\lambda}{m} \sum_{j=1}^m \delta_{n}^f(\tilde{\xi}^n_j) + \frac{n^\lambda}{2m^2} \sum_{\substack{i,j=1 \\ i \neq j}}^m (\delta_{n}^f(\tilde{\xi}^n_i) + \delta_{n}^f(\tilde{\xi}^n_j)) + \frac{\epsilon^{\frac{3}{2}} }{D_X} \geq C + t \right\} \\
%        &\leq \Pr \left\{ \sum_{i=1}^m \frac{n^\lambda(2m - 1)}{m^2} \delta_n^f(\tilde{\xi}^n_i) \geq \frac{(4m - 2) N_F}{m} + t\right\} \ \text{by definition of } C.   
%    \end{aligned}
%\end{equation}
Then rest of the proof is the same as the proof of Theorem \ref{theorem:complexity of epsilon optimal solution}. 
\endproof
To end this section, we obtain the upper bound of the variance of $\Delta(\check{X}_{N,\rho,\epsilon}(\tilde{\xi}^N), X_{\epsilon}^*)$ in the following theorem.
\begin{theorem} \label{thm:variance:decomposition}
    Suppose that Assumptions \ref{assumption:part 1}, \ref{assumption:part 2}, and \ref{assumption:part 2(2)} hold. Further suppose that Conditions \ref{condition:outer approximation} and \ref{condition:termination criterion} hold. Let $\lambda \in (0, \frac{1}{2})$ and $N_F$ be a constant defined in Lemma \ref{lemma:rademacher complexity of objective:holder continuity}. %$N_F = L_F D^{\gamma} d^{\frac{\gamma}{2}} + M_F \sqrt{2 (\log2)} + \frac{M_F d^{1/2}}{\sqrt{\gamma (1 - 2 \lambda)e}} $. 
    If $\epsilon \leq \epsilon_1$, then the following hold:
    \begin{itemize}
        \item[1.] 
    \begin{equation*}
        \text{Var}\left[\Delta(\check{X}_{N,\rho,\epsilon}(\tilde{\xi}^N), X_{\epsilon}^*) \right] \leq \frac{64 D_X^2 M_F N_F}{m \epsilon^2 n^\lambda} + 
        \left(\frac{8D_X N_F}{\epsilon n^\lambda} + \frac{(\tau_1 + \tau_2) D_X}{\epsilon}  + 2\sqrt{\epsilon/\rho} \right)^2,
        %\left(\frac{8D_X N_F}{\epsilon n^\lambda} + \frac{\epsilon_1 + 3\epsilon_2 - \epsilon}{\epsilon} 2D_X + 2\sqrt{\epsilon/\rho} \right)^2,
    \end{equation*}
    where $\tau_1 = \epsilon_1 - \epsilon$ and $\tau_2 = \epsilon_1 + \frac{m-1}{m} \epsilon_2 - \epsilon$.
    \item[2.] Furthermore, let $K_1$ be a positive constant. For $\epsilon_1 = \epsilon$, $\epsilon_2 = 0$, and $\rho \geq K_1 n$, we have 
    $$
    \mathbb{E}[\Delta(\check{X}_{N,\rho,\epsilon}(\tilde{\xi}^N), X_{\epsilon}^*)] = O(\frac{1}{n^\lambda}), \ \text{Var}[\Delta(\check{X}_{N,\rho,\epsilon}(\tilde{\xi}^N), X_{\epsilon}^*)] = O(\frac{1}{mn^\lambda} + \frac{1}{n^{2\lambda}}).
    $$
\end{itemize}
\end{theorem}
\begin{proof}
The proof strategy is similar to Theorem \ref{thm:variance:saa}. It follows from Theorem \ref{theorem:complexity of compromise sp with piecewise linear approximation} that 
\begin{equation*}
    \begin{aligned}
        \Delta(\check{X}_{N,\rho,\epsilon}(\xi^N), X_{\epsilon}^*) \leq \sum_{i=1}^n \frac{4D_X}{m\epsilon} \delta(\xi^n_i) + \frac{(\tau_1 + \tau_2) D_X}{\epsilon}  + 2\sqrt{\epsilon/\rho}. %+ \frac{\epsilon_1 + 3\epsilon_2 - \epsilon}{\epsilon} 2D_X + 2\sqrt{\epsilon/\rho}.
    \end{aligned}
\end{equation*}
Compared with Theorem \ref{thm:variance:saa}, the only difference is that we have an extra term $\frac{(\tau_1 + \tau_2) D_X}{\epsilon}$, which bounds the systematic error in the algorithm. Furthermore, when $\epsilon_1 = \epsilon$, $\epsilon_2 = 0$, the upper bound of $\Delta(\check{X}_{N,\rho,\epsilon}(\xi^N), X_{\epsilon}^*)$ is reduced to the one of Theorem \ref{thm:variance:saa}. 
\end{proof}
% mean variance problem 
Inspired by the above results, we end this section with a mean-variance optimization problem which summarizes the reliability of decisions in terms of statistical estimates of the pessimistic distance to the set of optimal decisions. 
\begin{equation} \label{eq:Mean-Variance}
    \begin{aligned}
        \min_{n, m} \ &\mathbb{E}[\Delta(\check{X}_{N,\rho,\epsilon}(\tilde{\xi}^N), X_{\epsilon}^*)] + \vartheta \text{Var}\left[\Delta(\check{X}_{N,\rho,\epsilon}(\tilde{\xi}^N), X_{\epsilon}^*) \right] \\
        \text{s.t.} \ & (\ref{eq:epsilon optimal solution set:true problem}) \text{ and } (\ref{eq:epsilon-optimal solution set}), a.s.
    \end{aligned}
\end{equation}
Here $\vartheta > 0$ is a parameter which reflects a decision-maker's desire for a choice which is near-optimum, even though it may be difficult (or even impossible) to verify optimality (of the decision) with certainty. Nevertheless, the mean and variance of the pessimistic distance provide a statistical measure of optimality.  The SSN instance \cite{sen1994network} of the SP literature is one such case.

% compromise SP with SD for SQQP
\section{Stochastic Decomposition for Compromise Decision Problem} \label{sec:SD compromise sp}
% previous one, general case, now it is special case with better reliability 
In the previous two sections, we have discussed the reliability of the compromise decisions for general convex stochastic programs. In this section, we focus on deriving the sample complexity of compromise decisions for two-stage stochastic quadratic programs with quadratic programming recourse (SQQP) \cite{liu2020asymptotic}. In particular, we consider that SD Algorithm is used to (approximately) solve the SQQP problem in the {\sl replication step}. We will later show that the associated compromise decision has higher reliability under certain assumptions. \\
\indent First, we formulate the SQQP problem as follows:
\begin{equation} \label{eq:SQQP}
\begin{aligned}
    \min \ &f(x) = \frac{1}{2} x^\top Q x + c^\top x + \mathbb{E}[h(x,\tilde{\xi})] \\
    \text{s.t.} \ &A x \leq b,
\end{aligned}
\end{equation}
where $b$ and $c$ are $n_1$ dimensional vectors, $A$ is a $m_1 \times n_1$ matrix, $Q$ is a $n_1 \times n_1$ symmetric and positive definite matrix, and $h(x,\xi)$ is the minimum cost of the following second-stage problem:
\begin{subequations} \label{eq:SQQP:subproblem}
\begin{align}
    h(x,\xi) = \min \ &\frac{1}{2} y^\top P y + d^\top y \\
    \text{s.t.} \ &D y = e(\xi) - C(\xi) x, \quad [\lambda], \label{eq:SQQP:subproblem:con1}\\
    & y \geq 0, \quad [\gamma]. \label{eq:SQQP:subproblem:con2}
\end{align}
\end{subequations}
Here, $P$ is a $n_2 \times n_2$ symmetric and positive definite matrix, $d$ is a $n_2$-dimensional vector, $D$ is a $m_2 \times n_2$ matrix, $e(\xi)$ a $m_2$-dimensional random vector that depends on $\xi$, and $C(\xi)$ is a $m_2 \times n_1$ random matrix that also depends on $\xi$. 
We let $\lambda$ and $\gamma$ denote the dual multipliers of constraints (\ref{eq:SQQP:subproblem:con1}) and (\ref{eq:SQQP:subproblem:con2}). \\ 
\indent According to \cite{liu2020asymptotic}, we make the following assumptions: 
\begin{assumption} \label{assumption:part 3}
Let $X$ and $h(x,\xi)$ denote the first-stage feasible region and second-stage recourse function, respectively. 
    \begin{enumerate}
        \item $Q$ and $P$ are symmetric and positive definite matrices. 
        \item The first-stage feasible region $X = \{x \in \mathbb{R}^{n_1}: Ax \leq b\}$ is compact. The set of possible realizations $\Xi$ is compact. 
        \item The second-stage problem satisfies relatively complete recourse condition.  
        \item The second-stage recourse function is uniformly bounded below by 0. 
        \item Linear independence constraint qualification holds at the optimal solution $x^*$ of (\ref{eq:SQQP}). 
        \item There exists $L_h \in (0, \infty)$ such that $|h(x_1,\xi) - h(x_2,\xi)| \leq L_h \|x_1 - x_2 \|$ for almost every $\xi \in \Xi$. 
        \item There exists a neighborhood $B(x^*,\delta)$ with $\delta \in (0, \infty)$, such that $h(x,\xi)$ is differentiable for all $x \in X$ and for almost every $\xi \in \Xi$. 
        \item The unique optimal solution $x^*$ is sharp. 
        \item Strict complementarity holds at $x^*$.
    \end{enumerate}
\end{assumption}
Under the assumptions above, the optimal solution to the SQQP problem (\ref{eq:SQQP}) is unique. Let $x^*$ denoe the optimal solution to (\ref{eq:SQQP}) and let $X^* = \{x^*\}$. For the discussion about the validation of those assumptions, we refer the readers to \cite{liu2020asymptotic}.

% include a sketch of SD
%
A sketch of the Stochastic Decomposition algorithm for SQQP is provided below. In Algorithm \ref{alg:SD:SQQP}, minorant generated in the $k^{\text{th}}$ iteration is an inexact lower bound approximation of $\frac{1}{k} \sum_{i=1}^k h(x, \xi_i)$. SD stores previously visited faces of the polyhedral set to accelerate the minorant construction. SD also rescales the previously generated minorants to adapt to the incremental sampling scheme. The incumbent selection rule in SD is used to maintain a sequence of incumbents with certain probabilistic improvement in terms of minimizing the true objective function. Please see \cite[Algorithm 3.3]{liu2020asymptotic} for detailed descriptions of the minorant construction and incumbent selection rule.
\begin{algorithm}[!htbp]
\caption{Stochastic Decomposition for SQQP} \label{alg:SD:SQQP}
\begin{algorithmic}
    \State Initialize $\{\alpha_k\}_k$, $x_0 \in X$, $h_0(x) = 0$, $k = 0$, and $\mathcal{J}_0 = \{0\}$. 
    \State Set $\hat{f}_0(x) = \frac{1}{2} x^\top Q x + c^\top x + \max \{h^k_j(x), j \in \mathcal{J}_0 \}$.
    \While{Stopping criterion is not satisfied}
        \State Let $k \leftarrow k + 1$.
        \State Compute $x_k = \arg \min\limits_{x \in X} \hat{f}_{k-1}(x) + \frac{1}{2 \alpha_k} \|x - \hat{x}^{k-1}\|$. 
        \State Construct $\mathcal{J}_k$ to store active minorants at $x_k$. 
        \State Sample $\xi_k \sim \mathbb{P}_{\tilde{\xi}}$.
        \State Compute dual optimal solutions $(\lambda_k, \gamma_k)$ of subproblem (\ref{eq:SQQP:subproblem}) at $(x_k, \xi_k)$. 
        \State Store visited dual faces for accelerating computing future lower bounds. 
        \State \textsl{Construct minorants}, $h^k_k(x_k,\xi_k)$ and $h^k_{-k}(x_k,\xi_k)$, of $\frac{1}{k} \sum_{i=1}^k h(x, \xi_i)$ at $\hat{x}^{k-1}$ and 
        \State $x^k$, respectively.
        \For{$j \in \mathcal{J}_k$} 
            \State Update minorant $h^k_j(x) \leftarrow \frac{|j|}{k} h^{|j|}_j(x)$
        \EndFor
        \State Update $\mathcal{J}_k \leftarrow \mathcal{J}_k \cup \{-k,k\}$. 
        \State Update $\hat{x}^k$ according to the \textsl{incumbent selection rule}.  
    \EndWhile
\end{algorithmic}
\end{algorithm}

\indent The following theorem due to \cite{liu2020asymptotic} derives the convergence rate of SD algorithm for solving SQQP problems. 
\begin{theorem} \label{thm:convergence rate of SD}
    Suppose that Assumption \ref{assumption:part 3} holds and samples used in SD are i.i.d random variables following the distribution of $\tilde{\xi}$. Let $\{\hat{x}_n\}_n$ denote the sequence of incumbent solutions generated by SD. Choose the initial stepsize $\tau > 0$ so that $\tau \theta_{\min}(Q) > 1$, where $\theta_{\min}(Q)$ is the smallest eigenvalue of $Q$. Set the sequence of step sizes of SD to be $\frac{\tau}{n+1}$. Then there exists a constant $K > 0$ such that for large enough $n$, we have 
    \begin{equation*}
    \mathbb{E}\left[\|\hat{x}_n - x^*\| \right] \leq \frac{K}{n}.
    \end{equation*}
\end{theorem}
\begin{proof}
    See \cite[Theorem 4.9]{liu2020asymptotic}. 
\end{proof}
We leave a detailed explanation of the constant $K$ in Appendix \ref{appendix sqqp}. In short, the constant $K$ depends on the Lipschitz modulus of the objective function as well as the Hessians of the first-stage objective and the quadratic pieces of the second-stage recourse function. \\
\indent By Markov's inequality, given $t \in (0, \infty)$, we can further get the following probabilistic bound from Theorem \ref{thm:convergence rate of SD}:
\begin{equation} \label{eq:sd:markov inequality}
    \Pr \left\{\|\hat{x}_n - x^*\| \geq t \right\} \leq \frac{\mathbb{E}[\|\hat{x}_n - x^*\|]}{t} \leq \frac{K}{tn}.
\end{equation}
We let $\hat{x}^j_n$ denote the incumbent decision generated by SD in the $j^\text{th}$ replication and let $\hat{f}^j_n$ denote the local approximation of $f_n^j$ also generated by SD. Let $L_f = \max\limits_{x \in X} \|Qx\| + \|c\| + L_h$. Then both $f(x)$ and $\{\hat{f}^i_n\}_i$ are Lipschitiz continuous with a constant $L_f$. Hence, we observe that with probability at least $(1 - \frac{K}{tn})$, 
\begin{equation} \label{eq:SD:inequality}
    f(\hat{x}_n) - f(x^*) \leq L_f \|\hat{x}_n - x^*\|.
\end{equation}
Combining (\ref{eq:sd:markov inequality}) and (\ref{eq:SD:inequality}), we observed that $\hat{x}_n$ is a $L_f t$-optimal solution of problem (\ref{eq:SQQP}) with probability at least $(1 - \frac{K}{tn})$.

\indent Compared with cutting-plane-type methods mentioned in section \ref{sec:algorithms compromise sp}, which creates a global approximation of the SAA objective function in each replication, SD only creates a local approximation of the counterpart. As a result, the function augmentation approach proposed in Section \ref{sec:algorithms compromise sp} is no longer valid here. Alternatively, we augment the SD's function approximation by including a lower bound estimate of the optimal cost of (\ref{eq:SQQP}). The following explains how we augment the function estimate by SD. First of all, we pick a tolerance constant $\epsilon' > 0$. Then for $i = 1,2\ldots,m$, we augment $\hat{f}^i_n$ by letting 
\begin{equation} \label{eq:SD augmented function}
    \check{f}^i_n(x) = \max \{\hat{f}^i_n(x), \hat{f}^i_n(\hat{x}^i_n) - \epsilon' \},
\end{equation}
where $\hat{f}^i_n(\hat{x}^i_n) - \epsilon'$ provides the estimate of the lower bound of the optimal cost while ensuring the stability of the later aggregation of the function estimates. \\
\indent The augmentation of the function estimate ensures that the incumbent solution from the $i^{\text{th}}$ replication, $\hat{x}^i_n$, is the $\epsilon$-optimal solution of the problem $\min_{x \in X} \check{f}^i_n(x)$, which further helps bound the later-introduced pessimistic distance. Note that
\begin{equation*}
    \check{f}^i_n(x) \geq \hat{f}^i_n(\hat{x}^i_n) - \epsilon', \ \forall x \in X, 
\end{equation*}
and 
\begin{equation*}
    \check{f}^i_n(\hat{x}^i_n) = \hat{f}^i_n(\hat{x}^i_n).
\end{equation*}
Hence, we have
\begin{equation} \label{eq:sd:lower bound}
    \check{f}^i_n(\hat{x}^i_n) - \min_{x \in X} \check{f}^i_n(x) \leq \epsilon'.
\end{equation}
\indent We introduce the following condition on the estimated solution of (\ref{eq:SQQP}) as an intermediate piece to derive the probabilistic bound of the reliability of the compromise decision. At the end of this section, we will derive the probability that such an intermediate condition holds. 
\begin{condition}[Intermediate Condition] \label{condition:intermediate condition}
    Given $t > 0$, we have $\|\hat{x}^j_n - x^*\| \leq t, \ \forall j \in \{ 1, \ldots, m \}$, where $x^*$ is the unique optimal solution of problem (\ref{eq:SQQP}).
\end{condition}
\indent With the incumbent solution and augmented location function approximation of each replication introduced, we formulate the \textsl{SD-augmented} compromise problem below:
\begin{equation} \label{eq:compromise SD}
    \min_{x \in X} \frac{1}{m} \sum_{i=1}^m \check{f}^i_n(x) + \frac{\rho}{2} \left\|\frac{1}{m} \sum_{j=1}^m \hat{x}_n^i- x \right\|^2.
\end{equation}
We let $\check{X}_{N,\rho,\epsilon}$ denote the $\epsilon$-optimal solution set of the \textsl{SD-augmented} compromise decision problem (\ref{eq:compromise SD}). 
The aggregated problem without the regularizer is written below:
\begin{equation} \label{eq:aggregated SD}
    \min_{x \in X} \frac{1}{m} \sum_{i=1}^m \check{f}^i_n(x).
\end{equation}
We let $\check{X}_{N,\epsilon}$ denote the  $\epsilon$-optimal solution set of the problem (\ref{eq:aggregated SD}). \\
\indent We start with deriving the upper bound of the pessimistic distance of $\check{X}_{N,\rho,\epsilon}$ to $X^*$. 
\begin{lemma} \label{lemma:bound 0 of SD for compromise SD}
The following relation holds:
    \begin{equation*}
        \Delta(\check{X}_{N,\rho,\epsilon}, X^*) \leq \left\|\frac{1}{m} \sum_{i=1}^m \hat{x}_n^i - x^* \right\| + 2 \sqrt{\epsilon / \rho} +  \inf_{x \in \check{X}_{N,\epsilon}} \left\| \frac{1}{m} \sum_{j=1}^m \hat{x}_n^i- x \right\|
    \end{equation*}
\end{lemma}
\begin{proof}
    Since $X^* = \{x^*\}$, we have $\Delta(\check{X}_{N,\rho,\epsilon}, X^*) = \sup_{x \in \check{X}_{N,\rho,\epsilon}} \|x - x^*\|$. 
    By following the same proof technique of Lemma \ref{lemma:compromise:relation} (see Appendix \ref{appendix:proof of lemma:compromise:relation}), we can get 
    \begin{equation} \label{eq:bound 0 of SD for compromise SD:01}
        \left\|\check{x}_{N,\rho,\epsilon} - \frac{1}{m} \sum_{i=1}^m \hat{x}_n^i \right\| \leq \left\|\check{x}_{N,\epsilon} - \frac{1}{m} \sum_{i=1}^m \hat{x}_n^i \right\| + 2 \sqrt{\epsilon / \rho}.
    \end{equation}
    Hence, inequality (\ref{eq:bound 0 of SD for compromise SD:01}) implies that 
    \begin{equation} \label{eq:bound 0 of SD for compromise SD:02}
        \left\|\check{x}_{N,\rho,\epsilon} - x^* \right\| \leq \left\|\check{x}_{N,\epsilon} - \frac{1}{m} \sum_{i=1}^m \hat{x}_n^i \right\| + \left\|\frac{1}{m} \sum_{i=1}^m \hat{x}_n^i - x^* \right\| + 2 \sqrt{\epsilon / \rho}
    \end{equation}
    We take the infimum of the right-hand side of (\ref{eq:bound 0 of SD for compromise SD:02}) over $\check{x}_{N,\epsilon} \in \check{X}_{N,\epsilon}$, and get 
    \begin{equation} \label{eq:bound 0 of SD for compromise SD:03}
        \|\check{x}_{N,\rho,\epsilon} - x^*\| \leq \left\|\frac{1}{m} \sum_{i=1}^m \hat{x}_n^i - x^* \right\| + 2 \sqrt{\epsilon / \rho} + \inf_{x \in \check{x}_{N,\epsilon}} \left\|\frac{1}{m} \sum_{i=1}^m \hat{x}_n^i - x \right\| 
    \end{equation}
    The final results follow by taking the supremum of the left-hand-side of (\ref{eq:bound 0 of SD for compromise SD:03}) over $\check{x}_{N,\rho,\epsilon} \in \check{X}_{N,\rho,\epsilon}$. 
\end{proof}
By the construction of augmented local function approximation $\check{f}_n^i$ in (\ref{eq:SD augmented function}), we can derive the upper bound of $\inf_{x \in \check{X}_{N,\epsilon}} \left\| \frac{1}{m} \sum_{i=1}^m \hat{x}^i_{n} - x \right\|$.
\begin{lemma} \label{lemma:bound 2 of SD for compromise sp}
    In the $i^{\text{th}}$ iteration ($i \in \{1,2,\ldots,m\}$), let $\hat{x}^i_n$ be the incumbent generated by the SD (Algorithm \ref{alg:SD:SQQP}). %and let $\check{f}^i_n$ be defined in (\ref{eq:SD augmented function}). %Let $\check{\theta} = \min_{x \in X} \frac{1}{m} \sum_{i=1}^m \check{f}_n^i(x)$, and $\check{X}_{N,\epsilon} = \{x \in X: \frac{1}{m} \sum_{i=1}^m \check{f}_n^i(x) \leq  \check{\theta} + \epsilon\}$.
    Suppose that Condition \ref{condition:intermediate condition} holds. If $\epsilon \leq 2 L_f t + \epsilon'$, then the following relation holds:
    \begin{equation*}
     \inf_{x \in \check{X}_{N,\epsilon}} \left\| \frac{1}{m} \sum_{i=1}^m \hat{x}^i_{n} - x \right\| \leq \frac{2 L_f t + \epsilon' - \epsilon}{\epsilon} D_X
    \end{equation*} 
\end{lemma}
\proof
This proof consists of two steps. In the first step, we will show that $\frac{1}{m} \sum_{i=1}^m \hat{x}^i_n$ is a $(2 L_f t + \epsilon')$-optimal solution of the aggregated problem in (\ref{eq:aggregated SD}). In the second step, we can use Theorem \ref{theorem:lipschitz continuity of solution set} to obtain the upper bound of its associated pessimistic distance. \\
\indent Here is the detailed proof. By Condition \ref{condition:intermediate condition}, for $j \in \{1,\ldots,m\}$, we have 
\begin{equation*}
\begin{aligned}
    \left\|\frac{1}{m} \sum_{i=1}^m \hat{x}^i_n - \hat{x}^j_n \right\| &\leq \frac{1}{m}\sum_{i=1}^m \left\|\hat{x}^i_n - \hat{x}^j_n \right\| \\
    &\leq \frac{1}{m}\sum_{i=1}^m \left(\|\hat{x}^i_n - x^*\| + \|\hat{x}^j_n - x^*\| \right) \\
    &\leq 2 t,
\end{aligned}
\end{equation*}
which implies that 
\begin{equation} \label{eq:lemma:bound 2 of algorithm for compromise sp00}
    \left|\check{f}_n^j(\frac{1}{m} \sum_{i=1}^m \hat{x}^i_n) - \check{f}_n^j(\hat{x}^j_n) \right| \leq 2 L_f t.
\end{equation}
Let $y \in \arg \min\limits_{x \in X} \frac{1}{m} \sum_{i=1}^m \check{f}^k_i(x)$ and let $\check{\theta}_N = \min_{x \in X} \frac{1}{m} \sum_{i=1}^m \check{f}_n^i(x)$, then we have 
\begin{equation} \label{eq:lemma:bound 2 of algorithm for compromise sp01}
    \begin{aligned}
        \frac{1}{m} \sum_{i=1}^m \check{f}_n^i \left(\frac{1}{m} \sum_{j=1}^m \hat{x}^j_n \right) &\leq \frac{1}{m} \sum_{i=1}^m \check{f}_n^i(\hat{x}_n^i) + 2 L_f t \quad \text{ by (\ref{eq:lemma:bound 2 of algorithm for compromise sp00})}\\
        &\leq \frac{1}{m} \sum_{i=1}^m \left[\check{f}_n^i(\hat{x}_n^i) - \check{f}_n^i(y) + \check{f}_n^i(y) \right] + 2 L_f t \\
        %&\leq \check{\theta}_N + 2 L_f t + \frac{1}{m} \sum_{i=1}^m [\check{f}_n^i(x_n^i) - \check{f}_n^i(y)] \\
        &\leq \check{\theta}_N + 2 L_f t + \epsilon' \quad \text{ by (\ref{eq:sd:lower bound})}.
    \end{aligned}
\end{equation}
The inequality above implies that $\frac{1}{m} \sum_{j=1}^m \hat{x}^j_n$ is the $(2 L_f t + \epsilon')$-optimal solution of the problem 
$$
\min\limits_{x \in X} \frac{1}{m} \sum_{i=1}^m \check{f}_n^i(x_n^i).
$$
Then the final result follows by Theorem \ref{theorem:lipschitz continuity of solution set}.%, we have 
%\begin{equation*}
%    \inf_{x \in \check{X}_{N,\epsilon}} \left\| \frac{1}{m} \sum_{j=1}^m \hat{x}^j_{n} - x \right\| \leq \frac{2 L_f t + \epsilon' - \epsilon}{\epsilon} D_X.
%\end{equation*}
\endproof
In the following theorem below, we aim to derive the lower bound of probability that Condition \ref{condition:intermediate condition} holds and provide the probabilistic bound of $\Delta(\check{X}_{N,\rho,\epsilon}, X^*) \leq \frac{2 L_f t  + \epsilon' - \epsilon}{\epsilon} D_X + t + 2 \sqrt{\epsilon / \rho}$. 
\begin{theorem} \label{thm:rate of convergence in compromise SD}
    Suppose that Assumptions \ref{assumption:part 2(2)} and \ref{assumption:part 3} hold. If $\epsilon \leq 2L_f t + \epsilon'$, then the following holds: 
    \begin{equation*}
        \Pr \left\{ \Delta(\check{X}_{N,\rho,\epsilon}, X^*) \leq \frac{2 L_f t  + \epsilon' - \epsilon}{\epsilon} D_X + t + 2 \sqrt{\epsilon / \rho} \right\} \geq \exp \left( - \frac{\frac{mK}{tn}}{1 - \frac{K}{tn}}\right)
    \end{equation*}
\end{theorem}
\begin{proof}
First, we note that $1 - y \geq \exp(-\frac{y}{1 - y})$ for $0\leq y <1$. By (\ref{eq:sd:markov inequality}), for $n > \frac{K}{t}$, we first note that 
\begin{equation} \label{eq:thm:rate of convergence in compromise SD:01}
\begin{aligned}
    \Pr \left\{ \|\hat{x}_n^i - x^*\| < t, \ \forall i \in \{1,\ldots,m\} \right\} &\geq \left(1 - \frac{K}{tn} \right)^m  \\
    &\geq \exp \left( - \frac{\frac{mK}{tn}}{1 - \frac{K}{tn}}\right)
\end{aligned}
\end{equation}
Given $ \|\hat{x}_n^i - x^*\| < t, \ \forall i \in \{1,\ldots,m\}$, it follows from Lemmas \ref{lemma:bound 0 of SD for compromise SD} and \ref{lemma:bound 2 of SD for compromise sp} that, 
\begin{equation} \label{eq:thm:rate of convergence in compromise SD:02}
    \begin{aligned}
        \Delta(\check{X}_{N,\rho,\epsilon}, X^*) &\leq \inf_{x \in \check{X}_{N,\epsilon}} \left\| \frac{1}{m} \sum_{j=1}^m \hat{x}^j_{n} - x \right\| + \left\|\frac{1}{m} \sum_{j=1}^m \hat{x}^j_{n} - x^* \right\|+ 2 \sqrt{\epsilon/\rho} \\
        &\leq \frac{2 L_f t  + \epsilon' - \epsilon}{\epsilon} D_X + t + 2 \sqrt{\epsilon/\rho} 
    \end{aligned}
\end{equation}
Thus, the final result follows from combining (\ref{eq:thm:rate of convergence in compromise SD:01}) and (\ref{eq:thm:rate of convergence in compromise SD:02}).
\end{proof}
\textbf{Remark:} We observe that there exists some freedom to choose parameters in Theorem \ref{thm:rate of convergence in compromise SD}. We now let $t \in (0,1)$, $\alpha \in (0,3)$, $\epsilon' = L_f t$, $\rho = \frac{(3 - \alpha) L_f}{t}$,  $\epsilon = (3 - \alpha) L_f t$.  Then it follows from Theorem \ref{thm:rate of convergence in compromise SD} that 
\begin{equation*}
    \Pr \left\{ \Delta(\check{X}_{N,\rho,\epsilon}, X^*) \leq \frac{\alpha}{3 - \alpha} D_X + 3 t \right\} \geq \exp \left( - \frac{\frac{mK}{tn}}{1 - \frac{K}{tn}}\right),
\end{equation*}
 where $\exp \left( - \frac{\frac{mK}{tn}}{1 - \frac{K}{tn}}\right)$ goes to 1 as $n \rightarrow \infty$. 
 % need another remark on \epsilon' = 0

 In the next result, we aim to study the compromise decision's impact on variance reduction. We let $d_n^i = \|\hat{x}^i_n - x^*\|$, for $i = 1,2,\ldots,m$. We start with bounding $\inf_{x \in \check{X}_{N,\epsilon}} \left\| \frac{1}{m} \sum_{j=1}^m \hat{x}^j_{n} - x \right\|$ using $d_n^i$ in the lemma below. 
\begin{lemma} \label{lemma:bound 3 of SD for compromise sp}
    Suppose that Assumptions \ref{assumption:part 2(2)} and \ref{assumption:part 3} hold. If $\epsilon \leq \epsilon'$, then the following relation holds:
    \begin{equation*}
     \inf_{x \in \check{X}_{N,\epsilon}} \left\| \frac{1}{m} \sum_{j=1}^m \hat{x}^j_{n} - x \right\| \leq \frac{\epsilon' - \epsilon + \frac{2 L_f}{m} \sum_{i=1}^m d_n^i}{\epsilon} D_X
    \end{equation*} 
\end{lemma}
\begin{proof}
The proof is similar to Lemma \ref{lemma:bound 2 of SD for compromise sp}. We show that $\frac{1}{m} \sum_{j=1}^m \hat{x}^j_{n}$ is $(\epsilon' + \frac{2 L_f}{m} \sum_{i=1}^m d_n^i)$-optimal solution of the problem in (\ref{eq:aggregated SD}) and then use Theorem \ref{theorem:lipschitz continuity of solution set} to finish the proof. See Appendix \ref{appendix:proof of lemma:bound 3 of SD for compromise sp} for the complete proof. 
\end{proof}
Finally, we conclude this section by providing the upper bounds on the expectation and variance of the pessimistic distance of $\check{X}_{N,\rho,\epsilon}$ to $X^*$, $\Delta(\check{X}_{N,\rho,\epsilon}, X^*)$.
\begin{theorem} \label{thm:expectation and variance of SD for compromise sp}
    Suppose that Assumptions \ref{assumption:part 2(2)} and \ref{assumption:part 3} holds. Let $\tau_3 = \frac{\epsilon' - \epsilon}{\epsilon} D_X+ 2\sqrt{\epsilon/\rho}$. If $\epsilon \leq \epsilon'$, then the following hold:
    \begin{itemize}
        \item[1.] 
    \begin{equation*}
        \mathbb{E}\left[\Delta(\check{X}_{N,\rho,\epsilon}, X^*)\right] \leq \left(1 + \frac{2 L_f D_X}{\epsilon} \right) \frac{K}{n} + \tau_3,
    \end{equation*}
    \item[2.] 
    \begin{equation*}
    \begin{aligned}
        & \text{Var}\left[\Delta(\check{X}_{N,\rho,\epsilon}, X^*)\right] \leq \left(1 + \frac{2 L_f D_X}{\epsilon} \right)^2 \frac{K D_X}{mn} + \left[\left(1 + \frac{2 L_f D_X}{\epsilon} \right) \frac{K}{n} + \tau_3 \right]^2.
    \end{aligned}
    \end{equation*}
    \item[3.] Furthermore, let $K_1$ be a positive constant. Let $\epsilon = \epsilon'$ and $\rho \geq K_1 n^2$. Then we have 
    $$
    \mathbb{E}\left[\Delta(\check{X}_{N,\rho,\epsilon}, X^*)\right] = O(\frac{1}{n}), \ \text{Var}\left[\Delta(\check{X}_{N,\rho,\epsilon}, X^*)\right] = O(\frac{1}{mn} + \frac{1}{n^2}). 
    $$
\end{itemize}
\end{theorem}
\begin{proof}
Generally speaking, the proof strategy is similar to Theorem \ref{thm:variance:saa}. At first, we show that the upper bound of the variance of the pessimistic distance, $\Delta(\check{X}_{N,\rho,\epsilon}, X^*)$, can be expressed as a linear combination of $d_n^i$ plus some systematic error. Then we can use the fact that $d_n^i$ is independent of $d_{n}^j$ given $i \neq j$ by Assumption \ref{assumption:part 2(2)} to decompose the upper bound of the variance of the pessimistic distance. Finally, we make use of the compactness of the feasible region and the convergence rate of the first moment of $d_n^i$ (by Theorem \ref{thm:convergence rate of SD}) to finish the proof.\\
\indent Let $\tau_3 = \frac{\epsilon' - \epsilon}{\epsilon} D_X+ 2\sqrt{\epsilon/\rho}$. By Lemma \ref{lemma:bound 3 of SD for compromise sp}, it follows from Lemma \ref{lemma:bound 0 of SD for compromise SD} that, 
\begin{equation} \label{eq:bound 0 of SD for compromise SD:sto1}
\begin{aligned}
    \Delta(\check{X}_{N,\rho,\epsilon}, X^*) &\leq \frac{1}{m} \sum_{i=1}^m d_n^i + \frac{\epsilon' - \epsilon + \frac{2 L_f}{m} \sum_{i=1}^m d_n^i}{\epsilon} D_X + 2 \sqrt{\epsilon / \rho} \\
    &= \frac{1}{m} \sum_{i=1}^m \left(1 + \frac{2 L_f D_X}{\epsilon} \right) d_n^i + \tau_3.
    %&= \frac{1}{m} \sum_{i=1}^m \left(1 + \frac{2 L_f D_X}{\epsilon} \right) d_n^i + \frac{\epsilon' - \epsilon}{\epsilon} D_X + 2\sqrt{\epsilon/\rho}.
\end{aligned}   
 \end{equation}

Furthermore, since $d_n^1, d_n^2, \ldots, d_n^m$ are independent of each other by Assumption \ref{assumption:part 2(2)} and $\tau_3$ %$\frac{\epsilon' - \epsilon}{\epsilon} + \sqrt{\epsilon/\rho}$ 
is a constant, the variance of the right-hand-side of (\ref{eq:bound 0 of SD for compromise SD:sto1}) is 
\begin{equation} \label{eq:variance of compromise SD:01}
    \begin{aligned}
        \text{Var}\left[ \frac{1}{m} \sum_{i=1}^m (1 + \frac{2 L_f D_X}{\epsilon}) d_n^i + \tau_3 \right] = \sum_{i=1}^m \frac{1}{m^2} \left(1 + \frac{2 L_f D_X}{\epsilon} \right)^2 \text{Var}[d_n^i]. 
        %\text{Var}\left[ \frac{1}{m} \sum_{i=1}^m (1 + \frac{2 L_f D_X}{\epsilon}) d_n^i + \frac{\epsilon' - \epsilon}{\epsilon} D_X + \sqrt{\epsilon/\rho} \right] = \sum_{i=1}^m \frac{1}{m^2} \left(1 + \frac{2 L_f D_X}{\epsilon} \right)^2 \text{Var}[d_n^i]. 
    \end{aligned}
\end{equation}
Since $0 \leq d_n^i \leq D_X$, we have $\mathbb{E}[(d_n^i)^2] \leq \mathbb{E}[d_n^i D_X] \leq \frac{K D_X}{n}$ for large enough $n$ by Theorem \ref{thm:convergence rate of SD}. This implies that $\text{Var}[d_n^i] \leq \mathbb{E}[(d_n^i)^2] \leq \frac{K D_X}{n}$.
Hence, it follows from (\ref{eq:variance of compromise SD:01}) that 
\begin{equation} \label{eq:variance of compromise SD:02}
    \begin{aligned}
        \text{Var}\left[ \frac{1}{m} \sum_{i=1}^m (1 + \frac{2 L_f D_X}{\epsilon}) d_n^i + \tau_3 \right] \leq \left(1 + \frac{2 L_f D_X}{\epsilon} \right)^2 \frac{K D_X}{mn}. 
        %\text{Var}\left[ \frac{1}{m} \sum_{i=1}^m (1 + \frac{2 L_f D_X}{\epsilon}) d_n^i + \frac{\epsilon' - \epsilon}{\epsilon} D_X + 2\sqrt{\epsilon/\rho} \right] = \frac{1}{m} \left(1 + \frac{2 L_f D_X}{\epsilon} \right)^2 \frac{K D_X}{n}. 
    \end{aligned}
\end{equation}
On the other hand, 
\begin{equation} \label{eq:variance of compromise SD:03}
    \begin{aligned}
        \mathbb{E}\left[\frac{1}{m} \sum_{i=1}^m (1 + \frac{2 L_f D_X}{\epsilon}) d_n^i + \tau_3 \right] %+ \frac{\epsilon' - \epsilon}{\epsilon} D_X + 2\sqrt{\epsilon/\rho} \right] \\
        \leq (1 + \frac{2 L_f D_X}{\epsilon}) \frac{K}{n} + \tau_3 
        %&\leq (1 + \frac{2 L_f D_X}{\epsilon}) \frac{K}{n} + \frac{\epsilon' - \epsilon}{\epsilon} D_X + 2\sqrt{\epsilon/\rho} \\
        %& = (1 + \frac{2 L_f D_X}{\epsilon}) \frac{K}{n} + \tau_3.
    \end{aligned}
\end{equation}
Finally, using (\ref{eq:common bound on the variance}), (\ref{eq:variance of compromise SD:02}), and (\ref{eq:variance of compromise SD:03}), we derive the upper bound of the variance of $\Delta(\check{X}_{N,\rho,\epsilon}, X^*)$ below:
\begin{equation} \label{eq:variance of compromise SD:04}
    \begin{aligned}
        &\text{Var}\left[\Delta(\check{X}_{N,\rho,\epsilon}, X^*)\right] \leq \left(1 + \frac{2 L_f D_X}{\epsilon} \right)^2 \frac{K D_X}{mn} + \left[\left(1 + \frac{2 L_f D_X}{\epsilon} \right) \frac{K}{n} + \tau_3 \right]^2.
    \end{aligned}
\end{equation}
Finally, the big $O$ argument follows from direct algebraic derivation based on (\ref{eq:variance of compromise SD:03}) and (\ref{eq:variance of compromise SD:04}). 
\end{proof}

In Theorem \ref{thm:expectation and variance of SD for compromise sp}, the reliability of the compromise decision could be understood in the following way. The first part of the upper bound (i.e., $(1 + \frac{2 L_f D_X}{\epsilon})^2 \frac{K D_X}{mn}$) of $\text{Var}\left[\Delta(\check{X}_{N,\rho,\epsilon}, X^*)\right]$ derives from the effect of averages of the solution and function estimate. The second part of the upper bound of the variance is due to the systematic error from the $\epsilon$-optimal solution and the expected distance of estimated solution by SD to the optimal solution set. We also observe that smaller $\epsilon$ is prone to have a more conservative (larger) upper bound on the variance. That is, smaller $\epsilon$ tends to make the compromise decision more adapted to the particular realization of the regularized objective in the Compromise Decision problem. \\
% practical case 

\section{Conclusion}
In this paper, we have presented a theory of reliability for compromise decisions  in stochastic programming problems under one unifying framework.   Our theoretical findings not only validate the computational observations provided in \cite{sen2016mitigating} and \cite{xu2023compromise}, but also extend the previous stopping rule to seeking a non-dominated solution of an estimated mean-variance problem \eqref{eq:Mean-Variance}.  From a computational perspective, we suggest using parallel computing to implement stopping rules which would lead to greater reliability as mentioned in \cite{sen2016mitigating}.  These developments also underscore the differences between deterministic optimization and truly large-scale stochastic optimization.    \\
\indent Based on our theory, we also provide some implementable guidance for choosing the regularization coefficient, $\rho$, and tolerance, $\epsilon$ based on the sample size, $n$. Let $K_1$ be a positive constant. If the ``replication" step involves a cutting-plane-type method, we require that $\rho \geq K_1 n$, $\epsilon_1 = \epsilon$, and $\epsilon_2 = 0$. On the other hand, if the ``replication" step involves SD, we require that $\rho \geq K_1 n^2$ and $\epsilon = \epsilon'$. \\
\indent There remain several open questions about extensions of the compromise decision framework. The first open question is whether it is possible to connect our decision-based regularizer to a class of Distributionally Robust Optimization (DRO) problems. In connection with this question, we note that a study of DRO and robust statistics (Blanchet et al. \cite{blanchet2024distributionally}), reveals that $\phi$-Divergence-Based DRO is related to the variance regularization problem and optimal transport-based DRO is related to norm regularization of the fitted parameters.  In the context of SP, we observe that our decision-based regularizer shares some properties which are similar to norm regularization. A second open question is about how one might reformulate compromise decision problems for nonconvex SP problems and problems with constraint sets that are specified by nonlinear equations/inequalities such as those arising in Neural Networks (NN). The latter have the reputation of being reasonably good on average, and yet they can be unreliable. Perhaps, a study of reliability of solutions for non-convex stochastic optimization models will lead to greater reliability in algorithms for NN. In connection with these directions, one could also investigate the impact of different regularizers (e.g., kernel methods \cite{xu2023ensemble}) for compromise decisions.  %Other directions of research include analysis tools to tighten bounds on the reliability of the compromise decision.   

\section*{Acknowledgements}
% Acknowledgement
This research originated during the grant FA9550-20-1-0006 from AFOSR, but was only completed due to a new grant DE-SC0023361 from the ASCR program of the Department of Energy.  We thank both sources for this invaluable support.  We also thank Roger Wets for highlighting the need for reliability of SP decisions (ICSP conference in Buzios, Brazil, 2016), and his encouragement since.
\begin{appendices}
% Appendix
\section{Table of Notations} \label{appendix:table of notations}

\begin{table}[!htbp]
    \footnotesize
    \centering
    \begin{tabular}{|c|c|c|}
         \hline
         Notations & Description & Usage\\
         \hline
         \multirow{2}{*}{$f_n(x;\xi^n_i)$} & Objective function of SAA problem given sample set $\xi^n_i$ & Computation and  \\
         & & Analysis\\
         \hline
         \multirow{2}{*}{$x_n(\xi^n_i)$} & Optimal solution of SAA problem given sample set $\xi^n_i$ & Computation and  \\
         & & Analysis\\
         \hline
         $X_{n,\epsilon}(\xi^n_i)$ & $\epsilon$-optimal solution set of SAA problem given sample set $\xi^n_i$ & Analysis\\
         \hline
         \multirow{2}{*}{$x_{n,\epsilon}(\xi^n_i)$} & $\epsilon$-optimal solution of SAA problem given sample set $\xi^n_i$ & Computation and \\
         & & Analysis\\ 
         \hline
         $\hat{X}_{N,\rho,\epsilon}(\xi^N)$ & $\epsilon$-optimal solution of the Compromise Decision problem & Analysis\\
         \hline
         $\bar{X}_{N,\epsilon}(\xi^N)$ & $\epsilon$-optimal solution of the aggregated SAA problem & Analysis\\
         \hline
    \end{tabular}
    \caption{Notations in Section \ref{sec:classic compromise SP}}
    \label{tab:notation1}
\end{table}

\begin{table}[!htbp]
    \footnotesize
    \centering
    \begin{tabular}{|c|c|c|}
    \hline
         Notations & Description & Usage\\
         \hline
            \multirow{2}{*}{$\hat{x}_n(\xi^n_i)$} & Estimated decision (by a deterministic algorithm) & Computation and \\
            & of SAA problem given sample set $\xi^n_i$ & Analysis\\
            \hline
            \multirow{2}{*}{$\hat{f}_n(x;\xi^n_i)$} & Piecewise linear approximation & Computation and \\
            & of SAA objective function given sample set $\xi^n_i$ & Analysis\\
            \hline
            \multirow{2}{*}{$\nu_{n,\epsilon_2}(x;\xi^n_i)$} & \multirow{2}{*}{$\epsilon_2$-subgradient of SAA objective function given sample set $\xi^n_i$} & Computation and  \\
            & & Analysis\\
            \hline
            \multirow{2}{*}{$\hat{f}_n(x;\xi^n_i)$} & Augmented piecewise linear approximation & Computation and \\
            & of SAA objective function given sample set $\xi^n_i$ & Analysis\\
            \hline
            $\hat{X}_{N,\epsilon}(\xi^N)$ & $\epsilon$-optimal solution of the problem $\min_{x \in X} \frac{1}{m} \sum_{i=1}^m \hat{f}_n(x;\xi^n_i)$& Analysis \\
            \hline
           $\check{X}_{N,\rho,\epsilon}(\xi^N)$ & $\epsilon$-optimal solution of the Compromise Decision problem & Analysis\\
         \hline
    \end{tabular}
    \caption{Notations in Section \ref{sec:algorithms compromise sp}}
    \label{tab:notation2}
\end{table}

\begin{table}[!htbp]
    \centering
    \begin{tabular}{|c|c|c|}
    \hline
         Notations & Description & Usage\\
         \hline
         \multirow{2}{*}{$\hat{x}^i_n$} & Incumbent solution at the $n^{\text{th}}$ iteration & Computation and  \\
         & of SD in the $i^{\text{th}}$ replication  & Analysis\\
         \hline
         \multirow{2}{*}{$\hat{f}^i_n(x)$} & Approximation function at the $n^{\text{th}}$ iteration & Computation and \\
         & of SD in the $i^{\text{th}}$ replication & Analysis\\
         \hline
         \multirow{2}{*}{$\check{f}^i_n(x)$} & Augmented approximation function at the $n^{\text{th}}$ iteration & Computation and \\
         & of SD in the $i^{\text{th}}$ replication & Analysis\\
         \hline 
         $\check{X}_{N,\rho,\epsilon}$ & $\epsilon$-optimal solution set of the Compromise Decision Problem & Analysis \\ 
         \hline
         $\check{X}_{N,\epsilon}$ & $\epsilon$-optimal solution set of the problem $\min_{x \in X} \frac{1}{m} \sum_{i=1}^m \check{f}^i_n(x)$ & Analysis \\
         \hline
    \end{tabular}
    \caption{Notations in Section \ref{sec:SD compromise sp}}
    \label{tab:notation3}
\end{table}

\section{Some Proofs}\label{appendix secs}

\subsection{Proof of Lemma \ref{lemma:rademacher:discrete}.} \label{appendix:proof of lemma:rademacher:discrete}
\proof
Let $\bar{A}$ be the symmetric hull of $A$, i.e., $\bar{A} = A \cup \{-a: a \in A \}$. Then $|\bar{A}| \leq 2 |A| = 2 N$ and $\max_{a \in A} \|a\|^2 = \max_{a \in \bar{A}} \|a\|^2$. 
Furthermore, 
\begin{equation} \label{eq:rademacher00}
\begin{aligned}
 \max_{a \in A} \frac{1}{n} \left|\sum_{i=1}^n \tilde{\sigma}_i a_i \right| &\leq \max_{a \in \bar{A}} \frac{1}{n} \left|\sum_{i=1}^n \tilde{\sigma}_i a_i \right| &  (A \subset \bar{A})\\
&= \max_{a \in \bar{A}} \frac{1}{n} \sum_{i=1}^n \tilde{\sigma}_i a_i & (\bar{A} \text{ is symmetric})
\end{aligned}
\end{equation}
By taking the expectation over $\tilde{\sigma}_1, \ldots, \tilde{\sigma}_n$ for both sides of (\ref{eq:rademacher00}) and by Definition \ref{def:rademacher average of finite set}, we have 
\begin{equation} \label{eq:rademacher01}
    R_n(A) \leq R_n(\bar{A}) = \mathbb{E}_{\tilde{\sigma}} \left[ \max_{a \in \bar{A}} \frac{1}{n} \sum_{i=1}^n \tilde{\sigma}_i a_i \right].
\end{equation}
We claim that the rest of the proof follows from the proof of \cite[Theorem 3]{boucheron2005theory}. Here is why: for $a \in \bar{A}$ and $s > 0$, we have
\begin{equation} \label{eq:rademacher02}
    \begin{aligned}
        \mathbb{E}_{\tilde{\sigma}} \left[\exp{\left(s \frac{1}{n}\sum_{i=1}^n \tilde{\sigma}_i a_i\right)} \right] &= \prod_{i=1}^n \mathbb{E}_{\tilde{\sigma}}\left[\exp{(s \frac{1}{n} \tilde{\sigma}_i a_i)} \right] & \text{(by independence of $\{\sigma_i\}_{i=1}^n$)}\\
        &\leq \prod_{i=1}^n \exp{\left(\frac{s^2 a_i^2}{2 n^2} \right)} & \text{(by Hoeffding's inequality)} \\
        &= \exp{\left(\frac{s^2 \|a\|^2 }{2 n^2} \right)}.
    \end{aligned}
\end{equation}
Hence, 
\begin{equation} \label{eq:rademacher03}
\begin{aligned}
    \exp{(s R_n(A))} &\leq \exp{\left(s \mathbb{E}_{\tilde{\sigma}}\left[\sup_{a \in \bar{A}} \frac{1}{n} \sum_{i=1}^n \tilde{\sigma}_i a_i \right] \right)} & \text{(by \ref{eq:rademacher01})}\\
    &\leq \mathbb{E}_{\tilde{\sigma}} \left[\exp{\left(s \sup_{a \in \bar{A}} \frac{1}{n} \sum_{i=1}^n \tilde{\sigma}_i a_i \right)} \right] & \text{(by Jensen's inequality)} \\
    &\leq \sum_{a \in \bar{A}} \mathbb{E}_{\tilde{\sigma}}\left[\exp{(s \frac{1}{n} \sum_{i=1}^n \tilde{\sigma}_i a_i)} \right] \\
    &\leq 2 N \max_{a \in \bar{A}} \exp{\left(\frac{s^2 \| a\|^2 }{2n^2} \right)} & \text{(by \ref{eq:rademacher02})}\\
    & = 2 N \max_{a \in A} \exp{\left(\frac{s^2 \| a\|^2 }{2n^2} \right)} 
\end{aligned}
\end{equation}
The last equality of (\ref{eq:rademacher03}) follows because $\bar{A}$ is the symmetric hull of $A$. Take the logarithm of both sides of (\ref{eq:rademacher03}), we have 
\begin{equation} \label{eq:rademacher04}
    R_n(A) \leq \frac{1}{s} \log \left( 2 N  \max_{a \in A} \exp{\left(\frac{s^2 \| a\|^2 }{2n^2} \right)} \right)
\end{equation}
By minimizing the right-hand side of (\ref{eq:rademacher04}) over $s > 0$, the final result follows.
\endproof

\subsection{Proof of Lemma \ref{lemma:rademacher complexity of objective:holder continuity}.} \label{appendix:proof of lemma:rademacher complexity of objective:holder continuity}
    \proof
     According to Lemma \ref{lemma:rademacher:discrete set} and the proof of \cite[Lemma B.2]{ermoliev2013sample}, we derive (\ref{eq:rademacher:holder continuity01}). \\
    For any $\lambda \in (0, \frac{1}{2})$, we have
    \begin{equation}
        \begin{aligned}
             R_n(F,\xi^n)
             &\leq (L_F D^{\gamma} p^{\frac{\gamma}{2}} + M_F \sqrt{2 (\log2 + \frac{p}{2\gamma}\log n)} )/ \sqrt{n} \\
             &\leq \left(L_F  D^{\gamma} p^{\frac{\gamma}{2}} + M_F \sqrt{2 (\log2)} + M_F \sqrt{\frac{p}{\gamma}\log n} \right)/ \sqrt{n} \\
             &\leq \frac{1}{n^\lambda} \left(\frac{L_F D^{\gamma} p^{\frac{\gamma}{2}} + M_F \sqrt{2 (\log2)}}{n^{\frac{1}{2} - \lambda}} + M_F p^{1/2} \gamma^{-1/2} \sqrt{\frac{\log{n}}{n^{1-2\lambda}}}\right) \\
             &\leq \left(L_F D^{\gamma} p^{\frac{\gamma}{2}} + M_F \sqrt{2 (\log2)} + \frac{M_F p^{1/2}}{\sqrt{\gamma (1 - 2 \lambda)e}}  \right) \frac{1}{n^\lambda}. 
        \end{aligned}
    \end{equation}
\endproof

\subsection{Proof of Lemma \ref{lemma:Lipschitz continuity of h}} \label{appendix:proof of lemma:Lipschitz continuity of h}
\proof
    Pick $(x_1,y_1), (x_2,y_2) \in X \times Y$. Note that both $(x_1,y_1)$ and $(x_1,y_1)$ are $(p+1)$-dimensional vectors. %We further let $x_1 = (x_{1,1}, \ldots, x_{1,p})$ and $x_2 = (x_{2,1}, \ldots, x_{2,p})$.
    For $\xi \in \Xi$, we have
    \begin{equation}
        \begin{aligned}
            &|H(x_1,y_1,\xi) - H(x_2,y_2,\xi)| \\
            &= \left|(F(x_1,\xi) - y_1)^2 - (F(x_2,\xi) - y_2)^2 \right| \\
            &= \left|F(x_1,\xi) - y_1 + F(x_2,\xi) - y_2 \right| \left|F(x_1,\xi) - y_1 - F(x_2,\xi) + y_2 \right| \\
            &\leq 4M_F\left(|F(x_1,\xi) - F(x_2,\xi)| + |y_1 - y_2| \right) \\
            &\leq 4M_F\left(L_F \|x_1 - x_2 \|^{\gamma} + |y_1 - y_2 | \right) \\
            &\leq 4M_F\left(L_F \|x_1 - x_2\| + |y_1 - y_2 | \right) \\
            &\leq 4M_F \sqrt{L_F^2 + 1} \left\|(x_1, y_1) - (x_2,y_2) \right\|.
        \end{aligned}
    \end{equation}
    The last inequality holds because
    \begin{equation}
    \begin{aligned}
        &(L_F^2 + 1) \|(x_1, y_1) - (x_2,y_2)\|^2  - \left(L_F \|x_1 - x_2\| + |y_1 - y_2 | \right)^2 \\
        %&= (L_F^2 + 1) \left(\sum_{i=1}^d (x_{1,i} - x_{2,i})^2 + (y_1 - y_2)^2 \right) \\
        %& \quad - \left(L_F^2 \sum_{i=1}^d (x_{1,i} - x_{2,i})^2 + (y_1 - y_2)^2 + 2 L_F \|x_1 - x_2\| |y_1 - y_2 | \right) \\
        &= \|x_1 - x_2\|^2 + L_F^2 (y_1 - y_2)^2 - 2L_F\|x_1- x_2 \| |y_1 - y_2| \\
        &= \left(\|x_1 - x_2 \| - L_F|y_1 - y_2| \right)^2 \geq 0.
    \end{aligned}
    \end{equation}
    Also note that for $x \in X$ and $y \in Y$, we have 
    \begin{equation*}
        H(x,y,\xi) = (F(x,\xi) - y)^2 \leq (2 M_F)^2 = 4 M_F^2. 
    \end{equation*}
    By following the same proof strategy as in the first part of Lemma \ref{lemma:Lipschitz continuity of h}, we can prove the rest of the second claim.  
\endproof

\subsection{Proof of Lemma \ref{lemma:complexity of compound variance}} \label{appendix:proof of lemma:complexity of compound variance}
\proof
The proof is similar to the proof of \cite[Lemma 3.4]{ermoliev2013sample}. Note that 
    \begin{equation*}
        \begin{aligned}
            \delta^h_n(\xi^n) &= \sup_{x \in X} \left|\frac{1}{n} \sum_{i=1}^n H\left(x,\frac{1}{n} \sum_{i=1}^n F(x,\xi_k), \xi_i \right) -  \mathbb{E}_{\tilde{\xi}}\left[H\left(x,\mathbb{E}_{\tilde{\xi}}[F(x,\tilde{\xi})], \tilde{\xi}\right)\right] \right| \\
            &\leq \sup_{x \in X}  \left|\frac{1}{n} \sum_{i=1}^n H\left(x,\frac{1}{n}\sum_{k=1}^n F(x,\xi_k), \xi_i\right) - \frac{1}{n} \sum_{i=1}^n H(x,\mathbb{E}_{\tilde{\xi}}[F(x,\tilde{\xi})], \xi_i) \right| \\
            & \quad + \sup_{x \in X} \left|\frac{1}{n} \sum_{i=1}^n H\left(x,\mathbb{E}_{\tilde{\xi}}[F(x,\tilde{\xi})], \xi_i\right) - \mathbb{E}_{\tilde{\xi}}\left[H\left(x,\mathbb{E}_{\tilde{\xi}}[F(x,\tilde{\xi})], \tilde{\xi}\right)\right] \right| \\
            &\leq  \sup_{x \in X} 4 M_F \left|\frac{1}{n}\sum_{k=1}^n F(x,\xi_k) - \mathbb{E}_{\tilde{\xi}}[F(x,\tilde{\xi})] \right| \\
            &\quad + \sup_{x \in X, y \in Y} \left|\frac{1}{n} \sum_{i=1}^n H(x,y, \xi_i) - \mathbb{E}_{\tilde{\xi}}[H(x,y,\tilde{\xi})] \right| \\
            &= 4 M_F \delta^f_n(\xi^n) + \hat{\delta}_n(\xi^n). 
        \end{aligned}
    \end{equation*}
    The second claim follows based on the symmetric argument \cite{boucheron2005theory,ermoliev2013sample}. 
\endproof

\subsection{Proof of Lemma \ref{lemma:compromise:relation}} \label{appendix:proof of lemma:compromise:relation}
\proof
By the definitions in (\ref{eq:aggregated SAA problem}), (\ref{eq:epsilon-optimal solution set:agg}), and (\ref{notations:inexact compromise decicisions}), we have
\begin{equation} \label{eq:lemma:compromise:relation:01}
    \bar{f}_N\left(\bar{x}_{N,\epsilon}(\xi^N);\xi^N\right) \leq \bar{\theta}_N(\xi^N) + \epsilon \leq \bar{f}_N\left(\hat{x}_{N,\rho,\epsilon}(\xi^N);\xi^N \right) + \epsilon,
\end{equation}
and 
\begin{equation} \label{eq:lemma:compromise:relation:02}
    \begin{aligned}
        &\bar{f}_N\left(\hat{x}_{N,\rho,\epsilon}(\xi^N);\xi^N\right) + \frac{\rho}{2} \left\| \hat{x}_{N,\rho,\epsilon}(\xi^N) - \frac{1}{m} \sum_{j=1}^m x_{n,\epsilon}(\xi^n_j) \right\|^2  - \epsilon \\
        &\leq \theta_{N,\rho,\epsilon}(\xi^N) \\
        &\leq \bar{f}_N\left(\bar{x}_{N,\epsilon}(\xi^N);\xi^N \right) + \frac{\rho}{2} \left\| \bar{x}_{N,\epsilon}(\xi^N) - \frac{1}{m} \sum_{j=1}^m x_{n,\epsilon}(\xi^n_j) \right\|^2.
    \end{aligned}
\end{equation}
The combination of (\ref{eq:lemma:compromise:relation:01}) and (\ref{eq:lemma:compromise:relation:02}) implies that 
\begin{equation*}
    \begin{aligned}
        &\bar{f}_N\left(\hat{x}_{N,\rho,\epsilon}(\xi^N);\xi^N \right) + \frac{\rho}{2} \left\| \hat{x}_{N,\rho,\epsilon}(\xi^N) - \frac{1}{m} \sum_{j=1}^m \hat{x}_{n,\epsilon}(\xi^n_j) \right\|^2  - \epsilon \\
        & \leq \bar{f}_N\left(\hat{x}_{N,\rho,\epsilon}(\xi^N);\xi^N\right) + \epsilon + \frac{\rho}{2} \left\| \bar{x}_{N,\epsilon}(\xi^N) - \frac{1}{m} \sum_{j=1}^m \hat{x}_{n,\epsilon}(\xi^n_j) \right\|^2.
    \end{aligned}
\end{equation*}
This further implies that 
\begin{equation*}
\begin{aligned}
    \left\| \hat{x}_{N,\rho,\epsilon}(\xi^N) - \frac{1}{m} \sum_{j=1}^m x_{n,\epsilon}(\xi^n_j) \right\|^2 &\leq \left\| \bar{x}_{N,\epsilon}(\xi^N) - \frac{1}{m} \sum_{j=1}^m x_{n,\epsilon}(\xi^n_j) \right\|^2 + \frac{4 \epsilon}{\rho} \\
    & \leq \left( \left\| \bar{x}_{N,\epsilon}(\xi^N) - \frac{1}{m} \sum_{j=1}^m x_{n,\epsilon}(\xi^n_j) \right\| + 2 \sqrt{\epsilon / \rho}\right)^2.
\end{aligned}
\end{equation*}
Hence, we have 
\begin{equation} \label{eq:lemma:compromise:relation:03}
    \left\| \hat{x}_{N,\rho,\epsilon}(\xi^N) - \frac{1}{m} \sum_{j=1}^m x_{n,\epsilon}(\xi^n_j) \right\| \leq \left\| \bar{x}_{N,\epsilon}(\xi^N) - \frac{1}{m} \sum_{j=1}^m x_{n,\epsilon}(\xi^n_j) \right\| + 2 \sqrt{\epsilon / \rho}. 
\end{equation}
Finally, by triangular inequality and (\ref{eq:lemma:compromise:relation:03}), we have 
\begin{equation}
    \begin{aligned}
        \left\|\hat{x}_{N,\rho,\epsilon}(\xi^N) - x_{\epsilon}^* \right\| &\leq \left\|\hat{x}_{N,\rho,\epsilon}(\xi^N) - \frac{1}{m} \sum_{j=1}^m x_{n,\epsilon}(\xi^n_j) \right\| + \left\|\frac{1}{m} \sum_{j=1}^m x_{n,\epsilon}(\xi^n_j) - x_{\epsilon}^* \right\| \\
        &\leq \left\| \bar{x}_{N,\epsilon}(\xi^N) - \frac{1}{m} \sum_{j=1}^m x_{n,\epsilon}(\xi^n_j) \right\| + \left\|\frac{1}{m} \sum_{j=1}^m x_{n,\epsilon}(\xi^n_j) - x_{\epsilon}^* \right\| + 2 \sqrt{\epsilon / \rho} %&& \text{by (\ref{eq:lemma:compromise:relation:03})}.
    \end{aligned}
\end{equation}
\endproof
It is worth noting that the proof of Lemma \ref{lemma:compromise:relation} does not specify the form of the approximation function $\bar{f}_N(x;\xi^N)$, which provides certain flexibility in the later analysis. For instance, $\bar{f}_N(x;\xi^N)$ may be replaced by its piecewise linear approximation or even some local inexact approximation function. 

\subsection{Proof of Lemma \ref{lemma:compromise:relation2}} \label{appendix:proof of lemma:compromise:relation2}
\proof
First, we observe that 
\begin{equation*}
    \inf_{x \in X_{\epsilon}^* } \left\|\hat{x}_{N,\rho,\epsilon}(\xi^N) - x \right\| \leq \left\|\hat{x}_{N,\rho,\epsilon}(\xi^N) - x_{\epsilon}^* \right\|.
\end{equation*}
By Lemma \ref{lemma:compromise:relation}, we have 
\begin{equation} \label{eq:lemma:compromise:relation2:01}
\begin{aligned}
    \inf_{x \in X_{\epsilon}^* } \left\|\hat{x}_{N,\rho,\epsilon}(\xi^N) - x \right\| 
    &\leq \left\|\hat{x}_{N,\rho,\epsilon}(\xi^N) - x_{\epsilon}^* \right\| \\
    &\leq \left\| \bar{x}_{N,\epsilon}(\xi^N) - \frac{1}{m} \sum_{j=1}^m x_{n,\epsilon}(\xi^n_j) \right\| \\
    &\quad + \left\|\frac{1}{m} \sum_{j=1}^m x_{n,\epsilon}(\xi^n_j) - x_{\epsilon}^* \right\| + 2 \sqrt{\epsilon / \rho}.
\end{aligned}
\end{equation}
Take the supremum of the left-hand side of (\ref{eq:lemma:compromise:relation2:01}) over $\hat{x}_{N,\rho, \epsilon}(\xi^N) \in \hat{X}_{N,\rho, \epsilon}(\xi^N)$, we have 
\begin{equation} \label{eq:lemma:compromise:relation2:02}
\begin{aligned}
    &\sup_{x' \in \hat{X}_{N,\rho, \epsilon}(\xi^N)} \inf_{x \in X_{\epsilon}^* } \|x' - x\| \\
    &\leq \left\| \bar{x}_{N,\epsilon}(\xi^N) - \frac{1}{m} \sum_{j=1}^m x_{n,\epsilon}(\xi^n_j) \right\| + \left\|\frac{1}{m} \sum_{j=1}^m x_{n,\epsilon}(\xi^n_j) - x_{\epsilon}^* \right\| + 2 \sqrt{\epsilon / \rho}. 
\end{aligned}
\end{equation}
Note that, by definition in (\ref{eq:Delta Set}), $\Delta(\hat{X}_{N,\rho,\epsilon}(\xi^N), X_{\epsilon}^*) =  \sup_{x' \in \hat{X}_{N,\rho, \epsilon}(\xi^N)} \inf_{x \in X_{\epsilon}^* } \|x' - x\|$. Next, by minimizing the right-hand side of (\ref{eq:lemma:compromise:relation2:02}) over $x^*_{\epsilon} \in X_{\epsilon}^*$, we have 
\begin{equation} \label{eq:lemma:compromise:relation2:03}
\begin{aligned}
    \Delta(\hat{X}_{N,\rho,\epsilon}(\xi^N), X_{\epsilon}^*) &\leq  \left\| \bar{x}_{N,\epsilon}(\xi^N) - \frac{1}{m} \sum_{j=1}^m x_{n,\epsilon}(\xi^n_j) \right\| \\
    & \quad + \inf_{x \in X_{\epsilon}^*} \left\|\frac{1}{m} \sum_{j=1}^m x_{n,\epsilon}(\xi^n_j) - x \right\| + 2 \sqrt{\epsilon / \rho}.
\end{aligned}
\end{equation}
Finally, by taking the infimum of the right-hand side of (\ref{eq:lemma:compromise:relation2:03}) over $\bar{x}_{N,\epsilon}(\xi^N) \in \bar{X}_{N,\epsilon}(\xi^N)$, we have 
\begin{equation*}
\begin{aligned}
    \Delta(\hat{X}_{N,\rho,\epsilon}(\xi^N), X_{\epsilon}^*) &\leq  \inf_{x \in \bar{X}_{N,\epsilon}(\xi^N) } \left\| \frac{1}{m} \sum_{j=1}^m x_{n,\epsilon}(\xi^n_j) - x \right\| \\
    & \quad + \inf_{x \in X_{\epsilon}^*}\left\|\frac{1}{m} \sum_{j=1}^m x_{n,\epsilon}(\xi^n_j) - x \right\| + 2 \sqrt{\epsilon / \rho}.
\end{aligned}
\end{equation*}
\endproof

\subsection{Proof of Lemma \ref{lemma:compromise:relation4}} \label{appendix:proof of lemma:compromise:relation4}
\proof
By (\ref{eq:complexity:cost2:02}) from the proof of Theorem \ref{thm:sample complexity cost of compromise decision problem}, we have
\begin{equation*} 
    \theta_{n}(\xi^n_j) - \theta^* \leq \delta^f_{n}(\xi^n_j), \ j = 1,2,\ldots,m,
\end{equation*}
which implies that 
\begin{equation} \label{eq:lemma:compromise:relation4:01}
    \frac{1}{m} \sum_{j=1}^m \theta_{n}(\xi^n_j) \leq \theta^* + \frac{1}{m} \sum_{j=1}^m \delta^f_{n}(\xi^n_j). 
\end{equation}
By convexity in Assumption \ref{assumption:part 2}, we also observe that 
\begin{equation}
\begin{aligned}
    f\left(\frac{1}{m} \sum_{j=1}^m x_{n,\epsilon}(\xi^n_j) \right) &\leq \frac{1}{m} \sum_{j=1}^m f\left(x_{n,\epsilon}(\xi^n_j) \right) \\
    &= \frac{1}{m} \sum_{j=1}^m f_{n}\left(x_{n,\epsilon}(\xi^n_j);\xi^n_j \right) + f\left(x_{n,\epsilon}(\xi^n_j)\right) - f_{n}\left(x_{n,\epsilon}(\xi^n_j);\xi^n_j\right) \\
    &\leq \frac{1}{m} \sum_{j=1}^m \left( \theta_{n}(\xi^n_j) + \epsilon \right)+ \frac{1}{m} \sum_{j=1}^m \delta^f_{n}(\xi^n_j) \\
    &\leq \theta^* + \epsilon + \frac{2}{m} \sum_{j=1}^m \delta_{n}^f(\xi^n_j), \quad  \text{ by (\ref{eq:lemma:compromise:relation4:01})}. 
\end{aligned}
\end{equation}
Let $\epsilon' = \epsilon + \frac{2}{m} \sum_{j=1}^m \delta_{n}^f(\xi^n_j)$, then $\frac{1}{m} \sum_{j=1}^m x_{n,\epsilon}(\xi^n_j) \in X^*_{\epsilon'}$ (i.e., the average of the estimated decisions across all the replications is an $\epsilon'$-optimal solution of the true problem). Hence, by Theorem \ref{theorem:lipschitz continuity of solution set}, we have 
\begin{equation}
    \inf_{x \in X^*_{\epsilon}} \| \frac{1}{m} \sum_{j=1}^m x_{n,\epsilon}(\xi^n_j) - x \| \leq \frac{\epsilon' - \epsilon}{\epsilon} D_X \leq \frac{\frac{2}{m} \sum_{j=1}^m \delta^f_{n}(\xi^n_j)}{\epsilon} D_X. 
\end{equation}
\endproof

\subsection{Proof of Lemma \ref{lemma:bound 1 of algorithm for compromise sp}} \label{appendix:proof of lemma:bound 1 of algorithm for compromise sp}
\proof
By Conditions \ref{condition:outer approximation}  and \ref{condition:termination criterion}, we observe that
\begin{equation*}
    f_n(\hat{x}_n(\xi^n_j);\xi^n_j) \leq \theta_n(\xi^n_j) + \epsilon_1. 
\end{equation*}
By the inequality in (\ref{eq:complexity:cost2:02}) from the proof of Theorem \ref{thm:sample complexity cost of compromise decision problem}, we have 
\begin{equation*}
    \theta_n(\xi^n_j) \leq \theta^* + \delta^f_n(\xi^n_j).
\end{equation*}
Furthermore, by the convexity of $f(x)$ according to Assumption \ref{assumption:part 2}, we have 
\begin{equation*}
\begin{aligned}
    f\left(\frac{1}{m} \sum_{j=1}^m \hat{x}_n(\xi^n_j) \right) &\leq \frac{1}{m} \sum_{j=1}^m f\left(\hat{x}_n(\xi^n_j) \right) \\
    &= \frac{1}{m} \sum_{j=1}^m \left( f_n\left(\hat{x}_n(\xi^n_j);\xi^n_j \right) + f\left(\hat{x}_n(\xi^n_j)\right) - f_n\left(\hat{x}_n(\xi^n_j);\xi^n_j \right) \right)\\
    &\leq \frac{1}{m} \sum_{j=1}^m \left( \theta_n(\xi^n_j) + \epsilon_1 + f(\hat{x}_n(\xi^n_j)) - f_n(\hat{x}_n(\xi^n_j);\xi^n_j) \right)\\
    & \leq \theta^* + \epsilon_1 + \frac{2}{m} \sum_{j=1}^m \delta^f_n(\xi^n_j).
\end{aligned}
\end{equation*}
Let $\epsilon' = \epsilon_1 + \frac{2}{m} \sum_{j=1}^m \delta^f_n(\xi^n_j)$, then $\frac{1}{m} \sum_{j=1}^m \hat{x}_n(\xi^n_j) \in X^*_{\epsilon'}$ and $\epsilon' > \epsilon$. By Theorem \ref{theorem:lipschitz continuity of solution set}, we have 
\begin{equation*}
    \inf_{x \in X^*_{\epsilon}} \left\| \frac{1}{m} \sum_{j=1}^m \hat{x}_{n}(\xi^n_j) - x \right\|  \leq \frac{\epsilon_1 - \epsilon + \frac{2}{m} \sum_{j=1}^m \delta_{n}(\xi^n_j)}{\epsilon} D_X. 
\end{equation*}
\endproof

\subsection{Proof of Lemma \ref{lemma:bound 2 of algorithm for compromise sp}} \label{appendix:proof of lemma:bound 2 of algorithm for compromise sp}
\proof 
We let $\hat{x}_{N}(\xi^N) \in \arg \min\limits_{x \in X} \frac{1}{m} \sum_{i=1}^m \check{f}_{n}(x;\xi^n_i)$. 
By the convexity of $\check{f}_n^i$ infered from Condition \ref{condition:outer approximation}, we have 
\begin{equation} \label{eq:lemma:pla01}
    \begin{aligned}
        \frac{1}{m} \sum_{i=1}^m \check{f}_{n} \left(\frac{1}{m} \sum_{j=1}^m \hat{x}_{n}(\xi^n_j);\xi^n_i \right) &\leq \frac{1}{m^2} \sum_{i=1}^m \sum_{j=1}^m \check{f}_{n}\left(\hat{x}_{n}(\xi^n_j);\xi^n_i \right) \\
        &= \frac{1}{m^2} \sum_{i=1}^m \sum_{j=1}^m \Bigg[\check{f}_{n}(\hat{x}_{N}(\xi^N);\xi^n_i) \\
        & \quad + \check{f}_{n}(\hat{x}_{n}(\xi^n_j);\xi^n_i) - \check{f}_{n}(\hat{x}_{N}(\xi^N);\xi^n_i)\Bigg]
    \end{aligned}
\end{equation}
For $i \neq j$, we have 
\begin{equation*}
    \begin{aligned}
        \check{f}_{n}(\hat{x}_{n}(\xi^n_j);\xi^n_i) = \check{f}_{n}(\hat{x}_{n}(\xi^n_j);\xi^n_j) + \check{f}_{n}(\hat{x}_{n}(\xi^n_j);\xi^n_i) - \check{f}_{n}(\hat{x}_{n}(\xi^n_j);\xi^n_j).
    \end{aligned}
\end{equation*}
By the definition of $\check{f}_n$ in (\ref{eq:augmented piecewise linear funtcion}) and Condition \ref{condition:outer approximation}, we have $0 \leq f_n(\hat{x}_{n}(\xi^n_j);\xi^n_i) - \check{f}_{n}(\hat{x}_{n}(\xi^n_j);\xi^n_i) \leq \epsilon_2$ and $ 0 \leq f_n(\hat{x}_{n}(\xi^n_j);\xi^n_j) - \check{f}_{n}(\hat{x}_{n}(\xi^n_j);\xi^n_j) \leq \epsilon_2$.
Since
\begin{equation*}
    \begin{aligned}
        \check{f}_{n}(\hat{x}_{n}(\xi^n_j);\xi^n_i) - f(\hat{x}_{n}(\xi^n_j)) &= \check{f}_{n}(\hat{x}_{n}(\xi^n_j),\xi^n_i) - f_n(\hat{x}_{n}(\xi^n_j);\xi^n_i) \\
        &\quad + f_n(\hat{x}_{n}(\xi^n_j);\xi^n_i) - f(\hat{x}_{n}(\xi^n_j)) \\
        &\leq \delta_n^f(\xi^n_i),
    \end{aligned}
\end{equation*}
and 
\begin{equation*}
    \begin{aligned}
        %\check{f}_{n}(\hat{x}^j_{n};\xi^n_j) - f(\hat{x}^j_{n}) &\leq \check{f}_{n}(\hat{x}^j_{n};\xi^n_j) - f_n(\hat{x}^j_{n};\xi^n_j) + f_n(\hat{x}^j_{n};\xi^n_j) - f(\hat{x}^j_{n}) \\
        f(\hat{x}_{n}(\xi^n_j)) - \check{f}_{n}(\hat{x}_{n}(\xi^n_j);\xi^n_j) &= f(\hat{x}_{n}(\xi^n_j)) - f_n(\hat{x}_{n}(\xi^n_j);\xi^n_j) \\
        & \quad + f_n(\hat{x}_{n}(\xi^n_j);\xi^n_j) - \check{f}_{n}(\hat{x}_{n}(\xi^n_j);\xi^n_j)\\
        &\leq \epsilon_2 + \delta_n^f(\xi^n_j),
    \end{aligned}
\end{equation*}
we have
\begin{equation*}
\begin{aligned}
     \check{f}_{n}(\hat{x}_{n}(\xi^n_j);\xi^n_i) - \check{f}_{n}(\hat{x}_{n}(\xi^n_j);\xi^n_j) &\leq \check{f}_{n}(\hat{x}_{n}(\xi^n_j);\xi^n_i) - f(\hat{x}_{n}(\xi^n_j)) \\
     & \quad + f(\hat{x}_{n}(\xi^n_j)) -\check{f}_{n}(\hat{x}_{n}(\xi^n_j);\xi^n_j)\\
     &\leq \epsilon_2 + \delta_n(\xi^n_i) + \delta_n(\xi^n_j).
\end{aligned}
\end{equation*}
On the other hand, by Condition \ref{condition:termination criterion} and the property $ \check{f}_n(x;\xi^n_j) \geq \hat{f}_n(x;\xi^n_j)$ due to (\ref{eq:augmented piecewise linear funtcion}), we have

\begin{equation} \label{eq:lemma:pla02}
    \check{f}_n\left(\hat{x}_{N}(\xi^N);\xi^n_j \right) \geq \min_{x \in X} \check{f}_n(x;\xi^n_j) \geq \min_{x \in X} \hat{f}_n(x;\xi^n_j) = \hat{f}_n(\hat{x}_n(\xi^n_j);\xi^n_j),
\end{equation}
which implies that 
\begin{subequations}
\begin{align}
    \check{f}_n(\hat{x}_n(\xi^n_j);\xi^n_j) - \check{f}_n(\hat{x}_{N}(\xi^N);\xi^n_j) &\leq  \check{f}_n(\hat{x}_n(\xi^n_j);\xi^n_j) - \hat{f}_n(\hat{x}_n(\xi^n_j);\xi^n_j) \label{eq:lemma:pla03:01}\\
    & \leq f_n(\hat{x}_n(\xi^n_j);\xi^n_j) - \hat{f}_n(\hat{x}_n(\xi^n_j);\xi^n_j) \label{eq:lemma:pla03:02}\\
    & \leq \epsilon_1. \label{eq:lemma:pla03:03}
\end{align}
\end{subequations}
The inequality in (\ref{eq:lemma:pla03:01}) holds due to (\ref{eq:lemma:pla02}). The inequality in (\ref{eq:lemma:pla03:02}) holds because of Condition \ref{condition:outer approximation}. The inequality in (\ref{eq:lemma:pla03:03})
holds because of Condition \ref{condition:termination criterion}. For $i \neq j$, we have 
\begin{equation*}
     \check{f}_n(\hat{x}_n(\xi^n_j);\xi^n_i) - \check{f}_n(\hat{x}_{N}(\xi^N);\xi^n_j) \leq \epsilon_1 + \epsilon_2 + \delta_n(\xi^n_i) + \delta_n(\xi^n_j).
\end{equation*}
Hence, it follows from (\ref{eq:lemma:pla01}) that 
\begin{equation*}
    \begin{aligned}
    &\frac{1}{m^2} \sum_{i=1}^m \sum_{j=1}^m \left[\check{f}_{n}(\hat{x}_{N}(\xi^N);\xi^n_i) + \check{f}_{n}(\hat{x}_{n}(\xi^n_j);\xi^n_i) - \check{f}_{n}(\hat{x}_{N}(\xi^N);\xi^n_i) \right] \\
    &\leq \frac{1}{m} \sum_{i=1}^m \check{f}_{n}(\hat{x}_{N}(\xi^N);\xi^n_i) + \frac{1}{m^2} \sum_{j=1}^m \left[\check{f}_{n}(\hat{x}_{n}(\xi^n_j);\xi^n_j) - \check{f}_{n}(\hat{x}_{N}(\xi^N);\xi^n_j) \right] \\
    & \quad + \frac{1}{m^2} \sum_{\substack{i,j=1 \\ i \neq j}}^m \left[\check{f}_{n}(\hat{x}_{n}(\xi^n_j);\xi^n_i) - \check{f}_{n}(\hat{x}_{N}(\xi^N);\xi^n_i) \right] \\
    &\leq \hat{\theta}_N + \epsilon_1 + \frac{m-1}{m} \epsilon_2 + \frac{1}{m^2} \sum_{\substack{i,j=1 \\ i \neq j}}^m \left[\delta_{n}^f(\xi^n_i) + \delta_{n}^f(\xi^n_j)\right]
\end{aligned} 
\end{equation*}
The rest of the proof is similar to Lemma \ref{lemma:compromise:relation3}. 
\endproof

\subsection{Proof of Lemma \ref{lemma:bound 3 of SD for compromise sp}} \label{appendix:proof of lemma:bound 3 of SD for compromise sp} 
\begin{proof}
The proof is similar to the Lemma \ref{lemma:bound 2 of SD for compromise sp}. We observe that 
\begin{equation*}
\begin{aligned}
    \left\|\frac{1}{m} \sum_{i=1}^m \hat{x}_n^i - \hat{x}_n^j \right\| &\leq \frac{1}{m} \sum_{i=1}^m \left\|\hat{x}_n^i - \hat{x}_n^j \right\| \\
    &\leq  \frac{1}{m} \sum_{i=1}^m \left(\|\hat{x}_n^i - x^*\| + \| x^* - \hat{x}_n^j \|\right) \\
    & \leq d_n^j + \frac{1}{m} \sum_{i=1}^m d_n^i.
\end{aligned}
\end{equation*}
This implies that 
\begin{equation*}
     \check{f}_n^i\left(\frac{1}{m} \sum_{j=1}^m \hat{x}^j_n \right) \leq \check{f}_n^i\left( \hat{x}^i_n \right) + L_f d_n^i + \frac{L_f}{m} \sum_{j=1}^m d_n^j.
\end{equation*}
By (\ref{eq:sd:lower bound}), we observe that $\check{f}_n^i(\hat{x}_n^i) - \check{f}_n^i(y) \leq \epsilon'$. Let $y$ and $\check{\theta}_N$ be one optimal solution and the optimal value of the problem $\frac{1}{m} \sum_{i=1}^m \check{f}_n^i(x)$, respectively. Then we have
\begin{equation}
    \begin{aligned}
        \begin{aligned}
        \frac{1}{m} \sum_{i=1}^m \check{f}_n^i\left(\frac{1}{m} \sum_{j=1}^m \hat{x}^j_n \right) &\leq \frac{1}{m} \sum_{i=1}^m \check{f}_n^i(\hat{x}_n^i) + \frac{2 L_f}{m} \sum_{i=1}^m d_n^i \\
        &= \frac{1}{m} \sum_{i=1}^m \left[\check{f}_n^i(\hat{x}_n^i) - \check{f}_n^i(y) + \check{f}_n^i(y) \right] + \frac{2 L_f}{m} \sum_{i=1}^m d_n^i \\
        &\leq \check{\theta}_N + \epsilon' + \frac{2 L_f}{m} \sum_{i=1}^m d_n^i.
    \end{aligned}
    \end{aligned}
\end{equation}
Hence, the final result follows by Theorem \ref{theorem:lipschitz continuity of solution set}.
\end{proof}

\section{Symmetric Argument of Rademacher Averages} \label{appendix:symmetric argument of rademacher averages}
Here, we aim to use symmetric argument to show that $\mathbb{E}[\delta_n^f (\tilde{\xi}^n)] \leq 2 R_n(F,\Xi)$.
% proof sketch of the symmetric argument 
\begin{subequations}
\begin{align}
    \mathbb{E}[\delta_n^f (\tilde{\xi}^n)] &= \mathbb{E} \left[\sup_{x \in X} \left| \frac{1}{n} \sum_{i=1}^n F(x,\tilde{\xi}_i) - \mathbb{E}_{\tilde{\xi}}[F(x,\tilde{\xi})] \right| \right] \nonumber \\
    &= \mathbb{E} \left[\sup_{x \in X} \left|\mathbb{E} [ \frac{1}{n} \sum_{i=1}^n F(x,\tilde{\xi}_i') - \frac{1}{n}\sum_{i=1}^n F(x,\tilde{\xi}_i) | \tilde{\xi}_1, \ldots, \tilde{\xi}_n] \right| \right] \label{eq:symmetric01}\\
    &\leq \mathbb{E} \left[\sup_{x \in X} \mathbb{E}[ \left| \frac{1}{n} \sum_{i=1}^n F(x,\tilde{\xi}_i') - \frac{1}{n}\sum_{i=1}^n F(x,\tilde{\xi}_i) \right| |\tilde{\xi}_1, \ldots, \tilde{\xi}_n] \right] \label{eq:symmetric02} \\
    &\leq \mathbb{E}\left[\sup_{x \in X} \left| \frac{1}{n} \sum_{i=1}^n F(x,\tilde{\xi}_i') - \frac{1}{n}\sum_{i=1}^n F(x,\tilde{\xi}_i) \right| \right] \label{eq:symmetric03}
\end{align}
\end{subequations}
The equality in (\ref{eq:symmetric01}) holds because $\xi_1', \ldots, \xi_n', \xi_1, \ldots, \xi_n$ are i.i.d. random variables. The inequality in (\ref{eq:symmetric02}) holds due to Jensen's inequality (\cite[Theorem 1.6.2.]{durrett2019probability}).The inequality in (\ref{eq:symmetric03}) holds is because of the following observations:
$$
\left| \frac{1}{n} \sum_{i=1}^n F(x,\tilde{\xi}_i') - \frac{1}{n}\sum_{i=1}^n F(x,\tilde{\xi}_i) \right| \leq \sup_{x \in X} \left| \frac{1}{n} \sum_{i=1}^n F(x,\tilde{\xi}_i') - \frac{1}{n}\sum_{i=1}^n F(x,\tilde{\xi}_i) \right|
$$
implies that 
$$
\begin{aligned}
    &\mathbb{E}\left[\left| \frac{1}{n} \sum_{i=1}^n F(x,\tilde{\xi}_i') - \frac{1}{n}\sum_{i=1}^n F(x,\tilde{\xi}_i) \right| | \tilde{\xi}_1, \ldots, \tilde{\xi}_n\right] \\
    &\leq \mathbb{E}\left[\sup_{x \in X} \left| \frac{1}{n} \sum_{i=1}^n F(x,\tilde{\xi}_i') - \frac{1}{n}\sum_{i=1}^n F(x,\tilde{\xi}_i) \right| | \tilde{\xi}_1, \ldots, \tilde{\xi}_n\right]
\end{aligned}
$$
which further implies that 
$$
\begin{aligned}
    &\sup_{x \in X} \mathbb{E}\left[\left| \frac{1}{n} \sum_{i=1}^n F(x,\tilde{\xi}_i') - \frac{1}{n}\sum_{i=1}^n F(x,\tilde{\xi}_i) \right| |\tilde{\xi}_1, \ldots, \tilde{\xi}_n\right] \\
    &\leq \mathbb{E}\left[\sup_{x \in X} \left| \frac{1}{n} \sum_{i=1}^n F(x,\tilde{\xi}_i') - \frac{1}{n}\sum_{i=1}^n F(x,\tilde{\xi}_i) \right| | \tilde{\xi}_1, \ldots, \tilde{\xi}_n\right]
\end{aligned}
$$
%based on the observation that $g(y) = \sup_{x \in X} F(x,y)$ is convex in $y$ and Jensen's inequality applies again. \\
\indent The symmetric argument is based on the following observation. Since $\xi_1,\ldots, \xi_n, \xi_1', \ldots, \xi_n'$ are i.i.d. random variables, for any $\mathcal{J} \subseteq \{1,\ldots,n\}$, we have
\begin{equation} \label{eq:symmetric04}
\begin{aligned}
    &\mathbb{E}\left[\sup_{x \in X} \left| \frac{1}{n} \sum_{i=1}^n F(x,\tilde{\xi}_i') - \frac{1}{n}\sum_{i=1}^n F(x,\tilde{\xi}_i) \right| \right] \\
    &= \mathbb{E}\Bigg[\sup_{x \in X} \bigg| \frac{1}{n} \Big(\sum_{j \in \mathcal{J}} F(x,\tilde{\xi}_j) + \sum_{\substack{i=1,\ldots,n \\ i \notin \mathcal{J}}} F(x,\tilde{\xi}_i') \Big)  \\
    & \quad - \frac{1}{n} \Big( \sum_{j \in \mathcal{J}}F(x,\tilde{\xi}_j') + \sum_{\substack{i=1,\ldots,n \\ i \notin \mathcal{J}}} F(x,\tilde{\xi}_i) \Big) \bigg| \Bigg]. 
\end{aligned}
\end{equation}
By definition of $\tilde{\sigma}_j$,  $\tilde{\sigma}_j \left(F(x,\tilde{\xi}_j') - F(x,\tilde{\xi}_j) \right)$ is equally likely to be $F(x,\tilde{\xi}_j') - F(x,\tilde{\xi}_j)$ or $F(x,\tilde{\xi}_j) - F(x,\tilde{\xi}_j')$. Given $\{\sigma^*_i\}_{i=1}^n$ as one realization of $\{\tilde{\sigma}_i\}_{i=1}^n$. Let $\mathcal{J}^* = \{j \in \{1,\ldots,n\}: \sigma^*_j = 1 \}$, we have 
$$
\begin{aligned}
    & \frac{1}{n} \sum_{j=1}^n \sigma^*_j \left(F(x,\tilde{\xi}_j') - F(x,\tilde{\xi}_j) \right) \\
    & = \frac{1}{n} \left(\sum_{j \in \mathcal{J}^*} F(x,\tilde{\xi}_j) + \sum_{\substack{i=1,\ldots,n \\ i \notin \mathcal{J}^*}} F(x,\tilde{\xi}_i') \right)  - \frac{1}{n} \left( \sum_{j \in \mathcal{J}^*}F(x,\tilde{\xi}_j') + \sum_{\substack{i=1,\ldots,n \\ i \notin \mathcal{J}^*}} F(x,\tilde{\xi}_i) \right)  
\end{aligned}
$$

Hence, it follows from (\ref{eq:symmetric04}) that 
\begin{equation} \label{eq:symmetric05}
\begin{aligned}
    &\mathbb{E}\left[\sup_{x \in X} \left| \frac{1}{n} \sum_{i=1}^n F(x,\tilde{\xi}_i') - \frac{1}{n}\sum_{i=1}^n F(x,\tilde{\xi}_i) \right| \right] \\
    &=  \mathbb{E}\left[\mathbb{E}_{\sigma}\left[\sup_{x \in X} \frac{1}{n}\left| \sum_{i=1}^n \sigma_i\left(F(x,\tilde{\xi}_i') - F(x,\tilde{\xi}_i) \right) \right| \right] \right] \\
    &\leq 2 \mathbb{E} \left[\mathbb{E}_{\sigma}\left[\sup_{x \in X} \frac{1}{n}\left| \sum_{i=1}^n \sigma_i F(x,\tilde{\xi}_i) \right| \right] \right] \\
    &\leq 2 R_n(F,\Xi)
\end{aligned}
\end{equation}

\section{Margin of Error: Sample Complexity} \label{appendix:margin of error}
 Recall that we let $\xi^n_i = \{\xi_{(i-1)n + j} \}_{j=1}^n$ denote the $i^\text{th}$ sample set with size $n$. Without replications, the $1- \alpha$ confidence interval of $f(x)$ is given below:
\begin{equation} \label{eq:classic SP:SAA:confidence interval}
    \left[ f_n(x;\xi^n) - \sqrt{\frac{s^2_n(x;\xi^n)}{n}} \Phi^{-1}\left(1- \frac{\alpha}{2} \right), f_n(x;\xi^n) + \sqrt{\frac{s^2_n(x;\xi^n)}{n}} \Phi^{-1}\left(1- \frac{\alpha}{2}\right)\right],
\end{equation}
where $\Phi(\cdot)$ denotes the cumulative distribution function of standard normal distribution and $\Phi^{-1}(\cdot)$ is its inverse. To further reduce the width of the confidence interval (margin of error), one could either increase $n$ or perform replications. \\
\indent The SAA objective function of the $i^\text{th}$ replication is
 \begin{equation*}
         f_{n}(x;\xi^n_i) \triangleq \frac{1}{n} \sum_{j=1}^{n} F(x,\xi_{(i-1)n + j}), \ i = 1,2,\ldots,m,
 \end{equation*}
 and $s^2_{n}(x;\xi^n_i)$ is the sample variance of $F(x,\xi)$ in the $i^\text{th}$ replication:
 \begin{equation}
     s^2_{n}(x;\xi^n_i) = \frac{1}{n-1} \sum_{j=1}^n [F(x,\xi_{(i-1)n + j}) - f_{n,i}(x)]^2.
 \end{equation}
Let $Z_{1 - \frac{\alpha}{2}} = \Phi^{-1} \left(1 - \frac{\alpha}{2} \right)$. With $m$ replications, the margin of error of the $1 - \alpha$ confidence interval of $f(x)$ is $\frac{1}{m}\sqrt{\frac{\sum_{i=1}^m s^2_{n}(x;\xi^n_i)}{n}} Z_{1 - \frac{\alpha}{2}}$. 
%\begin{equation}
    %\left[\frac{1}{m} \sum_{i=1}^m f_{n}(x;\xi^n_i) - \frac{1}{m}\sqrt{\frac{\sum_{i=1}^m s^2_{n}(x;\xi^n_i)}{n}} Z_{1 - \frac{\alpha}{2}}, \sum_{i=1}^m f_{n}(x;\xi^n_i) + \frac{1}{m}\sqrt{\frac{\sum_{i=1}^m s^2_{n}(x;\xi^n_i)}{n}} Z_{1 - \frac{\alpha}{2}}\right].
    %\frac{1}{m}\sqrt{\frac{\sum_{i=1}^m s^2_{n}(x;\xi^n_i)}{n}} Z_{1 - \frac{\alpha}{2}}
%\end{equation}

%Note that the margin of error (half-width) of the $(1 - \alpha)$ confidence above is 
%\begin{equation*}
%    \frac{1}{m}\sqrt{\frac{\sum_{i=1}^m s^2_{n,i}(x)}{n}} \Phi^{-1} \left(1 - \frac{\alpha}{2} \right).
%\end{equation*}
By (\ref{eq:bound:sample variance02}), we have % maybe something off, see the notes
\begin{equation}
    \begin{aligned}
        \frac{1}{m}\sqrt{\frac{\sum_{i=1}^m s^2_{n}(x;\tilde{\xi}^n_i)}{n}} %Z_{1 - \frac{\alpha}{2}} 
        &\leq \sqrt{\frac{m \sigma^2(x) + \frac{4 m M_F^2}{n-1} + 4 M_F \sum_{i=1}^m [\delta^f_n(\tilde{\xi}^n_i) + \hat{\delta}_n(\tilde{\xi}^n_i)] }{m^2n}}. %Z_{1 - \frac{\alpha}{2}}.
    \end{aligned}
\end{equation}
By Jensen's inequality, we have 
\begin{equation}
\begin{aligned}
    &\left(\mathbb{E}\left[\sqrt{\frac{m \sigma^2(x) + \frac{4 m M_F^2}{n-1} + 4 M_F \sum_{i=1}^m [\delta^f_n(\tilde{\xi}^n_i) + \hat{\delta}_n(\tilde{\xi}^n_i)] }{m^2n}}\right] \right)^2 \\
    &\leq \mathbb{E}\left[ \frac{m \sigma^2(x) + \frac{4 m M_F^2}{n-1} + 4 M_F \sum_{i=1}^m [\delta^f_n(\tilde{\xi}^n_i) + \hat{\delta}_n(\tilde{\xi}^n_i)] }{m^2n} \right] \\
    &\leq \frac{\sigma^2(x)}{mn} + \frac{8 M_F N_H + 2 N_F}{mn^{(1+\lambda)}} + \frac{4 M_F^2}{mn(n-1)},
\end{aligned}
\end{equation}
where $\lambda \in (0,\frac{1}{2})$. Hence, the upper bound of the expected margin of error is %(half-width of the $(1 - \alpha)$ confidence interval) is 
\begin{equation*}
\begin{aligned}
    \mathbb{E}\left[\frac{1}{m}\sqrt{\frac{\sum_{i=1}^m s^2_{n}(x;\tilde{\xi}^n_i)}{n}} Z_{1 - \frac{\alpha}{2}}\right] &\leq \sqrt{\frac{\sigma^2(x)}{mn} + \frac{8 M_F N_H + 2 N_F}{mn^{(1+\lambda)}} + \frac{4 M_F^2}{mn(n-1)}} Z_{1 - \frac{\alpha}{2}} \\
    &\leq O((mn)^{-\frac{1}{2}}).
\end{aligned}
\end{equation*}

\section{Properties of Two-stage Stochastic QPs with QP Recourse (SQQP)}\label{appendix sqqp}

This appendix is intended to keep our presentation somewhat self-contained. 
According to \cite[Theorem 4.9 and Proposition 4.2]{liu2020asymptotic}, the dominating term of the convergence rate of SD is $\frac{4 M_0 M_1}{(\Gamma - 1)n}$, where $M_0$ is the upper bound of the Lipschitz modulus of the objective function and $M_1$ is the upper bound of the matrix norm of the Hessians of the quadratic pieces of the objective function inside the ball $B(x^*,\delta)$. Hence, the quantity $K$ in Theorem \ref{thm:convergence rate of SD} could be multiples of $M_0 M_1$. We have derived that $M_0 = L_f = \max\limits_{x \in X} \|Qx\| + \|c\| + L_h$. So we now focus on deriving $M_1$.\\
\indent First, we provide the dual form of the quadratic piece of the second-stage recourse function. Recall that $P \in \mathbb{R}^{n_2 \times n_2}$ is a symmetric and positive definite matrix and $D \in \mathbb{R}^{m_2 \times n_2}$ is a matrix with full row rank. %Let $\|\cdot \|$ denote the $l_2$ norm. 
The quadratic recourse function is formulated as follows: 
\begin{equation} \label{eq:primal qp}
\begin{aligned}
    h(x,\omega) \triangleq \min \ & \frac{1}{2} y^\top P y + d^\top y \\
    \text{s.t.} \  &D y = e(\xi) -  C(\xi) x, \quad [\lambda]\\
    & y \geq 0, \ y \in \mathbb{R}^{n_2}, \quad [\gamma].
\end{aligned}
\end{equation}
The dual multiplier notation associated with \eqref{eq:primal qp} is shown in square brackets $[\cdot]$ above. Let $g(x,\xi) = e(\xi) - C(\xi) x$. The dual of (\ref{eq:primal qp}) can be written as 
\begin{equation} \label{eq:dual qp}
\begin{aligned}
    h(x,\omega) = \max \ & \psi(\gamma,\lambda;x, \xi) = -\frac{1}{2}(-d + D^\top \lambda + \gamma)^\top P^{-1} (-d + D^\top \lambda + \gamma) + g(x,\xi)^\top \lambda \\
    \text{s.t.} \ & \lambda \in \mathbb{R}^{m_2} , \ \lambda \text{ is free}, \ \gamma \in \mathbb{R}^{n_2},\  \gamma \geq 0 .
\end{aligned}
\end{equation}
\begin{proposition} \label{proposition:dual 2nd form}
 The dual problem in (\ref{eq:dual qp}) can be simplified to the following problem: 
 \begin{equation*}
     h(x,\xi) =  \frac{1}{2} g(x,\xi)^\top (M M^\top)^{-1} g(x,\xi) - \frac{1}{2} d^\top H d + \max_{\gamma \geq 0} \{ - \frac{1}{2} \gamma^\top H \gamma + q(x,\xi)^\top \gamma  \}
 \end{equation*}
 where $M = D P^{-\frac{1}{2}}$, $ \Phi = (I - M^\top (M M^\top)^{-1}M)$, $H = P^{-\frac{1}{2}} \Phi^{2} P^{-\frac{1}{2}}$, and 
 $$
 q(x,\xi) = H d - P^{-\frac{1}{2}} {M^\top} (M M^\top)^{-1} g(x,\xi). 
 $$
\end{proposition}
\begin{proof}
Given an $\gamma,x,\xi$, the maximizer of $\psi(\gamma,\lambda;x, \xi)$ over $\lambda$ is 
$\lambda^* = (DP^{-1}D^\top)^{-1} [- D P^{-1} (-d + \gamma) + g(x,\xi)]$. 
Let us denote $M = D P^{-\frac{1}{2}}$. \\
By plugging $\lambda^* = (DP^{-1}D^\top)^{-1} [- D P^{-1} (-d + \gamma) + g(x,\xi)]$ into $\psi(\gamma,\lambda;x, \xi)$, we have 
\begin{equation} \label{eq:dual:01}
\begin{aligned}
     &\psi(\gamma,\lambda^*;x, \xi) \\
     &= \overbrace{- \frac{1}{2} \Big\|(I - M^\top (M M^\top)^{-1}M) P^{-\frac{1}{2}} \gamma - (I  - M^\top (M M^\top )^{-1} M ) P^{-\frac{1}{2}} d}^{\text{Part I}} \\
     & \quad  \underbrace{+ M^\top (M M^\top)^{-1} g(x,\xi)  \Big\|^2}_{\text{Part I}} 
     + \underbrace{g(x,\xi)^\top (M M^\top)^{-1} \left[M P^{-\frac{1}{2}} (-d + \gamma) + g(x,\xi) \right]}_{\text{Part II}}
\end{aligned}
\end{equation}
Note that $\big[I - M^\top (M M^\top)^{-1}M\big]$ is a symmetric matrix and \
\begin{equation} \label{eq:projection matrix}
    \begin{aligned}
    &\big[I - M^\top (M M^\top)^{-1}M\big ] M^\top (M M^\top)^{-1} \\
    &=  M^\top (M M^\top)^{-1} - M^\top (M M^\top)^{-1} M M^\top (M M^\top)^{-1} \\
    & =  M^\top (M M^\top)^{-1} - M^\top (M M^\top)^{-1} \\ 
    & = 0
\end{aligned}
\end{equation}

Let us further denote $ \Phi = \big [ I - M^\top (M M^\top)^{-1}M \big] $ and $H = P^{-\frac{1}{2}} \Phi^2 P^{-\frac{1}{2}}$. 
By expanding Part I of \eqref{eq:dual:01}, we have 
\begin{equation} \label{eq:dual:02}
    \begin{aligned}
        %&-\frac{1}{2} \gamma^\top P^{-\frac{1}{2}} \big [ I - M^\top (M M^\top)^{-1} M \big ]^2 P^{-\frac{1}{2}} \gamma - \frac{1}{2} \Big\|M^\top (M M^\top)^{-1} g(x,\xi)- (I - M^\top (M M^\top)^{-1}M)P^{-\frac{1}{2}} d \Big \|^2 \\
        &-\frac{1}{2} \gamma^\top P^{-\frac{1}{2}} \Phi^2 P^{-\frac{1}{2}} \gamma - \frac{1}{2} \Big\|M^\top (M M^\top)^{-1} g(x,\xi)- \Phi P^{-\frac{1}{2}} d \Big \|^2 \\
        %& \quad - \gamma^\top P^{-\frac{1}{2}} \big [I - M^\top (M M^\top)^{-1} M \big] \Big [M^\top (M M^\top)^{-1} g(x,\xi) - \big[ I - M^\top(M M^{\top})^{-1}M \big] P^{-\frac{1}{2}}d \Big] \\
        & \quad - \gamma^\top P^{-\frac{1}{2}} \Phi \Big [M^\top (M M^\top)^{-1} g(x,\xi) - \Phi P^{-\frac{1}{2}}d \Big] \\
        %& = -\frac{1}{2} \gamma^\top H \gamma  + \gamma^\top H d \underbrace{- \frac{1}{2} \Big \|M^\top (M M^\top)^{-1} g(x,\xi)- \big [I - M^\top (M M^\top)^{-1}M \big ]P^{-\frac{1}{2}} d \Big \|^2}_{\text{Part III}}.
        & = -\frac{1}{2} \gamma^\top H \gamma  + \gamma^\top H d \underbrace{- \frac{1}{2} \Big \|M^\top (M M^\top)^{-1} g(x,\xi)- \Phi P^{-\frac{1}{2}} d \Big \|^2}_{\text{Part III}}.
    \end{aligned}
\end{equation}
By expanding Part III, we obtain
\begin{equation} \label{eq:dual:expansion of part III}
    \begin{aligned}
       \text{(Part III)} & =  -\frac{1}{2} \|M^\top (M M^\top)^{-1} g(x,\xi) \|^2 \\
       & \quad + d^\top P^{-\frac{1}{2}} \underbrace{(I - M^\top (M M^\top)^{-1} M) M^\top (M M^\top)^{-1}}_{\text{equals to 0 due to (\ref{eq:projection matrix})}} g(x,\xi) \\
        & \quad - \frac{1}{2} \|(I - M^\top (M M^\top)^{-1} M)P^{-\frac{1}{2}} d \|^2\\
        & = -\frac{1}{2} g(x,\xi)^\top (M M^\top)^{-1} M M^\top (M M^\top)^{-1} g(x,\xi)- \frac{1}{2} d^\top H d \\
        & = -\frac{1}{2} g(x,\xi)^\top (M M^\top)^{-1} g(x,\xi) - \frac{1}{2} d^\top H d
    \end{aligned}
\end{equation}
Hence, it follows from (\ref{eq:dual:02}) that 
\begin{equation}
\text{(Part I)} =  -\frac{1}{2} \gamma^\top H \gamma  + \gamma^\top H d - \frac{1}{2} d^\top H d -\frac{1}{2} g(x,\xi)^\top (M M^\top)^{-1} g(x,\xi)
\end{equation}
Thus, it follows from (\ref{eq:dual:01}) that 
\begin{equation} \label{eq:dual:03}
    \begin{aligned}
        \psi(\gamma,\lambda^*;x, \xi) 
        %&= -\frac{1}{2} \gamma^\top H \gamma  + \gamma^\top H d - \frac{1}{2} d^\top H d -\frac{1}{2} g(x,\xi)^\top (M M^\top)^{-1} g(x,\xi) \\ 
        %& - g(x,\xi)^\top (M M^\top)^{-1} M P^{-\frac{1}{2}} d + g(x,\xi)^\top (M M^\top)^{-1} M P^{-\frac{1}{2}} \gamma + g(x,\xi)^\top (M M^\top)^{-1} g(x,\xi) \\
        &= \frac{1}{2} g(x,\xi)^\top (M M^\top)^{-1} g(x,\xi) - \frac{1}{2} d^\top H d  \\
        & + [H d - P^{-\frac{1}{2}} M^\top (M M^\top)^{-1} g(x,\xi)]^\top \gamma \\
        & - \frac{1}{2} \gamma^\top H \gamma  + \gamma^\top H d 
    \end{aligned}
\end{equation}
We let $q(x,\xi) = H d - P^{-\frac{1}{2}} M^\top (M M^\top)^{-1} g(x,\xi)$, then it follows from (\ref{eq:dual:03}) that 
\begin{equation}
    \psi(\gamma,\lambda^*;x, \xi) = \frac{1}{2} g(x,\xi)^\top (M M^\top)^{-1} g(x,\xi) - \frac{1}{2} d^\top H d + q(x,\xi)^\top \gamma - \frac{1}{2} \gamma^\top H \gamma. 
\end{equation}
Hence, problem (\ref{eq:dual qp}) can be simplified to the following problem:
\begin{equation} \label{eq:dual:second form}
    h(x,\xi) =  \frac{1}{2} g(x,\xi)^\top (M M^\top)^{-1} g(x,\xi) - \frac{1}{2} d^\top H d + \max_{\gamma \geq 0} \{ - \frac{1}{2} \gamma^\top H \gamma + q(x,\xi)^\top \gamma  \}.
\end{equation}
\end{proof}
\begin{example}
\begin{equation} \label{eq:ex:primal qp}
\begin{aligned}
    h(x,\xi) \triangleq \min \ & \frac{1}{2} y^\top y \\
    \text{s.t.} \  &y = \xi(\xi) -  C(\xi) x\\
    & y \geq 0, \ y \in \mathbb{R}^{n_2} 
\end{aligned}
\end{equation}
In this example, $P = I_{n_2 \times n_2}$, $d = 0$, $D = I_{n_2 \times n_2}$. Then $M = I_{n_2 \times n_2}$, $H = 0$, $q(x,\omega) = -g(x,\omega)$. Then it follows from (\ref{eq:dual:second form}) that 
\begin{equation} \label{eq:ex:primal qp2}
    h(x,\xi) =  \frac{1}{2} \|g(x,\xi)\|^2 + \max_{\gamma \geq 0} - g(x,\xi)^\top \gamma.
\end{equation}
For $x$ which satisfies $g(x,\xi) = e(\xi) -  C(\xi) x \geq 0$, the optimal value of (\ref{eq:ex:primal qp}) satisfies $h(x,\xi) = \frac{1}{2}\|g(x,\xi)\|^2$. On the other hand $\frac{1}{2} \|g(x,\xi)\|^2 = \frac{1}{2} \|g(x,\xi)\|^2 + \max_{\gamma \geq 0} - g(x,\xi)^\top \gamma$ for $g(x,\xi) \geq 0$. So we have verified that dual optimal values in (\ref{eq:ex:primal qp2}) are equal to the primal optimal values \eqref{eq:ex:primal qp}. 
\end{example}

\begin{proposition} \label{proposition:M1}
Let $M, P$ be the same form in Proposition \ref{proposition:dual 2nd form}, Let $\Lambda(\xi) = P^{-\frac{1}{2}} {M^\top} (M M^\top)^{-1} C(\xi)$ and $T(\xi) = C(\xi)^\top (MM^\top)^{-1} C(\xi) + \Lambda(\xi)^\top H^{-1} \Lambda(\xi)$. Then the constant $M_1$ in \cite[Theorem 4.9]{liu2020asymptotic} can be written as follows:
    $$
    M_1 = \| P \|+ \sup_{\xi \in \Xi} \|T(\xi) \|.
    $$
Furthermore, let $M_0 = \max\limits_{x \in X} \|Qx\| + \|c\| + L_h$, then $K$ in Theorem \ref{thm:convergence rate of SD} satisfies $K = O(M_0 M_1)$. 
\end{proposition}
\begin{proof}
 Given $x \in X \cap B(x^*,\delta)$. We consider the following maximization problem:
\begin{equation}
    \max_{\gamma \geq 0} -\frac{1}{2} \gamma^\top H \gamma + q(x,\omega)^\top \gamma
\end{equation}
The KKT condition of the problem above is 
\begin{equation}
    -H \gamma^* + q(x,\omega) - \theta^* = 0,
\end{equation}
where $\theta^*$ is the Lagrange multiplier of the nonnegativity constraint $\gamma \geq 0$.
It implies that 
\begin{equation*}
    s\gamma^* = H^{-1} (q(x,\omega) - \theta^*).
\end{equation*}
Hence, the value function at $x$ can be written as 
\begin{equation}
    h(x,\xi) =  \frac{1}{2} g(x,\xi)^\top (M M^\top)^{-1} g(x,\xi) - \frac{1}{2} d^\top H d + \frac{1}{2} (q(x,\xi) + \theta^*)^\top H^{-1} (q(x,\xi) - \theta^*)
\end{equation}
Let $\Lambda(\xi) = P^{-\frac{1}{2}} {M^\top} (M M^\top)^{-1} C(\xi)$, the Hessian of $h(x,\xi)$ at $x$ is 
\begin{equation}
\begin{aligned}
    \nabla^2 h(x,\xi) &= C(\xi)^\top (MM^\top)^{-1} C(\xi) + \Lambda(\xi)^\top H^{-1} \Lambda(\xi),
    %& \quad + [P^{-\frac{1}{2}} {M^\top} (M M^\top)^{-1} C(\xi)]^\top H^{-1} P^{-\frac{1}{2}} {M^\top} (M M^\top)^{-1} C(\xi).
\end{aligned}
\end{equation}
which does not depend on $x$. Hence, let $T(\xi) = \nabla^2 h(x,\xi)$,
%\begin{equation}
%    T(\xi) = %C(\xi)^\top (MM^\top)^{-1} C(\xi) + [P^{-\frac{1}{2}} {M^\top} (M M^\top)^{-1} C(\xi)]^\top H^{-1} P^{-\frac{1}{2}} {M^\top} (M M^\top)^{-1} C(\xi)
%\end{equation}
we can get $M_1 = \| Q \|+ \sup_{\xi \in \Xi} \|T(\xi) \|$. The rest of the proof follows from \cite[Theorem 4.9 and Proposition 4.2]{liu2020asymptotic}.
\end{proof}
\end{appendices}
\newpage
% citation 
\bibliographystyle{abbrv}
\bibliography{main_cd}

\end{document}